\documentclass{lms}

\usepackage{amssymb}
\usepackage{amsmath}
\usepackage[dvips]{color}
\usepackage[margin=1in]{geometry}
\usepackage{bbm}
\usepackage{graphics}
\usepackage{graphicx}
\usepackage{pifont}
\usepackage{todonotes}

\newtheorem{thm}{Theorem}[section]
\newtheorem{lem}[thm]{Lemma}
\newtheorem{conj}[thm]{Conjecture}

\newtheorem{rem}[thm]{Remark}

\newtheorem{defn}[thm]{Definition}

\numberwithin{equation}{section}

\numberwithin{table}{section}

\newcommand{\ii}{i}
\newcommand{\vr}{\varrho}
\newcommand{\Mma}{\textit{Mathematica }}
\newcommand{\Reduce}{\texttt{Reduce}}
\newcommand{\NSolve}{\texttt{NSolve}}
\newcommand{\Order}{\mathcal O}
\newcommand{\one}{\mathbbm{1}}

\newcommand{\Real}{\mathbb{R}}

\newcommand{\Pisp}{\Pi_{s/s,p}}
\newcommand{\Pihatsp}{\widehat{\Pi}_{s/s,p}}
\newcommand{\Pihat}{\widehat{\Pi}}
\newcommand{\Rsp}{R_{s/s,p}}
\newcommand{\Rhatsp}{\widehat{R}_{s/s,p}}

\title{Rational functions with maximal radius of absolute monotonicity}

\author{Lajos L\'oczi \and David I. Ketcheson}

\classno{65L06 (primary), 68W30, 41A20, 49K30 (secondary)}

\extraline{This publication is based on work supported by Award No.  FIC/2010/05 -- 2000000231, made by King Abdullah University of Science and Technology (KAUST).}

\begin{document}

\maketitle

{\abstract{We study the radius of absolute monotonicity $R$ of rational functions with
numerator and denominator of degree $s$ that approximate
the exponential function to order $p$.  Such functions arise in the application of implicit
$s$-stage, order $p$ Runge--Kutta methods for initial value problems, and the radius of absolute 
monotonicity governs the numerical preservation of properties like positivity
and maximum-norm contractivity.  
We construct a function with $p=2$ and $R>2s$, disproving a conjecture of van de Griend and Kraaijevanger.
We determine the maximum attainable radius for
functions in several one-parameter families of rational functions. Moreover, we
prove earlier conjectured optimal radii in some families with 2 or 3 parameters
via uniqueness arguments for systems of polynomial inequalities.  Our results
also prove the optimality of some strong stability preserving implicit and singly diagonally
implicit Runge--Kutta methods.  Whereas previous results in this area were
primarily numerical, we give all constants as exact algebraic numbers.
}}


\section{Introduction and aims\label{sec:intro}}
A smooth function $\psi: \mathbb{R} \to \mathbb{R}$ is said to be {\em absolutely monotonic} at a point if $\psi$ and all
of its derivatives are non-negative there.  The {\em radius of absolute monotonicity} $R(\psi)\in [0,+\infty]$ is defined as
\[
R(\psi) = \sup\left(\{ r\in [0,+\infty)  : \psi \text{ is absolutely monotonic at each point of } [-r,0]\}\cup \{0\}\right).
\]
The radius of absolute monotonicity
of polynomials and rational functions plays an important role in the analysis of positivity,
monotonicity, and contractivity
properties of numerical methods for initial value problems and is often referred to
as the {\em threshold factor} in this context \cite{bc,spijker1983,kp,vdgk,ketcheson2009a}. 
Specifically, the maximal positive or contractive step-size is given by $R(\psi) h_0$,
where $\psi$ is the stability function of the numerical method and $h_0$ is the maximum
step-size under which the corresponding property holds for the explicit Euler method.

It is therefore natural to consider the problem of finding a function $\psi$ that 
achieves the \textit{maximal radius of absolute monotonicity} within a given class.  In this work, we study absolute monotonicity of rational functions that correspond to the stability
functions of  certain \textit{implicit} or \textit{singly diagonally implicit Runge--Kutta methods}.

A Runge--Kutta (RK) method of $s\in\mathbb{N}^+$ stages is defined by its coefficients, an $s\times s$ matrix $A$ and
an $s\times 1$ vector $b$ \cite{Butcher_2008}.
The \textit{stability function} of the method is
\begin{equation}\label{stabilityfunctionformula}
\psi(z)=\psi^{A,b}(z) := \frac{\det(I-zA + z \one b^\top)}{\det(I-zA)}\quad\quad (z\in\mathbb{C}),
\end{equation}
where $\one$ is a column vector of length $s$ with all unit entries and $I$ is the $s\times s$ identity matrix
\cite{hnwbook}.
The \textit{order} of a method, denoted by $p\in\mathbb{N}^+$, indicates how accurately the computed solution approximates
the exact solution in an asymptotic sense.  The stability function of an RK method of order $p$ must approximate
the exponential function to at least order $p$ near the origin.

In an {\em implicit Runge--Kutta method} (IRK), all entries of $A$ may be non-zero.
An important subclass of the IRK methods are the \textit{singly diagonally implicit} RK (SDIRK) methods,
with $A$ lower triangular with identical diagonal
entries. For {\em explicit RK methods}, $A$ is strictly lower triangular.

We now define the classes of rational functions to be studied.
For $m,n\in\mathbb{N}$,  let $\Pi_m$ denote the set of real polynomials of degree at most $m$:
\[\Pi_m=\bigg\{\sum_{j=0}^m \beta_j z^j : \beta_j\in\mathbb{R}, \ j=0,1,\ldots, m\bigg\},\]
and let
$\Pi_{m/n,p}$ denote the set of all real $(m,n)$-rational functions 
approximating the exponential to order $p$ near the origin:
\[\Pi_{m/n,p}=\Big\{ \psi : \psi=\frac{P}{Q}, P\in\Pi_m, 0 \not\equiv Q\in\Pi_n, \psi(z) - \exp(z) = \Order(z^{p+1}) \text{ as } z\to 0 \Big\}.\]
Let $\widehat{\Pi}_{m/n,p}$ denote the elements of $\Pi_{m/n,p}$ whose
denominator has (at most) a single, non-zero real root:
\[
\widehat{\Pi}_{m/n,p}=\Big\{ \psi\in \Pi_{m/n,p}: \psi(z)=\frac{P(z)}{(1-a z)^n},\  P\in\Pi_m, a\in\mathbb{R}  \Big\}.
\]
Obviously, for every $0\le m\le \widetilde{m}$, $0\le n\le \widetilde{n}$ and $1\le p\le \widetilde{p}$, we have $\widehat{\Pi}_{m/n,p}\subset \Pi_{m/n,p}$, $\Pi_{m/n,p}\subset \Pi_{\widetilde{m}/\widetilde{n},p}$ and $\Pi_{m/n,p}\supset \Pi_{m/n,\widetilde{p}}$.

If $A$ and $b$ correspond to an IRK method of $s$ stages and order $p$, then 
\begin{equation}\label{psiAbinPissp} \psi^{A,b} \in \Pi_{s/s,p},\end{equation}
while for SDIRK methods, 
\begin{equation}\label{psiAbinPihatssp} \psi^{A,b} \in \widehat{\Pi}_{s/s,p}.\end{equation} 
For explicit RK methods we have $\psi^{A,b} \in \Pi_s$. 

A thorough study of polynomial approximations to the exponential with
maximal radius of absolute monotonicity can be found in \cite{kp,ketcheson2009a}.
We are interested in determining the maximal radius of absolute monotonicity that
can be achieved among the stability functions of IRK or SDIRK methods of a
given order $p$.  Therefore, for non-empty sets $\Pi_{m/n,p}$ and $\widehat{\Pi}_{m/n,p}$, let us define the quantities
\begin{align*}
R_{m/n,p} & := \sup\{ R(\psi) : \psi \in \Pi_{m/n,p} \} \\
\widehat{R}_{m/n,p} & := \sup\{ R(\psi) : \psi \in \widehat{\Pi}_{m/n,p} \}.
\end{align*}
We will focus on the cases for which $m=n=s$.

The seminal work on this topic
is \cite{vdgk}, in which an algorithm is presented for computing the radius of absolute
monotonicity of a rational function, and many properties of the radius of absolute
monotonicity are proved. The determination of $R(\psi)$ is not trivial even for a single rational function $\psi$, 
so the difficulty of obtaining $R_{m/n,p} $ or  $\widehat{R}_{m/n,p} $ for a particular $(m,n,p)$ triple ranges from fairly challenging to currently impossible.
Nevertheless, some patterns in numerically computed values have led to important conjectures.

In order to be able to fully describe these conjectures, we recall the concept
of \textit{radius of absolute monotonicity of a Runge--Kutta method} \cite{kc}, denoted 
by $R(A,b)$. The quantity $R(A,b)$ is also referred to as
\textit{Kraaijevanger's coefficient} \cite{fs}, or \textit{SSP coefficient}
\cite{SSPbook}. For a RK method with coefficients $A,b$, define $K\in \Real^{(s+1)\times (s+1)}$ by
\[
K=K(A,b):=\left(
\begin{array}{cc}
 A & 0 \\
 b^\top & 0 \\
\end{array}
\right).
\]
Let $\one$ denote now the vector $(1,1, \ldots, 1)^\top\in\Real^{s+1}$. 
Then the radius of absolute monotonicity of the RK method is
\[
R(A,b):=\sup\{  r\in\mathbb{R} : \forall \varrho\in [0,r]\  \exists (I+\varrho K)^{-1}, \varrho K(I+\varrho K)^{-1}\ge 0\ \mathrm{and}\  \varrho K(I+\varrho K)^{-1}\one\le \one\},
\]
where vector and matrix inequalities are understood componentwise.
Notice that $R(A,b)\ge 0$.
Absolute monotonicity of a Runge--Kutta method implies absolute monotonicity
of its stability function \cite{kc}; thus we have 
\begin{align} \label{Rineq}
R(A,b) \le  R(\psi^{A,b}).
\end{align}
The coefficient $R(A,b)$ plays the same role in numerical preservation of positivity and contractivity 
for non-linear problems that the coefficient $R(\psi^{A,b})$ plays for linear problems \cite{SSPbook}.

The following conjectures served as motivation for our work. 
\begin{conj}[({\cite[p. 421]{vdgk}})]\label{conj:vdgk}
For $m,n\in {\mathbb N}^+$, $R_{m/n,2} = m + \sqrt{mn}.$
\end{conj}
In \cite{vdgk}, this conjecture was proved for all $m\ge 1$ with $n=1$ or $2$.
In the special case $m=n=s\in\mathbb{N}^+$, the conjecture claims 
\begin{align} \label{2sbound}
R_{s/s,2} = 2s.
\end{align}
For $A,b$ corresponding to an RK method with $s$ stages and order $p$, we have $R(\psi^{A,b})\le R_{s/s,p}$.
Moreover, $R_{s/s,p}$ is a non-increasing function of $p$ for fixed $s$; therefore \eqref{Rineq} and (\ref{2sbound}) 
together would 
imply that
$R(A,b)\le 2s$ for all Runge--Kutta methods that are more than first order
accurate.  Indeed, evidence in the literature reinforces belief in the bound
\eqref{2sbound}; we have the following conjectures based on numerical searches.
\begin{conj}[({\cite[Conjecture 3.1]{fs}})]\label{conj:fs}
Let $A,b$ denote the coefficients of an SDIRK method of
order $p\ge 2$ with $s\ge 1$ stages.  Then $R(A,b)\le 2s$.
\end{conj}

\begin{conj}[({\cite{kmg}})]\label{conj:kmg}
Let $A,b$ denote the coefficients of an implicit RK method of
order $p\ge 2$ with $s\ge 1$ stages.  Then $R(A,b)\le 2s$.
\end{conj}
It is thus very surprising that, as we will see, Conjecture
\ref{conj:vdgk}---and thus equality (\ref{2sbound})---\textit{does not hold}
for $m=n=3$. An
immediate consequence is that Conjecture \ref{conj:kmg}
cannot be proved by analyzing $R_{s/s,2}$. A weaker conjecture that
arises naturally is
\begin{conj}\label{conj:hat}
For each $s\ge 1$, $\widehat{R}_{s/s,2} = 2s$.
\end{conj}
Conjectures \ref{conj:fs}, \ref{conj:kmg}, and \ref{conj:hat} were previously proved only in
the cases $s=1$ or $s=2$.

We have presented these conjectures in the order they were formulated. By also
taking into account that

$\bullet$ $\widehat{R}_{s/s,p}\le \widehat{R}_{s/s,2}\le R_{s/s,2}$ ($p\ge 2$),

$\bullet$ $R(\psi^{A,b})\le \widehat{R}_{s/s,p}$ for $A,b$ corresponding to an SDIRK method,

$\bullet$ there exists an SDIRK method (consisting of the repetitions of the
            implicit midpoint method) with $R(A,b)=2s$ (\cite[formula (3.1)]{fs}) whose
            stability function
 \[
\psi^{A,b}(z)=\frac{\left(1+\frac{z}{2s}\right)^s}{\left(1-\frac{z}{2s}\right)^s}, \quad \psi^{A,b}\in \widehat{\Pi}_{s/s,2}\subset \Pi_{s/s,2}\]
satisfies $R(\psi^{A,b})=2s$,
we have the following implications for each value of $s$:

\begin{center}
Conjecture \ref{conj:vdgk} with $m=n=s$ $\implies$ Conjecture \ref{conj:hat} $\implies$ Conjecture \ref{conj:fs}
\end{center}
and
\begin{center}
Conjecture \ref{conj:vdgk} with $m=n=s$ $\implies$ Conjecture \ref{conj:kmg} $\implies$ Conjecture \ref{conj:fs}.
\end{center}

\indent As for the $p=3$ case, we pose the following new conjecture, which is stronger than \cite[Conjecture 3.2]{fs}.

\begin{conj}\label{conj:hatp3}
For each $s\ge 2$, $\widehat{R}_{s/s,3} = s-1+\sqrt{s^2-1}$.
\end{conj}

 The present work is devoted to determining the exact values of $R_{s/s,p}$ and
 $\widehat{R}_{s/s,p}$ for certain $2\le s\le 4$ and $2\le p \le 7$ pairs.
 Our results lend some support to Conjectures \ref{conj:fs} and \ref{conj:hat}, 
 since we prove each of them for the cases $s=3$ and $s=4$.
We also prove Conjecture \ref{conj:hatp3} for $s\le 4$ (simultaneously proving the $3\le s \le 4$ special cases of \cite[Conjecture 3.2]{fs}).\\ 

The structure of the paper is as follows.

In Section \ref{earlierresultssection} we recall some 
theorems from  \cite{vdgk,kc} on which our computations are
based.  

All of our new results are overviewed and interpreted in Section
\ref{sectionmainresults}; the proofs are given only in Sections
\ref{sectionIRKs2}--\ref{appendixsection}. Section
\ref{sectionmainresults} has been written so that the casual reader need not
refer to any later sections.  In Section
\ref{furtherresultssection}, we
give some auxiliary, but, in our opinion, related and interesting results that we could not (yet) tie to the main pieces of the
puzzle (\textit{e.g.,} to Conjecture \ref{conj:hat}), along with a few remarks
about the successful (or failed) proof attempts and techniques.  In Section \ref{notationsection}, we introduce some compact
notation for the algebraic numbers that will be prevalent in this work. 

The interested reader will find the actual proofs of our theorems in Sections \ref{sectionIRKs2}--\ref{appendixsection}. Generally, we proceed from smaller $s$ values to larger ones, and if $s$ is fixed, from larger $p$ values to smaller ones, \textit{i.e.,} sections are ordered roughly in increasing difficulty within both the $\Pi_{s/s,p}$ and in the $\widehat{\Pi}_{s/s,p}$ classes. These sections represent the fruits of several dozens of pages of computations---or of a few hundred pages, depending on the level of detail---so we had inevitably to omit some parts of some proofs. In Section \ref{appendixsection}, a few more algebraic expressions, mentioned only implicitly in the proofs, are collected to enable the reproducibility of certain longer computations.

Let us close this introduction with a remark explaining why the $p=1$ case is exceptional.

\begin{rem}[(on the $p=1$ case)] For $\psi\in\Pi_{m/n,p}$ with $p\ge 2$, we have $R(\psi)<+\infty$ \cite{spijker1983}. 
On the other hand, for $\psi(z) = 1/(1-z)$,  we have $\psi \in \Pi_{0/1,1}$ and $R(\psi)=\infty$.
One may ask whether there exists $\psi \in \Pi_{s/s,1}$ such that $R(\psi)=\infty$ and $\psi \notin \Pi_{(s-1)/(s-1),1}$
(and correspondingly for the sets $\widehat{\Pi}_{s/s,1}$).  The answer is
affirmative.  For any $s\in \mathbb{N}^+$ define
\[
\widehat{\psi}_s(z)=1-\frac{1}{2 s}+\frac{1}{2 s (1-2 z)^{s}},
\]
and for any $2\le s\in \mathbb{N}$, consider 
\[
\psi_s(z)=\frac{2^s-2}{2^s+1}+\frac{3}{4^{s}-1}\sum_{m=1}^s \frac{1}{2^{1-m}-z}.
\] 
It is easily seen that both
$\widehat{\psi}_s$ and $\psi_s$ can be represented as  rational functions with
numerator and denominator having degree {\rm{exactly}} $s$.  Moreover,
$\widehat{\psi}_s(0)=\widehat{\psi}^\prime_s(0)=1=\psi_s(0)=\psi^\prime_s(0)$,
so $\widehat{\psi}_s\in\widehat{\Pi}_{s/s,1}$ and
$\psi_s\in\Pi_{s/s,1}\setminus \widehat{\Pi}_{s/s,1}$ (of course,
$\widehat{\psi}_1\in \widehat{\Pi}_{1/1,1} = \Pi_{1/1,1}$). Direct
differentiation shows that for all $k\in\mathbb{N}$ and $x\le 0$ we have
$\widehat{\psi}_s^{(k)}(x)\ge 0$ and $\psi_s^{(k)}(x)\ge 0$, hence
$R(\widehat{\psi}_s)=R(\psi_s)=+\infty$. Consequently, we are going to consider
only the $p\ge 2$ case in this work.  
\end{rem}


\subsection{Some general results on the radius of absolute monotonicity}\label{earlierresultssection}

Let us briefly summarize some useful theorems from \cite{vdgk,kc} that will frequently be used in this work. In \cite{vdgk}, the following assumptions are made on the rational function $\psi$:

\begin{description}
\item{\ding{192}\quad}  $\psi=\frac{P}{Q}$, where $P\in\Pi_m$ and $Q\in\Pi_n$ with some $m, n\in\mathbb{N}$, but $\psi$ is not a polynomial,
\item{\ding{193}\quad} $P$ and $Q$ have no common roots,
\item{\ding{194}\quad} $P(0)=Q(0)=1$.
\end{description}

\begin{rem}  We will see that removing these assumptions (i.e., not
excluding removable singularities {\em a priori}, or considering, when setting up the form of
the families of rational functions, the case $P(0)=Q(0)=0$ as well with
interpreting $\psi(0)=1$ as $\lim_0 \psi=1$ in the order conditions) does not
make a difference in the optimal values $R_{s/s,p} $ and $\widehat{R}_{s/s,p}$ 
in the classes we are going to consider.  Nevertheless, these assumptions are 
convenient in allowing us to immediately use results from \cite{vdgk}.
\end{rem}

\begin{defn}[({\cite[Definition 3.2]{vdgk}})]\label{Bdef}
Suppose that $\psi$ satisfies  Assumptions \ding{192}$-$\ding{194} above. Let $A^+(\psi)$ denote the set of poles of $\psi$ with non-negative imaginary part. If $\alpha\in A^+(\psi)$, we set \[I(\alpha):=\{ x\in\mathbb{R} : \alpha\in A^+(\psi)   \mathrm{\ is\ the\ unique\ pole\ closest \ to\  } x \}.\] 
The disjoint union of these intervals is the set $\mathbb{R}$ with only finitely many exceptions. Now we let
\[ 
B(\psi):=  \begin{cases}
-\inf I(\alpha_0) & \text{if } 0\in I(\alpha_0) \text{ for some positive real pole } \alpha_0\in A^+(\psi)\cap (0,+\infty), \\
0 & \text{otherwise.}
\end{cases}
\]
\end{defn}

Note that $B(\psi)$ is determined solely by the location of the poles of $\psi$. The significance of this quantity is highlighted by the following theorem.


\begin{thm}[({\cite[Theorem 3.3]{vdgk})}]\label{vdgkTheorem3.3}
Suppose that $\psi$ satisfies  Assumptions  \ding{192}$-$\ding{194} above.  Then \[R(\psi)\le B(\psi).\]
\end{thm}

\begin{thm}[({\cite[Corollary 3.4]{vdgk}})]\label{vdgkCorollary3.4}
Suppose that $\psi$ satisfies  Assumptions  \ding{192}$-$\ding{194} and $\psi$ has no positive real pole. Then $R(\psi)=0$.
\end{thm}

Additionally, the negative real roots of the derivatives of $\psi$ form upper bounds on $R(\psi)$, as shown by the next theorem.

\begin{thm}[({\cite[Lemma 4.5]{vdgk}})]\label{vdgkLemma4.5}
Suppose that $\psi$ satisfies Assumptions  \ding{192}$-$\ding{194}, and $\psi^{(\ell)}(x)=0$ for some $0\ge x\in\mathbb{R}$ and $\ell\in \mathbb{N}$. Then $R(\psi)\le -x$.
\end{thm}

We will also use the fact that under certain assumptions absolute monotonicity at the left endpoint of  an interval implies absolute monotonicity on the whole interval.

\begin{thm}[({\cite[Lemma 3.1]{kc}})] \label{absmon_int}
Let $\psi = \frac{P}{Q}$ be absolutely monotonic at some $x<0$, where $P$ and $Q$ are polynomials
and $Q$ has no zeros in $(x,0]$. Then $R(\psi)\ge -x$.
\end{thm}

\begin{rem}\label{remarkL(x)}  In \cite{vdgk}, formula (4.3) introduces an auxiliary quantity \[L(x)=\max\left(0,m-n+1,\max\{\ldots\}\right)\] that is used in their algorithm to compute $R(\psi)$ for a given rational function. It may happen however that the $\{\ldots\}$ set above is empty (corresponding to the SDIRK case, for example), when a correct interpretation of this $\max\varnothing$ is $-\infty$ (or, say, $0$).
\end{rem}

\section{Main results}\label{sectionmainresults}
In this section we state the main results and describe the approach of the proofs.  The 
proofs themselves are deferred to Sections \ref{sectionIRKs2}--\ref{appendixsection}.

For a given set $\Pisp$ or $\Pihatsp$, let
$q\equiv q(s,p)$ denote the number of parameters needed to describe the set.
Generically, the class $\Pisp$ can be written
as a family in $q=2s-p$ parameters, while the class $\Pihatsp$ can be written as a family
in $q=s+1-p$ parameters.  In the cases where $q=0$,
these sets contain a finite number of members ("finitely many" is of course understood in the sense of functions, \textit{i.e.,} a normalized representation is chosen: if $\psi=\frac{P}{Q}$ with $\psi(0)=1$, then $P(0)=Q(0)=1$).  Our investigation
is restricted to sets consisting of $q$-parameter families of rational functions with $q=0$ (3 such cases), $q=1$ (5 cases), $q=2$ (2 cases), $q=3$ (1 case), or $q=4$ (a lower bound in 1 case).

\subsection{A lower bound on $R_{3/3,2}$}
In Section \ref{counterexamplesection}, we construct a function $\psi\in\Pi_{3/3,2}$
with $R(\psi) > 6.7783 >6$.  This shows that Conjecture \ref{conj:vdgk} does not hold
for $m=n=3$.
However, the exact value of $R_{3/3,2}$ is still unknown,
because in order to describe all rational functions $\psi$ in $\Pi_{3/3,2}$, we
need 4 parameters, rendering the (exact or numerical) optimization within this
class impossible with our current techniques. In contrast, Theorem
\ref{Thm2.2} in Section \ref{determinationRhatp2s34} asserts that
$\widehat{R}_{3/3,2}=6$.

\subsection{Determination of $\Rsp$ for $s\le4$, $q\le1$}

\subsubsection{Known values of $R_{1/1,p}$ and $R_{2/2,p}$}\label{knownvaluessle4qle1}
For $p\ge 3$, $\Pi_{1/1,p}$ is empty.  The set $\Pi_{1/1,2}$ consists only of
the function $z\mapsto\frac{1+z/2}{1-z/2}$ with $R=2$, so $R_{1/1,2}=2$.  

For $p\ge 5$, $\Pi_{2/2,p}$ is empty.  The set $\Pi_{2/2,4}$
consists of a single function: the $(2,2)$ Pad\'e approximation of the exponential 
function.  This function has $\psi^{(6)}(0)=0$, so by
Theorem \ref{vdgkLemma4.5} we have $R_{2/2,4}=0$ (but see Remarks \ref{remarkIRKobstacle} and \ref{remark2.4} as well).
The same result was deduced in a different way in 
\cite[Theorem 5.1]{vdgk}.

If $\psi=\frac{P}{Q}\in \Pi_{2/2,2}$, and $P$ and $Q$ have no common roots, then \cite[Section 6.3]{vdgk}
says that $R(\psi)\le 4$.  If $P$ and $Q$ have a common root, then it is easily
seen that $\psi\in \Pi_{1/1,2}$ also, hence $R(\psi)\le 2$.  For the function
$\psi(z)=\frac{\left(1+z/4\right)^2}{\left(1-z/4\right)^2}$ we have
$R(\psi)=4$, so $R_{2/2,2}=4$.

The exact value of $R_{2/2,3}$ was previously unknown; it is determined in Section \ref{sectionIRKs2},
thus completing all cases with $1\le s\le 2$.

\subsubsection{New exact values of $\Rsp$}
For the cases in which $p=2s$, the set $\Pisp$ has only one member:
the $(s,s)$ Pad\'e approximation to the exponential.  The value of
$\Rsp$ can be obtained by computing $R(\psi)$ for this unique member; see also Remark \ref{remark2.4}.

For the cases in which $p=2s-1$, the set $\Pisp$ is a one-parameter
family.  In these cases, the method of proof we use is the following procedure (possibly in an iterative manner).

\begin{description}
\item{1.\quad} Conjecture an optimal parameter value $a^*$ by inspection of $B(\psi_a)$ and
            the first several derivatives of $\psi_a$.
\item{2.\quad} Rigorously exclude all the parameter values $\mathbb{R}\setminus
    \{a^*\}$ by appealing to Theorem \ref{vdgkLemma4.5} and the intermediate
    value theorem.
\item{3.\quad} Explicitly compute a formula for the $k^\mathrm{th}$ derivative
            ($k\in\mathbb{N}$ arbitrary) and prove that
            $\psi^{(k)}_{a^*}\Big|_{[-x^*,0]}\ge 0$, where $x^*>0$ is as large as possible due to Step 1. 
\end{description}

Decimal approximations of proven optimal values are given in Table \ref{tableofoptimalvalues},
along with some other properties.  Exact values are given in the sections indicated.

\begin{table}
\centering
    \begin{tabular}{c|l|c|c|c|l}
        $(s, p)$ & $R_{s/s,p} $ & deg($a^*$) & deg($R_{s/s,p}$) & Section(s) & Obstacle (with $R:=R_{s/s,p}$) \\ \hline 
	$(2,2)$     & $=4^\dagger$ & $+$ & 1 & \ref{knownvaluessle4qle1} & $\psi(-R)=\psi^\prime(-R)=0$ \\
	$(2,3)$     & $\approx 2.732050^{\ddagger}$ & $2$  & $2$ & \ref{sectionIRKs2} & $\psi^\prime(-R)=0$ \\
	$(2,4)$     & $=0^\dagger$ & $-$ & 1 & \ref{knownvaluessle4qle1} & $R=B=0$ or $\psi^{(6m)}(-R)=0$ \\
	$(3,5)$     & $\approx 2.301322$ & $6$ & $6$  &\ref{sectionIRKs3p5} & $R=B$ \\ 
        $(3,6)$     & $\approx 2.207606^{\ddagger}$ & $-$ & $3$ & \ref{IRKs3p6section} & $R=B$ \\ 
	$(4,7)$     & $\approx 2.743911$ & $30$ & $30$ & \ref{sectionIRKs4p7} and \ref{appendixsection9.1} & $R=B$ \\ 
	$(4,8)$     & $=0^\dagger$ & $-$ & 1 & Remark \ref{remark2.4} & $R=B=0$ \\ 
    \end{tabular}
\caption{Optimal $R_{s/s,p}$ values, together with the algebraic degree of the optimal parameter
choice within the given parametrization, the algebraic degree of $R_{s/s,p}$, the section number in which $R_{s/s,p}$ is
given as an \textit{exact} algebraic number, and the factor that limits the
optimal value (see Section \ref{earlierresultssection}). Superscripts $^\dagger$ and 
$^\ddagger$ indicate, respectively, that the optimal $R_{s/s,p}$ value was
already proved earlier exactly or was presented earlier numerically. A $-$ means that the corresponding $\Pi_{s/s,p}$ class contains only finitely many elements hence the parameter $a$ is not present, whereas a $+$ symbol denotes that the class cannot be described by only one parameter.  As for the integer  $m$ in the $(s,p)=(2,4)$ case, see Remark \ref{remarkIRKobstacle}.
\label{tableofoptimalvalues}}
\end{table}

\begin{rem}\label{remarkIRKobstacle}
We remark that the optimal $\psi$ functions corresponding to the $(s,p)=(2,2)$ and $(s,p)=(2,3)$ rows of 
Table \ref{tableofoptimalvalues} are also elements of the corresponding 
class $\Pihatsp$.  Hence for those functions $B=+\infty$, and so $R<B$. As for
the optimal $\psi$ corresponding to $(s,p)=(2,4)$, we have checked that for $0< k\le 100$, $\psi^{(k)}(0)=0$ if
and only if $k$ is divisible by 6. 
\end{rem}

\begin{rem}
Our computation of $R_{2/2,3}$ confirms the corresponding numerical result
presented in \cite[Table 6.1]{vdgk}. 
\end{rem}


\begin{rem}\label{remark2.4}
The classes $\Pisp$ where $p=2s$ each contain a single element: the $(s,s)$
Pad\'e approximation to the exponential.  For $s$ even, it was shown that
$R_{s/s,2s}=0$ in \cite[Theorem 5.1]{vdgk}.  For $s=3$, the optimal value
$R_{3/3,6}\approx 2.2076$ is presented numerically in \cite[Table 5.3]{vdgk},
and in Section \ref{IRKs3p6section} we have given its exact value as an algebraic
number of degree 3.  
\end{rem}

\subsection{Determination of $\Rhatsp$ for $s\le 4$}\label{determinationofRhatforsle4}

The known and newly obtained $\Rhatsp$ values are summarized in Table \ref{tableofSDIRKoptimalvalues}. Now let us briefly interpret our results in these classes.

  For $1\le s\le 4$, the classes $\widehat{\Pi}_{s/s,p}$ are empty when $p\ge s+2$; the trivial class  $\widehat{\Pi}_{1/1,2}$ is mentioned in Section \ref{determinationRhatp2s34}. So by taking into account Table \ref{tableofSDIRKoptimalvalues}, we have a \textit{complete description} of the optimal $\widehat{R}_{s/s,p}$ values for $s\le 4$ and $p \ge 2$.

 The classes $\widehat{\Pi}_{3/3,2}$ and $\widehat{\Pi}_{4/4,2}$ will be treated separately in Section \ref{determinationRhatp2s34} so as to be able to give a detailed exposition of their (more challenging) proofs.

The $\widehat{R}_{s/s,3}$ values for $2\le s \le4$ support Conjecture \ref{conj:hatp3}, being stronger
than \cite[Conjecture 3.2]{fs}. In other words, we have proved  \cite[Conjecture 3.2]{fs} for $3\le s\le 4$.

\begin{rem}\label{remark2.5onR(A,b)=Rhat}
On one hand, according to (\ref{Rineq}), if $A$ and $b$ correspond to any SDIRK method of $s$ stages and order $p$, then we have $R(A,b) \le \widehat{R}_{s/s,p}$. On the other hand, an SDIRK method is constructed in \cite{fs} for each pair $(s,p)$ satisfying $2\le s\le 4$ and $2\le p\le 3$ such that $R(A,b) = \widehat{R}_{s/s,p}$. Based on this, one might suspect that $\Rhatsp$
is equal to the optimal $R(A,b)$ radius for $s$-stage, order $p$ SDIRK methods
for each $s, p \ge 2$. However, Remark \ref{remark2.6s3p4} shows that this is not the case in general.
\end{rem}

\begin{rem}\label{remark2.6s3p4}
There exist exactly 3 functions in the set $\Pihat_{3/3,4}$.
These 3 rational functions are also mentioned in \cite[Section 3.4.1]{fs}.
The method satisfying the non-negativity condition $K\ge 0$ in \cite{fs} is the
one whose stability function yields the optimal $\widehat{R}_{3/3,4}$ value.
This function has radius of absolute monotonicity $\widehat{R}_{3/3,4}\approx 3.2872$,
but the corresponding optimal SDIRK method has only $R(A,b)\approx 1.7587$.
\end{rem}

\begin{rem}
The optimal $\psi_{a^*}$ in the $\widehat{\Pi}_{4/4,4}$ class is different from the stability function of the optimal method of \cite[Section 3.4.2]{fs} in this class obtained by numerical search. In other words, according to the numerical tests in \cite{fs}, the optimal $R(A,b)$ in the SDIRK $s=p=4$ case is $\approx 4.2081<\widehat{R}_{4/4,4}$ (\textit{c.f.} Remark \ref{remark2.6s3p4}).  
\end{rem}

\begin{table}
\centering
    \begin{tabular}{c|l|c|c|c|l}
        $(s, p)$ & $\widehat{R}_{s/s,p}$ & deg($a^*$) & deg($\widehat{R}_{s/s,p}$) & Section(s) & Derivatives that vanish at $-\widehat{R}_{s/s,p}$ \\ \hline
	$(2,2)$     & $=4^\dagger$  & $1$ & 1 & \ref{knownvaluessle4qle1} & $\ell=0,1$ \\
	$(3,2)$     & $=6$ & $+$ & 1 & \ref{sectionSDIRKs3p2} & $\ell=0,1,2$ \\ 
	$(4,2)$     & $=8$ & $+$ & 1 & \ref{sectionSDIRKs4p2} and \ref{appendixsection9.2} & $\ell=0,1,2,3$ \\ 
	$(2,3)$     & $\approx 2.732050^\ddagger$  & $-$ & 2 & Remark \ref{remarkIRKobstacle} & $\ell=1$ \\
	$(3,3)$     & $\approx 4.828427$ & $2$ & $2$&\ref{sectionSDIRKs3p3} & $\ell=1,2$ \\ 
	$(4,3)$     & $\approx 6.872983$ & $+$ & $2$ & \ref{section8.3} & $\ell=1,2,3$ \\ 
        $(3,4)$     & $\approx 3.287278$ & $-$ & $9$& \ref{sectionSDIRKs3p4} & $\ell=0$ \\ 
	$(4,4)$     & $\approx 5.167265$ & $9$ & $9$ & \ref{section8.2} & $\ell=0,1$ \\ 
	$(4,5)$     & $\approx 3.743299$ & $-$ & $12$&\ref{section8.1} & $\ell=1$ \\ 
    \end{tabular}
\caption{Optimal $\widehat{R}_{s/s,p}$ values, together with the algebraic
degree of the optimal parameter value in the given parametrization and the algebraic degree of $\widehat{R}_{s/s,p}$, the
section number in which they are given as \textit{exact} algebraic numbers,
and the first few derivatives of $\psi$ that vanish at $-\widehat{R}_{s/s,p}$ (see Theorem \ref{vdgkLemma4.5}).  The
symbols $^\dagger, ^\ddagger, +, -$ have the same meaning as in Table
\ref{tableofoptimalvalues}.}\label{tableofSDIRKoptimalvalues}
\end{table}

Finally, we give some additional information on the structure of the optimal $\psi$ functions in Table \ref{tableSDIRKstructures}.

\begin{table}[h!]
\centering
    \begin{tabular}{c|c}
        $(s, p)$    & Form of the optimal $\psi$ or its derivative \\ \hline 
	 $(2,2)$    &   $\psi(z)=(1+z/4)^2/(1-z/4)^2$ \\
        $(3,2)$     &   $\psi(z)=(1+z/6)^3/(1-z/6)^3$ \\
        $(4,2)$     &   $\psi(z)=(1+z/8)^4/(1-z/8)^4$ \\
        $(2,3)$     &  $\psi^\prime(z)=\left(1+\frac{1}{2} \left(\sqrt{2^2-1}-1\right) z\right)/\left(1+\frac{1}{6} \left(\sqrt{2^2-1}-3\right) z\right)^3$ \\ 	
	$(3,3)$     &  $\psi^\prime(z)=\left(1+\frac{1}{4} \left(\sqrt{3^2-1}-2\right) z\right)^2/\left(1+\frac{1}{8} \left(\sqrt{3^2-1}-4\right) z\right)^4$\\ 
	$(4,3)$     &    $\psi^\prime(z)=\left(1+\frac{1}{6} \left(\sqrt{4^2-1}-3\right) z\right)^3/\left(1+\frac{1}{10} \left(\sqrt{4^2-1}-5\right) z\right)^5$\\
        $(4,4)$     &$\psi(z)=(1+0.193526 z)^2 \left(1+0.224416 z+0.0437638 z^2\right) /(1-0.0971331 z)^4$ \\ 
    \end{tabular}
\caption{Each row shows the appropriate derivative of the optimal element in $\widehat{\Pi}_{s/s,p}$  having multiple roots. In the $(s,p)=(4,4)$ case, approximate constants are used. In the $(s,p)\in\{(3,4),(4,5)\}$ cases, no multiple roots were found among the first few derivatives of the optimal $\psi$ functions.}\label{tableSDIRKstructures}
\end{table}

\subsection{Determination of $\widehat{R}_{s/s,2}$ for $3\le s\le4$}\label{determinationRhatp2s34}
The case $p=2$ is of particular interest in light of the conjectures
presented in Section \ref{sec:intro}.
\begin{thm}\label{Thm2.2}
Fix $1\le s\le 4$.  Then we have $\widehat{R}_{s/s,2}=2s$, and the unique
$\psi\in \widehat{\Pi}_{s/s,2}$ that attains $R(\psi)=2s$ is
\begin{align} \label{opt_f}
\psi(z)=\frac{\left(1+\frac{z}{2s}\right)^s}{\left(1-\frac{z}{2s}\right)^s}.
\end{align}
\end{thm}

The values $R_{1/1,2}=2$ and $R_{2/2,2}=4$ were determined already in \cite{vdgk}
and formed part of the basis for Conjecture \ref{conj:vdgk}. Moreover, 
we have seen in Section \ref{knownvaluessle4qle1} that the optimal elements in 
$\Pi_{1/1,2}$ and $\Pi_{2/2,2}$ are also elements of $\widehat{\Pi}_{1/1,2}$ and $\widehat{\Pi}_{2/2,2}$,
respectively, so  $\widehat{R}_{1/1,2}=R_{1/1,2}=2$ and $\widehat{R}_{2/2,2}=R_{2/2,2}=4$.
For the sake of completeness, we give two
short and direct proofs of the equality $\widehat{R}_{2/2,2}=4$ at the beginning of Section \ref{sectionSDIRKs2s3}, illustrating the two proof strategies we use in this work, and simultaneously proving the $s=2$ case of Theorem \ref{Thm2.2}. The $3\le s\le 4$ cases are new. The proof for $s=3$ is given in Section \ref{sectionSDIRKs3p2}, while the $s=4$ case is described in Sections \ref{sectionSDIRKs4p2} and \ref{appendixsection9.2}.
Theorem \ref{Thm2.2} confirms Conjecture \ref{conj:hat} for $s\le 4$, furthermore,  
it also proves Conjecture 3.1 in \cite[Section 3.2]{fs} for $3\le s\le 4$ (the truth of this conjecture for $1\le s \le 2$ being already established in \cite{fs}) regarding SDIRK methods that are optimal with respect to $R(A,b)$.

We give here just the common ingredients used in both proofs ($s=3$ and $s=4$), then the proofs are finished in the corresponding sections as indicated above.\\

For $s\ge 3$, any $\psi\in \widehat{\Pi}_{s/s,2}$ can be written as 
\begin{equation}\label{generalSDIRKpsi}
\psi(z)=\frac{1+\left(1-a \binom{s}{1}\right)z+ \left(\frac{1}{2}-a \binom{s}{1}+a^2 \binom{s}{2}\right)z^2+\sum_{n=3}^{s} a_n z^n}{(1-a z)^{s} }
\end{equation}
with some $a, a_3, a_4, \dots, a_s \in\mathbb{R}$.  Let us denote the
numerator of (\ref{generalSDIRKpsi}) by $P(z)$ and set $Q(z):=1-a z$, so that
$\psi=\frac{P}{Q^s}$.  

The next step is to exclude the non-positive $a$ values. Since Theorem \ref{Thm2.2} will be proved via the uniqueness argument

\smallskip

\begin{center}
$R(\psi)=2s$ implies the unique form (\ref{opt_f}) of $\psi$,
\end{center}

\smallskip

\noindent and $2s>s-1$, the assumption of the lemma below is justified.

\begin{lem}\label{a_pos}
Let $\psi\in\widehat{\Pi}_{s/s,2}$ be given in the form \eqref{generalSDIRKpsi}, and suppose that
$R(\psi)>s-1$.  Then the parameter $a$ appearing in \eqref{generalSDIRKpsi} satisfies
$a>0$.
\end{lem}
\begin{proof}
If $a=0$, then $\psi$ is a polynomial approximating the exponential function to order $p=2$ near the
origin. For such polynomials, $R(\psi)\le s-1$ \cite[Theorem~2.1]{kp}. For $a<0$,
$\psi$ has no positive real poles, so we will apply Theorem \ref{vdgkCorollary3.4} to show that $R(\psi)=0$. 
We need only to verify Assumptions \ding{192}--\ding{194}.
Assumptions \ding{192} and \ding{194} are automatically satisfied by functions of the form \eqref{generalSDIRKpsi}, so we are done if Assumption \ding{193} is also fulfilled. If not, then except at $z=1/a$, $\psi$
can be expressed as $\widetilde{P}/\widetilde{Q}$, where $\widetilde{P}$ and $\widetilde{Q}$
have degree $\widetilde{s}<s$ and no common roots.  Clearly,  $\widetilde{s}>0$, since $p=2$ still holds. 
Moreover, $\widetilde{P}(0)/\widetilde{Q}(0)=P(0)/Q(0)=1$, so we can assume $\widetilde{P}(0)=\widetilde{Q}(0)=1$. Now application of Theorem \ref{vdgkCorollary3.4} to $\widetilde{P}/\widetilde{Q}$ shows that $a< 0$ implies $R(\psi)=0$.
\end{proof}

Now a useful necessary condition is derived based on the fact that if a polynomial in $k$ is non-negative for all $k\in \mathbb{N}$, then its leading coefficient is also non-negative. Assumption $a>0$ in the next lemma is guaranteed of course by the previous lemma.

\begin{lem}
Suppose that $\psi=\frac{P}{Q^s}$ given by \eqref{generalSDIRKpsi} with $a>0$
is absolutely monotonic at some $x<0$.  Then
\begin{equation}\label{limitcondition}
\lim_{k\to\infty} \frac{1}{k^{s-1}}  \sum_{m=0}^{s} \frac{1}{m!}\binom{s-1+k-m}{s-1}
\left(\frac{1-ax}{a}\right)^{m} P^{(m)}(x)\ge 0.
\end{equation}
\end{lem}
\begin{proof}
From \eqref{generalSDIRKpsi} it can be shown by induction that the $k^{\mathrm{th}}$ derivative of $\psi$ ($k\in\mathbb{N}$)
is
\begin{equation}\label{psigeneralderivative}
\psi^{(k)}=\left(\frac{P}{Q^s}\right)^{(k)}=\frac{ \left(-Q^\prime\right)^{k} k!}{Q^{s+k}}\  \sum_{m=0}^{\min(k,s)} \frac{1}{m!}\binom{s-1+k-m}{s-1} \left(-\frac{Q}{Q^\prime}\right)^{m} P^{(m)}. 
\end{equation}
Suppose that $\psi=\frac{P}{Q^s}$ is absolutely monotonic at some $x<0$.
We have now $-Q'(x) = a>0$ and 
$Q(x)>0$, so \eqref{psigeneralderivative} implies
\begin{equation}\label{oldlimitcondition}
\lim_{k\to\infty} \frac{1}{k^{s-1}}  \sum_{m=0}^{\min(k,s)} \frac{1}{m!}\binom{s-1+k-m}{s-1}
\left(-\frac{Q(x)}{Q^\prime (x)}\right)^{m} P^{(m)}(x)\ge 0.
\end{equation}
Note that the above limit always exists since
the sum is a polynomial in $k$ of degree at most $s-1$ and the sum is
finite.
\end{proof}

After these preparations, we show in Section \ref{sectionSDIRKs3p2} for $s=3$,
and in Sections \ref{sectionSDIRKs4p2} and \ref{appendixsection9.2} for $s=4$,
that \eqref{opt_f} is the unique solution of the following non-linear system of
polynomials in the variables $a>0$ and $a_3, a_4, \dots, a_s \in\mathbb{R}$:
\begin{equation}\label{firstgroup}
\psi^{(k)}(-2s)\ge 0\quad \mathrm{for}\quad k=0,1,\ldots s-1
\end{equation}
and
\begin{equation}\label{secondgroup}
\forall k\in\mathbb{N}, k\ge s:\quad \psi^{(k)}(-2s)\ge 0.
\end{equation}

The strength of the above uniqueness-type argument is that we could handle more than one parameter within this framework. The disadvantage is, however, that the optimal value (here $R=2s$) must be known in advance. The reason for separating indices $0\le k\le s-1$ and $k\ge s$ in the above non-linear system will be explained in Remark \ref{gapconditionremark}.

\subsection{Further results and questions}\label{furtherresultssection}
In this section we present some intermediate results that may prove useful
in future studies of the radius of absolute monotonicity. 

The motivation for Section \ref{traceinequalitiessection} came from Conjecture \ref{conj:kmg}. Let 
$A, b$ denote the coefficients of an arbitrary RK method; then $R(A,b)>0 \implies A\ge 0$ \cite[Observation~5.2]{SSPbook}
(again, matrix inequalities are understood componentwise).
 The trace inequality (\ref{trineq}) for non-negative matrices 
 is used to derive a certain non-linear relation (\ref{non-linear relation
 origin}) between the first few coefficients of the polynomial appearing in the
 numerator of the stability function (\ref{stabilityfunctionformula}). 
 
Next we turn our attention to Conjecture \ref{conj:hat}.
Let us fix $s\ge 3$ and $p=2$, choose an arbitrary SDIRK method with
coefficients $A, b$, and consider its stability function
$\psi \in \Pihat_{s/s,2}$  as described by formula  (\ref{generalSDIRKpsi}).
Let $a_0, a_1$ and $a_2$ denote the first few coefficients of the numerator of $\psi$, that is $a_0:=1$, $a_1:=1-a \binom{s}{1}$ and
$a_2:=\frac{1}{2}-a \binom{s}{1}+a^2 \binom{s}{2}$, then a simple computation
shows that (\ref{non-linear relation origin}) is satisfied---in fact, with
\textit{equality}---despite the fact that the non-negativity condition $A\ge
0$ here is \textit{not} assumed. If $3\le s \le 8$, then Lemma \ref{lemma2.5}
in Section \ref{section2.5.3poly} gives a remarkable uniqueness result for
polynomials: the unique polynomial given by Lemma \ref{lemma2.5} is identical
to the numerator of the conjectured optimal and unique $\psi$ function
appearing in Theorem \ref{Thm2.2} (or in Conjecture \ref{conj:hat}). Apart
from the fact that we could prove Lemma \ref{lemma2.5} only for $3\le s \le
8$, the missing link is the following: the uniqueness result in Theorem
\ref{Thm2.2} is essentially obtained under the condition 
\begin{equation}\label{transition1}
\psi^{(k)}(-2s)\ge 0\quad \mathrm{for}\quad k\in\mathbb{N},
\end{equation}
whereas in Lemma \ref{lemma2.5} we assumed 
\begin{equation}\label{transition2}
P^{(k)}(-2s)\ge 0\quad \mathrm{for}\quad   k=0,1, \ldots, s.
\end{equation}
We are very curious whether a result similar to Lemma \ref{lemma2.5} could lead to a proof of Theorem \ref{Thm2.2} for general $s$ values, that is, to a proof of Conjecture \ref{conj:hat}.

In Section \ref{structureofPihatssp} we give a formula for the elements of the $\widehat{\Pi}_{s/s,p}$ class for general $s$ and $p$ values, while in Section \ref{possibleextensions} we briefly touch on some of our proof techniques and their possible extensions.

\begin{rem}
We have found the following conjecture whose assumption---with $p_m=P^{(m)}(x)/m!$, $q=Q(x)$ and $\widetilde{q}=Q^\prime (x)=-a$ for some $x\in [-2s,0]$---is similar to the sum in (\ref{psigeneralderivative}) apart from the factor $(-1)^{k-m}$ and the fact that the $p_m$ quantities can be independent of one another. The conjecture has been proved by \Mma for $1\le s\le 10$, and we have a  "manual" proof for $s=2$. We do not know whether the conjecture could be used in a transition from conditions (\ref{transition1}) to (\ref{transition2}).  
\end{rem}

\begin{conj} Fix any $s\in\mathbb{N}^+$, $0\ne q\in\mathbb{R}$, $0<\widetilde{q}\in\mathbb{R}$ and $p_m\in\mathbb{R}$ ($m=0,1,\ldots,s+1$), and suppose that for each $k=0,1,\ldots, s+1$ we have
\[
\sum_{m=0}^{\min(k,s)} \binom{s-1+k-m}{s-1}\cdot  p_m\cdot q^m\cdot (-\widetilde{q})^{k-m}\ge 0.
\]
Then $p_0\ge 0$ and for $m=1,2,\ldots, s+1$, we have 
$p_m=p_0\binom{s}{m}\cdot\left(\frac{\widetilde{q}}{q}\right)^m$.  Moreover, the above sum  is equal to $p_0$, if $k=0$, and $0$, if $k=1, 2, \ldots, s+1$.
\end{conj}

\subsubsection{Trace inequalities for non-negative matrices and the numerator of the stability function}\label{traceinequalitiessection}

Conjecture \ref{conj:kmg} in the $s=3$, $p=2$ case claims that $R(A,b)\le 6$ for any RK method with 
coefficients $A, b$. 
 The following general lemma was discovered while investigating this conjectured bound.
Due to the remark in the beginning of Section \ref{furtherresultssection}, $A\ge 0$ will be assumed throughout the current Section \ref{traceinequalitiessection}.

Let us fix a positive integer $s\ge 2$ and apply the following notation: if $A$ is an $s$-by-$s$ matrix, then $\tau:=\mathrm{tr}(A)$ and $\tau_k:=\mathrm{tr}(A^k)$ for any integer $k\ge 2$.

\begin{lem}\label{lemma2.12aboutthetrace} Fix a positive integer $n\ge 2$ and suppose that $A$ is an $s$-by-$s$ (componentwise) non-negative matrix. Then
\begin{equation}\label{trineq}
s^{n-1} \tau_n \ge \tau^{n}. 
\end{equation}
Moreover, in the $n=2$ case equality holds if and only if $a_{k,k}=a_{1,1}$ for $k=2,3,\ldots,s$ and $a_{i,j} a_{j,i}=0$ for all $i\neq j$, further, if $n\ge 3$ and $s^{n-1} \tau_n = \tau^{n}$, then $a_{k,k}=a_{1,1}$ for $k=2,3,\ldots,s$.
\end{lem}
\begin{proof}
If $A$ and $B$ are non-negative $s$-by-$s$ matrices, then 
$(AB)_{k,k}\ge a_{k,k} b_{k,k}$.  Applying this recursively, we get that tr($A^n)\ge \sum_{k=1}^s a_{k,k}^n$. Then the inequality between the $n^\mathrm{th}$ power mean and the arithmetic mean shows that
\[
\sqrt[n]{\frac{\mathrm{tr}(A^n)}{s}}\ge \sqrt[n]{\frac{\sum_{k=1}^s a_{k,k}^n}{s}}\ge \frac{\sum_{k=1}^s a_{k,k}}{s}=\frac{\mathrm{tr}(A)}{s},
\]
proving the trace inequalities. Now if $n\ge 3$ and $s^{n-1} \tau_n = \tau^{n}$, then we have equality in the power mean inequality, which implies $a_{k,k}=a_{1,1}$ for $k=2,3,\ldots,s$. Finally, suppose that $n=2$. Then $s \tau_2 = \tau^{2}$ holds if and only if we have equality in the power mean inequality and tr($A^2)= \sum_{k=1}^s a_{k,k}^2$. But the former holds if and only if $a_{k,k}=a_{1,1}$ for $k=2,3,\ldots,s$, while the latter holds if and only if $a_{i,j} a_{j,i}=0$ for all $i\neq j$.
\end{proof}

\begin{rem} \  \\ \indent $\bullet$ If $s=2$, $n=3$, $a_{1,1}=a_{2,2}=0$ and $a_{1,2}=a_{2,1}=1$, then $s^{n-1} \tau_n = \tau^{n}$, but $a_{1,2}\, a_{2,1}\ne 0$.\\ \indent $\bullet$ The constant $s^{n-1}$ is the best possible (as shown by $A=I$).\\ \indent $\bullet$ We can not expect a "converse" trace inequality, since for the matrix $A$ with 0's in the diagonal and with all other entries 1, we have tr$^2(A)=0$ and tr$(A^2)>0$.\\ \indent $\bullet$ For the conjectured optimal RK method satisfying $R(A,b)=6$, we have equality in (\ref{trineq}) for $n=2$.   
\end{rem}

Let us now fix any $s\ge 3$. By repeatedly using the formulae for the derivative of the determinant and the trace 
\[
\left(\det\Phi(\cdot)\right)^\prime = (\det\Phi(\cdot))\cdot \textrm{tr}(\Phi^{-1}(\cdot)\Phi^\prime(\cdot)) 
\quad \mathrm{and}\quad 
\left(\mathrm{tr\,}\Phi(\cdot)\right)^\prime = \mathrm{tr}(\Phi^\prime(\cdot)), 
\]
where $\Phi:\mathbb{R}\to \mathbb{R}^{s\times s}$ is a smooth---and for the first formula, invertible---matrix function, the cyclic invariance of the trace together with the $p=2$ order conditions, we see that the numerator of the stability function (\ref{stabilityfunctionformula}), $\det(I-zA + z \one b^\top)$,
can be written in the following form 
\[
P(z)=1+(1-\tau)z+\frac{1}{2}\left((1-\tau)^2-\tau_2\right)z^2+\sum_{k=3}^s a_k z^k,
\]
with suitable real
parameters  $a_k$ (there may be further restrictions on the $a_k$ parameters which are ignored here). Let us introduce $a_0:=1$, $a_1:=1-\tau$, $a_2:=\frac{1}{2}(a_1^2-\tau_2)$. Then, due to $A\ge 0$, we have $\tau\ge 0$ and $\tau_2\ge 0$, so  (\ref{trineq}) with $n=2$ implies $-\tau_2\le -\frac{\tau^2}{s}$, thus $a_2=\frac{1}{2}(a_1^2-\tau_2)\le \frac{1}{2}(a_1^2-\frac{\tau^2}{s})$, but $\tau=1-a_1$, hence we have derived a non-linear condition
\begin{equation}\label{non-linear relation origin}
\frac{1}{2}\left(a_1^2-\frac{(1-a_1)^2}{s}\right)\ge a_2.
\end{equation}

\subsubsection{Uniqueness results for polynomials closely related to the class $\Pihat_{s/s,2}$}\label{section2.5.3poly}

In this section, the main lemma is Lemma \ref{lemma2.5}, although Lemma \ref{boundson} can be of independent interest, giving lower and upper bounds on the coefficients of a general polynomial $P$ with $P(0)=1$ that is absolutely monotonic at a point $-r<0$, together with a uniqueness result.

\begin{lem}\label{boundson}
Let $s$ denote a fixed positive integer and set $P(z):=1+\sum_{n=1}^s a_n z^n$ with some real coefficients $a_n$. Fix any $r>0$ and suppose that $P^{(k)}(-r)\ge 0$ for all $k=0,1,\ldots,s$.
Then for each index $1\le n\le s$ we have 
$
0\le a_n \le \frac{\binom{s}{n}}{r^n}.
$
Moreover, if there is at least one $n$  $(1\le n\le s)$ with $a_n = \frac{\binom{s}{n}}{r^n}$, then $P(z)=\left(1+\frac{z}{r}\right)^s.$
\end{lem}
\begin{proof}
By assumption, with $\gamma_k :=\frac{P^{(k)}(-r)}{k!}r^k$ we have $\gamma_k\ge 0$  for $k=0,1,\ldots,s$. Taylor expansion, the binomial theorem and interchanging the order of summations show that 
\[
P(z)=\sum_{k=0}^s \gamma_k \left(1+\frac{z}{r}\right)^k=
\sum_{k=0}^s \sum_{n=0}^{k}\gamma_k \binom{k}{n}\frac{z^n}{r^n}=
\]
\[
\sum_{n=0}^s \sum_{k=n}^{s}\gamma_k \binom{k}{n}\frac{z^n}{r^n}=
\sum_{n=0}^s \left(\frac{1}{r^n}\sum_{k=n}^{s}\gamma_k \binom{k}{n}\right)z^n.
\]
Now fix $1\le n\le s$. By equating the coefficients of $z^n$ we get 
$
a_n=\frac{1}{r^n}\sum_{k=n}^{s}\gamma_k \binom{k}{n}, 
$
further, from the equality of the constant terms
\begin{equation}\label{gamma1}
\sum_{k=0}^{s}\gamma_k=1. 
\end{equation}
Non-negativity of the $\gamma_k$ coefficients implies $a_n\ge 0$ and (\ref{gamma1}) shows $0\le \gamma_k\le 1$ for all $0\le k\le s$. Finally we seek the maximum of $\sum_{k=n}^{s}\gamma_k \binom{k}{n}$ knowing $0\le \gamma_k$ and (\ref{gamma1}). Suppose that $\gamma_m>0$ for some $m$ with $0\le m\le s-1$. Then we define $\widetilde{\gamma_s}:=\gamma_s+\gamma_m$, $\widetilde{\gamma_m}:=0$ and $\widetilde{\gamma_k}:=\gamma_k$ for $s\ne k\ne m$. Clearly, $0\le \widetilde{\gamma_k}$ for each $0\le k\le s$ and $\sum_{k=0}^{s}\widetilde{\gamma_k}=1$, but (by using the convention that $\binom{m}{n}=0$ if $m < n$)
\[
\sum_{k=n}^{s}\widetilde{\gamma_k} \binom{k}{n}-\sum_{k=n}^{s}\gamma_k \binom{k}{n}=\binom{s}{n}\gamma_m+\binom{m}{n}(-\gamma_m)=
\gamma_m \left(\binom{s}{n}-\binom{m}{n}\right),
\]
and this last expression is strictly positive, because $1\le n\le s$ is fixed, $s > m$ and $\gamma_m>0$. We can therefore conclude that (for fixed $s\ge 1$ and $1\le n\le s$) the value of $\sum_{k=n}^{s}\gamma_k \binom{k}{n}$ over all $0\le \gamma_k$ ($k=0,1,\ldots,s$) and under condition (\ref{gamma1})  is maximal if and only if $\gamma_0=\gamma_1=\ldots=\gamma_{s-1}=0$ and $\gamma_s=1$. This property establishes the upper bound on $a_n$ and the uniqueness part as well.
\end{proof}
Next we present an interesting  lemma about representing a certain linear combination of the  three lowest order coefficients of a general polynomial of degree at most $s$ in terms of another linear combination of its derivatives evaluated at $-2s$.

\begin{lem}\label{linearcombinationlemma}
Fix any integer $s\ge 2$ and choose an arbitrary $P(z):=1+\sum_{k=1}^s a_k z^k$ (with complex coefficients). Then
\[
\sum_{k=0}^s \frac{(2s)^k (s-k)(s-k-1)}{k! s(s-1)}P^{(k)}(-2s)=1-4a_1+\frac{8s}{s-1}a_2.
\]
\end{lem}
\begin{proof}
Direct comparison of the $a_m$ coefficients on both sides by using some simple binomial identities (that can be obtained, for example, after (successively) differentiating the binomial expansion of $x\mapsto (1+x)^n$).
\end{proof}

The preceding two lemmas are used in proving the following uniqueness result in a class of polynomials of degree at most $s$ (unfortunately only in a restricted range $s\le 8$) which are absolutely monotonic at $-2s$ and whose lowest order coefficients satisfy a certain non-linear relation. The origin of condition  (\ref{non-linear relation}) is of course inequality (\ref{non-linear relation origin}) in a slightly different context.

\begin{lem}\label{lemma2.5} 
Let us pick an integer $s$ with $3\le s\le 8$, and set $P(z):=1+\sum_{k=1}^s a_k z^k$ with some real coefficients $a_k$. Suppose
$P^{(k)}(-2s)\ge 0$ for all $k=0,1, \ldots, s$, further, that
\begin{equation}\label{non-linear relation}
\frac{1}{2}\left(a_1^2-\frac{(1-a_1)^2}{s}\right)\ge a_2. 
\end{equation}
Then $P(z)=\left(1+\frac{z}{2s}\right)^s$.
\end{lem}
\begin{proof}
By Lemma \ref{boundson}  with $r=2s$, we have $a_n\ge 0$ for all $n=1, 2, \ldots, s$. On the other hand, since now $P^{(k)}(-2s)\ge 0$ by assumption, by applying Lemma \ref{linearcombinationlemma} we get that
\[
1-4a_1+\frac{8s}{s-1}a_2\ge 0.
\]
From this inequality we estimate $a_2$ as $a_2\ge \frac{s-1}{8s}(4a_1-1)$, so $\frac{1}{2}\left(a_1^2-\frac{(1-a_1)^2}{s}\right)\ge \frac{s-1}{8s}(4a_1-1)$, which can be factorized as
\[
\frac{\left(2 a_1-1\right) \left(2 a_1(s-1)+5-s\right)}{8 s}\ge 0,
\] 
showing (by using $s>1$ only) that either 
\begin{equation}\label{1stcase}
a_1\le \frac{1}{2}\quad\textrm{and}\quad a_1\le\frac{s-5}{2(s-1)}
\end{equation}
or
\begin{equation}\label{2ndcase}
a_1\ge \frac{1}{2}\quad\textrm{and}\quad a_1\ge\frac{s-5}{2(s-1)}.
\end{equation}
We now show that (\ref{1stcase}) can not occur for $3\le s\le 8$: indeed, for any $1 < s < 9$, we have $\frac{s-5}{2(s-1)}=\frac{1}{2}-\frac{2}{s-1} < \frac{1}{4}$,
therefore $a_1 < \frac{1}{4}$. But then (\ref{non-linear relation}) says that
\[
0\le a_2\le \frac{1}{2}\left(a_1^2-\frac{(1-a_1)^2}{s}\right)=\frac{a_1^2 (s-1)+2 a_1-1}{2 s}, 
\]
so $a_1^2 (s-1)+2 a_1-1\ge 0$, implying either $a_1\le\frac{1}{1-\sqrt{s}} < 0$ or $a_1\ge\frac{1}{1+\sqrt{s}}>\frac{1}{1+\sqrt{9}}=\frac{1}{4}$. The first case is ruled out by $a_1\ge 0$, and the second one is by $a_1 < \frac{1}{4}$.

Hence the only possibility is (\ref{2ndcase}) above. But then $\frac{1}{2}\le a_1 \le  \frac{\binom{s}{1}}{r^1}=\frac{s}{2s}=\frac{1}{2}$ by Lemma \ref{boundson}, and so---again by the same lemma---we have $P(z)=\left(1+\frac{z}{2s}\right)^s.$ 
\end{proof}

\begin{rem}
Notice that \cite{kp} contains many uniqueness results for polynomials with maximal radius of absolute monotonicity, but polynomials in that article should approximate the exponential function near the origin to at least first order, \textit{i.e.,} $P(0)=P^{\prime}(0)=1$ should hold. In Lemma \ref{boundson}, only $P(0)=1$ is required, whereas in Lemma \ref{lemma2.5} we assume $P(0)=1$ together with the non-linear condition (\ref{non-linear relation}).
\end{rem}

\subsubsection{The structure of $\widehat{\Pi}_{s/s,p}$}\label{structureofPihatssp}

 Let us fix $3\le s\in \mathbb{N}$. The non-trivial $p$ values are $2\le p\le s+1$, but the $p=s+1$ case is special, because then $\widehat{\Pi}_{s/s,p}$ contains only finitely many functions. For $2\le p\le s$, the general form of $\psi \in \widehat{\Pi}_{s/s,p}$ reads as
\begin{equation}\label{structureofPihat}
\psi(z)=\frac{\displaystyle{\sum_{m=0}^p\left(\sum_{k=0}^m \frac{(-1)^k a^k}{(m-k)!}\binom{s}{k}\right) z^m+\sum_{m=p+1}^s a_m z^m}}{(1-a z)^s},
\end{equation}
with suitable real parameters $a, a_{p+1}, \ldots, a_s$, and, of course, with the usual convention that $\sum_{m=n}^{N}(\cdot)=0$, if $n>N$. Formula (\ref{generalSDIRKpsi}) is recovered if we choose $p=2$ in (\ref{structureofPihat}).

\begin{rem}\label{remark2.18onoptimalconjecturedavalues}
The conjectured optimal $\psi$ (\textit{i.e.,} the one with maximal radius of absolute monotonicity) in the $\widehat{\Pi}_{s/s,2}$ class has $a=\frac{1}{2s}$, whereas in the $\widehat{\Pi}_{s/s,3}$ class it has  
$a=\frac{1}{2}\left(1-\sqrt{\frac{s-1}{s+1}}\right)$: in light of Section \ref{determinationofRhatforsle4}, these values have been  
verified for $3\le s\le 4$. See  \cite[formulae (3.1) and (3.2)]{fs} also.
\end{rem}

\begin{rem}
The optimal $\psi$ rational function in the $\widehat{\Pi}_{s/s,2}$ class for $1\le s\le 4$ satisfies the relation $\psi(-z)=\frac{1}{\psi(z)}.$
\end{rem}

\subsubsection{Possible extension of certain proofs and some questions}\label{possibleextensions}

\begin{rem}\label{remark1.13aboutThm4.4}
In \cite[Theorem 4.4]{vdgk}, a necessary and sufficient condition is formulated for a general rational function $\psi$ satisfying  Assumptions \ding{192}--\ding{194} in our Section \ref{earlierresultssection} to be absolutely monotonic on an interval  $[x,0]\subset (-B(\psi),0]$. This condition, among others, requires checking the positivity of the "dominant coefficient" $c(\alpha_0,\mu(\alpha_0))$, being the numerator in the partial fraction decomposition of $\psi$  corresponding to the positive real pole $\alpha_0$ in  Definition \ref{Bdef} in our work, and having the highest pole order. It can be seen that $c(\alpha_0,\mu(\alpha_0))>0$
is equivalent to our necessary condition (\ref{limitcondition}) formulated in the $\widehat{\Pi}_{s/s,2}$ class, if "$\ge 0$" is replaced by "$>0$" in (\ref{limitcondition}). Moreover, "$=0$" in (\ref{limitcondition}) holds precisely if the numerator and denominator of $\psi$ have a common root. We add that we have not used \cite[Theorem 4.4]{vdgk} directly, because the explicit construction of its functions
$F(k,x)$, $L(x)$ and $K(x)$ (\textit{c.f.} our Remark \ref{remarkL(x)} also)---now possibly depending on an additional parameter---would not be straightforward. Instead, we reproduce the "partial fraction decomposition + factoring out the dominant term + checking the sign of the remainder in finitely many cases" idea behind the proof of \cite[Theorem 4.4]{vdgk} only after the optimal parameter value $a^*$ and the optimal radius of absolute monotonicity $R_{s/s,p}$ (or $\widehat{R}_{s/s,p}$) have been conjectured.
\end{rem}

\begin{rem}[(on the explicit computation of $\psi^{(k)}$)] The sum in
(\ref{psigeneralderivative}) has a bounded number of terms for all
$k\in\mathbb{N}$, yielding a nice representation of the derivatives of functions in $\Pihat_{s/s,2}$. 
For $\psi\in \Pi_{s/s,p}$, a similar representation would require a sum with unbounded number of terms as $k$ increases,
which precludes the extension of our analysis for $\psi \in \widehat{\Pi}_{s/s,2}$
to the more general class.
Fortunately, the optimal $\psi\in \Pi_{2/2,3}$ is also an element of $\widehat{\Pi}_{2/2,3}$.
\end{rem}

\begin{rem} In the proof of Theorem \ref{Thm2.2} we wish to show that  (under condition $a>0$) $\psi^{(k)}(-2s)\ge 0$ ($k\in\mathbb{N}$) implies uniqueness. Instead of appealing to formula (\ref{psigeneralderivative}) for $\psi^{(k)}$, one could first write up the partial fraction decomposition of $\psi$ (depending on the parameters $a, a_3, \ldots, a_s$), then differentiate, and get $\psi^{(k)}$ in a different form. In this form, 
the "basis elements" 
\[
1, k+1, (k+1)(k+2), \ldots, \Pi_{m=1}^{s-1}  (k+m)
\]
appear naturally in the numerators of the partial fraction decomposition of $\psi^{(k)}$. Some of our
observations suggest that changing this "rising" basis to the "falling" one 
\[
1, k-1, (k-1)(k-2), \ldots, \Pi_{m=1}^{s-1}  (k-m)
\]
could be advantageous: for the optimal $\psi$ function, coefficients in the numerators in the falling basis are simpler (in fact, for $3\le s \le 4$, they are $0, 0, \ldots, 0, \frac{1}{s! s^{s-1}}$) than in the rising basis. We add that the sum in (\ref{psigeneralderivative}) provides us with a third type of $k$-basis.  We do not know whether this approach could be used to characterize the optimal $\psi\in \widehat{\Pi}_{s/s,2}$.

\end{rem}


\begin{rem}[(on the point and interval conditions)]\label{gapconditionremark} 
In our proof of Theorem \ref{Thm2.2} for $3\le s \le 4$ we show that $a>0$, (\ref{firstgroup}) and (\ref{secondgroup}) imply uniqueness. According to our investigations with \Mma however, different 
formulations of absolute monotonicity of $\psi$ on $[-2s,0]$ can lead to the same uniqueness result.
For example, if $s=3$ and $a>0$, then 
\[
\psi^{(k)}(-6)\ge 0\quad (k=0,1,2) \quad \mathrm{and}\quad \forall x\in [-6,0] : \psi^{(4)}(x)\ge 0
\]
implies uniqueness. Moreover, if $s=3$, $a>0$ and $\lambda>0$, then
\[
\psi^{(k)}(-6)\ge 0\quad (k=0,1,2) \quad \mathrm{and}\quad \forall x\in [-\lambda,0] : \psi^{(4)}(x)\ge 0
\]
implies uniqueness if and only if $0.103209\le \lambda \le 10.5836$ (the exact values of these algebraic numbers are known). We also have the following observations on the "gaps" in the $a>0$ parameter region:
if $s\in \{3,4,5\}$, $a>0$ and $\psi^{(k)}(-2s)\ge 0$  ($k=0,1,\ldots, s-1$), then \[a\ge \frac{1}{2s}\]
(recall that for the optimal $\psi$, we conjecture $a=\frac{1}{2s}$, as mentioned in Remark  \ref{remark2.18onoptimalconjecturedavalues}).
Moreover, there are other gaps simultaneously implied by the above conditions:  

$\bullet$ the interval $\frac{1}{6}<a<\frac{1}{2}$, if $s=3$, 

$\bullet$ the interval $\frac{1}{8}<a<\frac{7}{24}$, if $s=4$, and 

$\bullet$ the interval $\frac{1}{10}<a<0.1835$ (with the upper bound being the root of a cubic polynomial), if $s=5$, are also excluded.
\end{rem}

\begin{rem}[(on the \textit{Mathematica} implementations)] 
This work could not have been completed without \textit{Mathematica}, 
and especially, without one of its key
commands in polynomial algebra, \texttt{Reduce}. 
  Both in the conjecture and proof phase, Mathematica's abilities to manipulate high-order root objects symbolically, numerically or graphically played a crucial role. Throughout our investigations, we tried to apply symbolic methods to the fullest
extent possible. At the same time, each pivotal formula was tested numerically
with several parameter values (some by using different \Mma algorithms). For
typical checks we have used 25 digits of precision, but when huge coefficients
were involved, we switched to 100 or 1000 digits of precision.
Results of the
symbolic computations and those given by the numerical algorithms were also compared
graphically, and found to be in perfect agreement. In fact, as a
by-product of our investigations we have discovered and reported a few minor
\Mma bugs, occurring with very low probability, in contexts when certain
algebraic relations between different parameter values are satisfied.
\end{rem}

Let us close this section by posing two natural questions.\\

\textbf{Question.}  What properties---\textit{e.g.,} what geometrical configuration of the roots and poles---of the rational function $\psi$ determine whether   
 $R(\psi)=B(\psi)$ or $R(\psi)<B(\psi)$? If $R(\psi)<B(\psi)$, then which roots of which derivatives of $\psi$ will limit the value of $R(\psi)$ in the sense of Theorem \ref{vdgkLemma4.5}?\\

\textbf{Question} (\textit{c.f.} the uniqueness of the optimal polynomials in \cite{kp}). Let us consider a given $(s,p)$ pair ($s\ge 1, p\ge 2$). Is there always a unique rational function $\psi=\frac{P}{Q}\in \Pi_{s/s,p}$ (or $\psi=\frac{P}{Q}\in \widehat{\Pi}_{s/s,p}$) with $P(0)=Q(0)=1$ such that $R(\psi)=R_{s/s,p}$ (or $R(\psi)=\widehat{R}_{s/s,p}$)?

\subsection{Notation used in the proof sections}\label{notationsection}


For $3\le n\in \mathbb{N}, n\ge m \in \mathbb{N}^+$ and $a_j\in\mathbb{R}$
($j=0,1,\ldots, n$) the $m^{\mathrm{th}}$ root of the polynomial equation $a_n
z^n+a_{n-1}z^{n-1}+\ldots +a_1 z+a_0=0$ will be denoted by
\[
\mathrm{root}_m(a_n,a_{n-1},\dots,a_1,a_0).
\] We use the following ordering for the
roots: real roots are listed first in increasing order, then non-real complex
roots follow as their real part increases, finally,  complex roots with
equal real part are sorted according to increasing imaginary part (in this
context, multiple roots will not be encountered).  So for example, the value of
root$_m(1,0,0,0,0,0,-1)$ for $m=1,2,\ldots,6$ is $-1, 1,
\frac{-1-\ii\sqrt{3}}{2}, \frac{-1+\ii\sqrt{3}}{2}, \frac{1-\ii\sqrt{3}}{2}$
and $\frac{1+\ii\sqrt{3}}{2}$, respectively, corresponding to the equation $z^6-1=0$. We will
give numerical approximations to these roots as well, typically to 6  digits, if they occur only in auxiliary computations, but more digits will be displayed for significant constants. We will use the $\approx$ symbol without rounding, \textit{i.e.}, in the sense that all shown digits are correct. However, we underline that these
approximations are only given for the sake of the reader's convenience, and the proofs are not based upon them.

In our arguments, several exact algebraic
numbers (typically denoted by $\vr_{spn}$), polynomials (denoted by $P_{spn}$) or other
auxiliary functions ($f_{spn}$) will appear. Here the subscripts $s$
and $p$ correspond to the ones in $\Pi_{s/s,p}$ or $\widehat{\Pi}_{s/s,p}$,
whereas the positive integer $n$  serves as a counter within the actual section
and is increased sequentially. Superscript  $^*$ will often be used to denote optimal values within the family (\textit{e.g.}, $a_{sp}^*$ being the parameter value corresponding to the rational function with maximal radius of absolute monotonicity in the class $\Pi_{s/s,p}$ or $\widehat{\Pi}_{s/s,p}$). If $\psi$ depends on a parameter, say, on $a$,  then the $k^\mathrm{th}$ derivative ($k\in\mathbb{N}$) of $z\mapsto \psi_a(z)$ is denoted by $\psi^{(k)}_a$. Moreover, we will often use the simpler forms $R(a)$ and $B(a)$ instead of $R(\psi_a)$ and $B(\psi_a)$, based on the bijection $a \leftrightarrow \psi_a$. 

\begin{rem}
We remark that the function $a\mapsto R(a)\in [0,+\infty]$ can be discontinuous. We will also see examples when the function $a\mapsto B(a)\in [0,+\infty]$ is non-differentiable, or convex on an interval and concave on another one.
\end{rem}

\section{Determination of $R_{2/2,3}$\label{sectionIRKs2}}

The functions in $\Pi_{2/2,3}$ have the form
\[
\psi_a(z)=
\frac{\left(\frac{1}{3}+\frac{a }{2}\right) z^2+(a +1) z+1}{\left(-\frac{a }{2}-\frac{1}{6}\right)z^2+a   z+1},\]
with $a\in\mathbb{R}$. For any $a\in\mathbb{R}$, the fraction can not be simplified, since the resultant of the numerator and the denominator w.r.t.\ $z$ is $\frac{1}{12}\ne 0$. This implies that $\psi_a(\cdot)$ does not have removable singularities, only pole(s). This last statement remains valid if $\psi_a(\cdot)$ is replaced by its $k^{\mathrm{th}}$ derivative ($k\ge 1$), since differentiation cannot introduce new removable or pole singularities, nor annihilate an existing pole. According to Theorem \ref{vdgkCorollary3.4}, $R(\psi_a)>0$ implies that $\psi_a$ has a positive real pole. This last condition holds if and only if
$a \le -1-\frac{1}{\sqrt{3}}\approx-1.57735$
or
$a\ge -1+\frac{1}{\sqrt{3}}\approx-0.42264$.

\begin{figure}[h]
  \centering
  \includegraphics[width=5in]{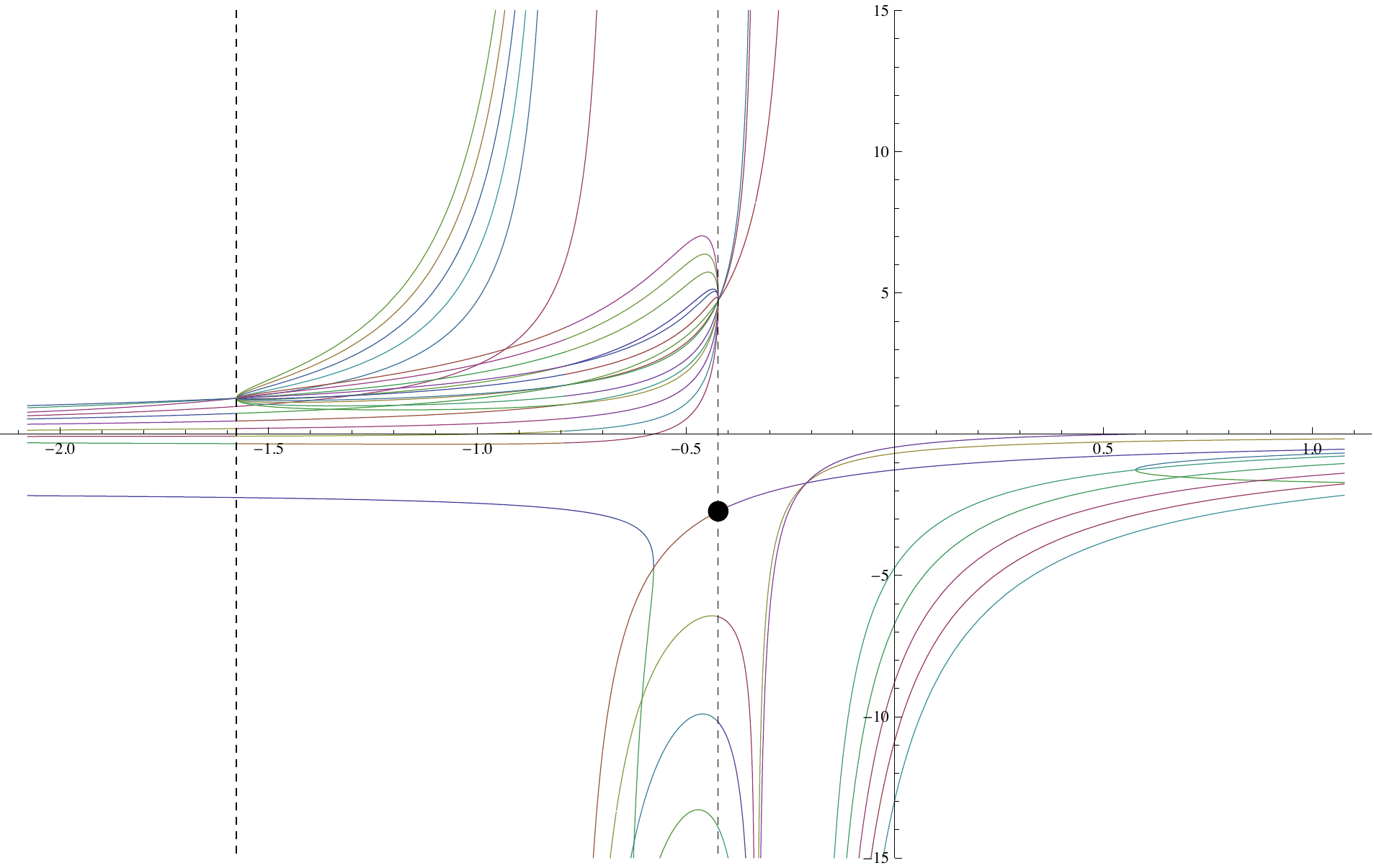}
  \caption{The figure shows roots of the derivatives $\psi_a^{(k)}(\cdot)$ as a
  function of $a$ for $0\le k\le 5$ for the family of functions in the $\Pi_{2/2,3}$ class.  
  The parameter region between the vertical dashed lines $\left(\mathit{i.e.,\
  }-1-\frac{1}{\sqrt{3}}<a<-1+\frac{1}{\sqrt{3}}\right)$ should be ignored,
  since here $R(\psi_a)=0$. The black dot indicates the optimal $a$ parameter
  value and (the negative of) the optimal radius of absolute monotonicity
  within this class.}\label{231figure}
\end{figure}

Figure \ref{231figure} suggests that the maximal radius of absolute
monotonicity within this class is $1+\sqrt{3}$, corresponding to parameter
value $a^*_{23}=-1+\frac{1}{\sqrt{3}}$ (when the rational function has a pole
of maximal possible order, \textit{i.e.}, 2). 

For $a \le -1-\frac{1}{\sqrt{3}}$, it is easily seen that $\psi_a$ has only
positive real poles.  Furthermore, we have $\psi_a(-2)=\frac{1}{1-12 a}>0$ and
$\psi_a\left(-\frac{5}{2}\right)=-\frac{15 a+14}{135 a+1}<0$, so by the intermediate value theorem $\psi_a(\cdot)$ has a root in $\left(-\frac{5}{2},-2\right)$, implying by Theorem \ref{vdgkLemma4.5} with $\ell=0$ that for these $a$ values $R(\psi_a)\le \frac{5}{2}< 1+\sqrt{3}.$

On the other hand, if $a>-1+\frac{5}{3 \sqrt{3}}>-1+\frac{1}{\sqrt{3}}$, then $\psi_a$ has a pole in $(-1-\sqrt{3},0)$, meaning that $R(\psi_a)< 1+\sqrt{3}$ here, since $\psi_a$ is not defined on the whole $(-1-\sqrt{3},0]$. Now we are going to exclude the parameter region $a\in \left( -1+\frac{1}{\sqrt{3}},-1+\frac{5}{3 \sqrt{3}}\right]$. If $a=-1+\frac{5}{3 \sqrt{3}}$, then $R(\psi_a)\le 2<1+\sqrt{3}$, because $\psi^\prime_a(-2)<0$. 

Let $a\in \left( -1+\frac{1}{\sqrt{3}},-1+\frac{5}{3 \sqrt{3}}\right)$ be fixed.  Then
$\psi^\prime_a(-1-\sqrt{3})<0<\psi^\prime_a(0)$,  so $\psi^\prime_a$ has a root in $(-1-\sqrt{3},0)$ and then---by Theorem \ref{vdgkLemma4.5} with $\ell=1$---inequality $R(\psi_a)< 1+\sqrt{3}$ holds.

Finally, we see that if $a=a^*_{23}=-1+\frac{1}{\sqrt{3}}$, then
\[
\psi_{a^*_{23}}(x)=\frac{\left(2+\sqrt{3}\right) \left(\left(\sqrt{3}-1\right) x^2+2 \sqrt{3}\ x+6\right)}{\left(\sqrt{3}+3-x\right)^2},
\]
from which one can prove recursively for $k\ge 1$ that 
\[
\psi_{a^*_{23}}^{(k)}(x)=\frac{6 k! \left((12 k-3)+\sqrt{3} (7 k-2)+\left(3+2 \sqrt{3}\right) x\right)}{\left(\sqrt{3}+3-x\right)^{k+2} }.
\]
With the help of these formulae, one verifies that
$\psi_{a^*_{23}}^{(k)}(x)\ge 0$ for all $k\ge 0$ and $-1-\sqrt{3}\le x \le 0$,
proving our claim.

\section{Determination of $R_{3/3,p}$ for $5\le p\le 6$ and a lower bound when $p=2$}

\subsection{Determination of $R_{3/3,6}$\label{IRKs3p6section}}
The
set $\Pi_{3/3,6}$ consists of the unique member
$
\psi_{36}(z)=\frac{\frac{z^3}{120}+\frac{z^2}{10}+\frac{z}{2}+1}{-\frac{z^3}{120}+\frac{z^2}{10}-\frac{z}{2}+1}.
$
This Pad\'e approximation of the exponential function has radius of absolute monotonicity 
\[
R(\psi_{36})=\text{root}_1(1,-6,0,40)=
-2+\sqrt[3]{4} \left( \sqrt[3]{3-\sqrt{5}}+ \sqrt[3]{3+\sqrt{5}}\right)\approx 2.207606.
\]
To prove it, we first determine $B(\psi_{36})$ (see Definition \ref{Bdef}). Let
$\alpha_0:=\text{root}_1(1,-12,60,-120)\approx 4.64437$ and
$\alpha_1:=\text{root}_2(1,-12,60,-120)\approx 3.67781-\ii\cdot 3.50876$ denote
the real and one of the complex roots of the denominator of $\psi_{36}$.  Then
$x=x^*:=-B(\psi_{36})<0$ is the unique real solution of the equation
\[
|\alpha_0-x|=|\alpha_1-x|.
\]
From this we see that $B(\psi_{36})=\text{root}_1(1,-6,0,40)$.  From Theorem \ref{vdgkTheorem3.3} we know that $R(\psi_{36})\le B(\psi_{36})$.  In order to show that $R(\psi_{36})=B(\psi_{36})$, we verify that $\psi_{36}^{(k)}(x^*)\ge 0$ for any $k\ge 0$. 
(Then Theorem \ref{absmon_int} with $x=x^*$ guarantees absolute monotonicity on $[-B(\psi_{36}),0]$, since the denominator of $\psi_{36}$ does not vanish on, say, $(-\infty,0]$.) The partial fraction decomposition of $\psi_{36}$ is
$
\psi_{36}(x)=
-1+\frac{c_0}{\alpha_0-x}+\frac{c_1}{\alpha_1-x}+\frac{\overline{c_1}}{\overline{\alpha_1}-x},
$
with $c_0:=\text{root}_1(1,-24,-1200,-40000)\approx 57.2025$ and $c_1:=\text{root}_3(1,-24,-1200,-40000)\approx -16.6012+\ii\cdot 20.5831$. But
$\psi_{36}(x^*)=\text{root}_1(1,-15,2127,-233)>0$, and for $k\ge 1$ we have
\[
\psi_{36}^{(k)}(x^*)=k! (\alpha_0-x^*)^{-k-1}\left( c_0+c_1\left(\frac{\alpha_0-x^*}{\alpha_1-x^*}\right)^{k+1}+
\overline{c_1}\left(\frac{\alpha_0-x^*}{\overline{\alpha_1}-x^*}\right)^{k+1}\right).
\]
Since $\left|\frac{\alpha_0-x^*}{\alpha_1-x^*}\right|=\left|\frac{\alpha_0-x^*}{\overline{\alpha_1}-x^*}\right|=1$ by construction and $52.8874\approx |c_1|+|\overline{c_1}|<c_0$, the positivity of $\psi_{36}^{(k)}(x^*)$ follows.

\subsection{Determination of $R_{3/3,5}$\label{sectionIRKs3p5}}
The set $\Pi_{3/3,5}$ can be written in terms of one real parameter $a\in\mathbb{R}$ as
\[
\psi_a(z)=\frac{\left(a+\frac{1}{60}\right) z^3+\left(6 a+\frac{3}{20}\right) z^2+\left(12 a+\frac{3}{5}\right) z+1}{a z^3+\left(\frac{1}{20}-6 a\right) z^2+\left(12 a-\frac{2}{5}\right) z+1}.
\]
The numerator and  denominator do not have a common root for any $a$, because their resultant w.r.t.\ $z$ is $\frac{1}{8640}\ne 0$.

Theorem \ref{vdgkCorollary3.4} tells us that if $R(\psi_a)>0$, then $\psi_a$ has at least one positive real pole.  We claim that the statement $\exists\  x>0\ :\ a x^3+\left(\frac{1}{20}-6 a\right) x^2+\left(12 a-\frac{2}{5}\right) x+1=0$ implies $a<0$. Indeed, after a rearrangement we get
\begin{equation}\label{a_s3p5}
a=\frac{-x^2+8 x-20}{20 x \left(x^2-6 x+12\right)}<0,
\end{equation}
since the numerator is negative and the denominator is positive for all $x>0$.

Next we prove via the intermediate value theorem that for $a\le -\frac{1}{50}$, the $10^{\mathrm{th}}$ derivative of $\psi_a(\cdot)$ has a root between $-2$ and $0$:  indeed, for $a\le-\frac{1}{50}$
\[
\psi_a^{(10)}(-2)<0<\psi_a^{(10)}(0).
\]
The formal proof of this last statement is completely analogous to that of the
corresponding statement in the $\Pi_{4/4,7}$ class in the opening
paragraphs of Section \ref{sectionIRKs4p7},  hence we omit the details here. The
only thing we have to verify is that the intermediate value theorem is
applicable in $[-2,0]$. On one hand, the poles of $\psi_a^{(10)}(\cdot)$
coincide with those of $\psi_a(\cdot)$ (only their order can be different). On
the other hand, for any fixed $a<0$, the function $\psi_a(\cdot)$ does not have
any poles in $[-2,0]$ (nor in $(-\infty,0]$), because for
$x\in\mathbb{R}\setminus\{0\}$, inequality (\ref{a_s3p5}) implies $x>0$. This
means that for $a\le-\frac{1}{50}$ we have $R(\psi_a)\le 2$
due to Theorem \ref{vdgkLemma4.5} with $\ell=10$. 

So in order to get $R(\psi_a)>2$, it is enough to consider $a\in \left(-\frac{1}{50},0\right)$. As it turns out, the maximal radius of absolute monotonicity for these rational functions will be determined by the maximum of the function $a\mapsto B(a)\equiv B(\psi_a)$  according to Definition \ref{Bdef} and Theorem \ref{vdgkTheorem3.3}.

First we give the real and imaginary parts of the three poles (even for all $a<0$), but due to symmetry, it is enough to take into account only the closed upper half-plane. If $a<0$, then the real pole is located at \[\vr_{351}(a)=\text{root}_1(20 a,1-120 a,240 a-8,20).\] As for the other pole, first 
we set \[\vr_{352}=\text{root}_1(864000,43200,360,1)=-\frac{1}{120} \left(2+\sqrt[3]{2}+\sqrt[3]{4}\right)\approx -0.0403944.\] Then we consider the denominator of $\psi_a$, substitute $z=x+\ii y$, separate the real and imaginary parts and solve the resulting system: if $\vr_{352}\ne a<0$, then the upper complex pole is found at 
\[
\vr_{353}(a)+
\frac{\ii}{4 \sqrt{5 |a|}}\sqrt{\frac{14400 a^2+1200 a+4-(240 a+1 )\vr _{353}(a)}{60 a (2-\vr _{353}(a))-1}}
\]
with \[\vr_{353}(a)=\text{root}_1(1600 a^2,-9600 a^2+80 a,19200 a^2-400 a+1,-14400 a^2+400 a-4),\] while for $a=\vr_{352}$, the upper complex pole is located at
$4-\sqrt[3]{4}+\ii \left(2 \sqrt[3]{2}+\sqrt[3]{4}-2\right) \approx 2.4126+\ii \cdot 2.10724$. A case separation is necessary in the above formula to avoid a  $0/0$ under the square root:  geometrically this $a=\vr_{352}$ value corresponds to the case when the three poles have the same real part.

\begin{figure}[h]
  \centering
  \includegraphics[width=3in]{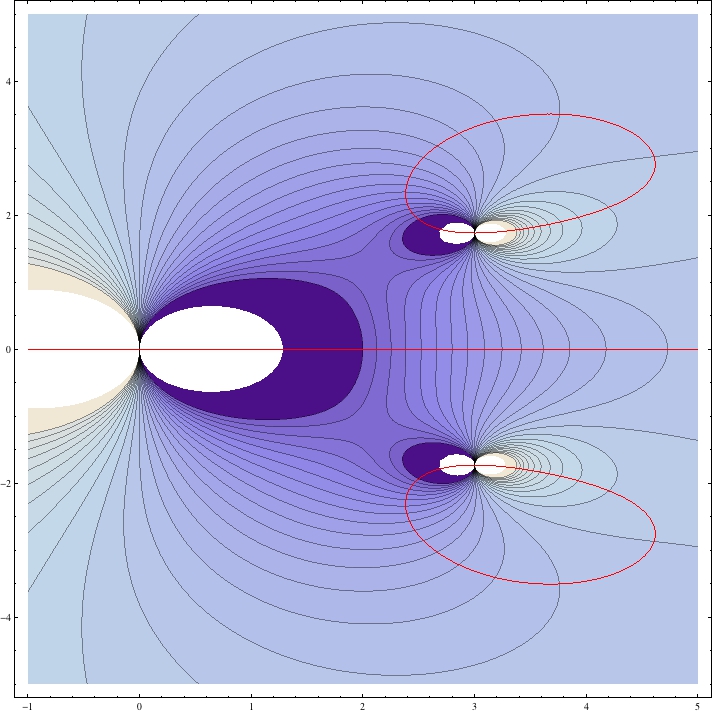}
  \caption{By substituting $z=x+\ii y$ ($x, y\in\mathbb{R}$) into the fraction in (\ref{a_s3p5}) and separating  real and imaginary parts, two $\mathbb{R}^2\to\mathbb{R}$ surfaces ($S_{\text{re}}$ and $S_{\text{im}}$) are obtained. The contour $S_{\text{im}}=0$ yields the equation $y (x^4-16 x^3+2 x^2 y^2+$ $96 x^2-16 x y^2-240 x+y^4+16 y^2+240)=0$ depicted as red curves (with three connected components). They describe the locus of the poles of $\psi_a$ for $a\in\mathbb{R}$ (including the limiting case $a=\pm\infty$) on the complex plane in the $\Pi_{3/3,5}$ class. For a particular, fixed $a\in\mathbb{R}$, the three poles of $\psi_a$ are found when the red curves intersect the contour curves $S_{\text{re}}=a$  (some of which are shown as blue curves). The poles in the $a=\pm\infty$ limiting case are located near the three singular points (squeezed in between the white regions), where $S_{\text{re}}$ and $S_{\text{im}}$ are undefined.}\label{351figure}
\end{figure}

Now with the real and imaginary parts separated, it is easy to find, for any $a<0$ the unique point on the real axis being equidistant from the real pole and the upper complex pole. We check that for $a<\vr_{352}$ the (positive) real pole is strictly smaller than the real part of the (upper) complex pole, hence---by applying the notations of Definition \ref{Bdef}  with $\alpha_0\equiv \alpha_0(a)$ being the positive real pole---the equidistant point on the real axis is strictly positive, meaning that $0\in I(\alpha_0(a))$ and $B(a)=-\inf I(\alpha_0(a))=+\infty$ here. We easily see that $B(\vr_{352})=+\infty$ also, but for $\vr_{352}<a<0$, $B(a)$ will be finite. If $\vr_{352}<a<0$, then the point on the real axis equidistant from the real pole and the complex pole is found at $\frac{f_{351}(a)}{f_{352}(a)}$ with $f_{351}(a):=$
\[
\vr^2_{351}(a) \left(60 a \vr _{353}(a)-120a+1\right)- 80 a \vr^3_{353}(a)-(2-240 a) \vr^2_{353}(a)- (240 a-8) \vr _{353}(a)-20
\]
and
\[
f_{352}(a):=2 \left(\vr _{351}(a)-\vr _{353}(a)\right) \left(60 a \vr _{353}(a)-120a+1\right).
\]
Now \Reduce\  establishes quickly that $\frac{f_{351}(a)}{f_{352}(a)}$ has a unique zero in $\vr_{352}<a<0$: by defining  \[\vr_{354}=\text{root}_1(13824000,345600,2880,-152,-1)\approx -0.00625485,\] it turns out that   \[\frac{f_{351}}{f_{352}}\Big|_{(\vr_{352},\vr_{354})}<0,\quad \frac{f_{351}}{f_{352}}\Big|_{a=\vr_{354}}=0,\quad \frac{f_{351}}{f_{352}}\Big|_{(\vr_{354},0)}>0.\]
(However, our "manual" computations to reproduce this sign property \textit{without} \Reduce\ occupy several pages in the corresponding \Mma notebook when written out fully, so we only sketch the main steps. From the equation $f_{351}(a)=0$ we simply express $\vr_{351}(a)$ as
$
\vr_{351}(a)=\pm\frac{\sqrt{2} \sqrt{40 a \vr^3_{353}(a)-\ldots}}{\sqrt{60 a \vr_{353}(a)-\ldots}},
$
where we have displayed (and will display) only some (typical) "highest order" terms to indicate some structure. This $\pm$ expression is then substituted back in place of $x$ into the defining polynomial equation $20 a x^3+(1-120 a) x^2+(240 a-8) x+20=0$ of the root object  $\vr_{351}(a)$. After clearing the denominators, we arrive at a relation of the form
\[
\mp(40 a (120 a-1) \vr^3_{353}(a)-\ldots)\sqrt{60 a \vr_{353}(a)-\ldots}=\]
\[\sqrt{32}\cdot (200 a^2 \vr^3_{353}(a)+\ldots)\sqrt{40 a \vr^3_{353}(a)-\ldots}
\]
with the square roots being the same as above in the $\pm$ expression. A squaring and rearranging cancel the $\mp$ symbol and the square roots, and yield a polynomial in $a$ and in $\vr_{353}(a)$ of the form $-51200000 a^5 \vr^9_{353}(a)+\ldots+80=0$, with $39$ terms omitted in between. Now we express $\vr^3_{353}(a)$ from the defining polynomial equation of this root object $\vr_{353}(a)$ as $\vr^3_{353}(a)=$
\[
\frac{1}{1600 a^2}\left( \left(9600 a^2-80 a\right) \vr^2_{353}(a)+\left(-19200 a^2+400 a-1\right)\vr_{353}(a)+14400 a^2-400 a+4 \right).
\]
This formula is applied recursively within the above equation $-51200000 a^5 \vr^9_{353}(a)+\ldots+80=0$ to eliminate all the powers $\vr^k_{353}(a)$ with $k=9, 8, \ldots, 3$ and to have 
\[
P_{353}(a)+P_{354}(a)\vr_{353}(a)+P_{355}(a)\vr^2_{353}(a)=0,
\]
where $P_{35\ell}(a)$ ($\ell=3,4,5$) are suitable polynomials in $a$ of degree 7 with integer coefficients having at most 14 digits. The above quadratic polynomial is now solved to get
\[
\vr_{353}(a)=\frac{-P_{354}(a)\pm\sqrt{P^2_{354}(a)-4P_{355}(a)  P_{353}(a)}}{2 P_{355}(a)},
\] 
which information is then substituted back into the very defining polynomial equation of the root object $\vr_{353}(a)$, and we can derive the equation
$
P_{356}(a)\pm P_{357}(a)\sqrt{P_{358}(a)}=0.
$ 
(Here $P_{356}(a)$ is a polynomial in $a$ of degree 18 with integer coefficients having at most 35 digits. The degree of the integer polynomial $P_{357}(a)$ is 10, while that of $P_{358}(a)$ is 14.) A final squaring and rearrangement allow a factorization of the form
\[
0=a \left(864000 a^3+43200 a^2+360 a+1\right) \left(13824000 a^4+345600 a^3+2880 a^2-152 a-1\right)\times\]\[ \left(17280000 a^4+1728000 a^3+57600 a^2+440 a+1\right)\times\]\[
   \left(248832000000 a^6+26956800000 a^5+639360000 a^4+14112000 a^3+34800 a^2-310 a-1\right).
\]
Now we see that only 3 out of the 18 roots of the above polynomial are located in the (real) interval $(\vr_{352},0)$, and a check (similar in spirit to the manipulations so far) yields that 2 out of these 3 are \textit{not} solutions to  the equation $f_{353}(a)=0$, just the third one. (Of course, we have checked during all the computations that the quantities appearing under the square roots are non-negative, further, that the possible vanishing leading coefficients of the intermediate linear or quadratic equations do not yield extra solutions to  $f_{353}(a)=0$.) Summarizing the above, we have proved independently of \Reduce\ that the equation $f_{351}(a)=0$ has a unique solution $\vr_{354}$ in $(\vr_{352},0)$, and a continuity argument yields also the signs of the fraction $\frac{f_{351}(a)}{f_{352}(a)}$ on the subintervals.)

Now let us proceed further. Knowing the sign of $\frac{f_{351}(a)}{f_{352}(a)}$ implies that on $a\in(\vr_{352},0)\setminus\{\vr_{354}\}$ relation $0\in I(\alpha_0(a))$ can hold only for $a\in (\vr_{352},\vr_{354})$, so 
\[B(a)=0,\quad \mathrm{if\ } a\in [\vr_{354},0),\]
 and 
\[B(a)=-\inf I(\alpha_0(a))=-\frac{f_{351}(a)}{f_{352}(a)},\quad \mathrm{if\ } a\in (\vr_{352},\vr_{354}).\]

\begin{figure}[h]
  \centering
  \includegraphics[width=5in]{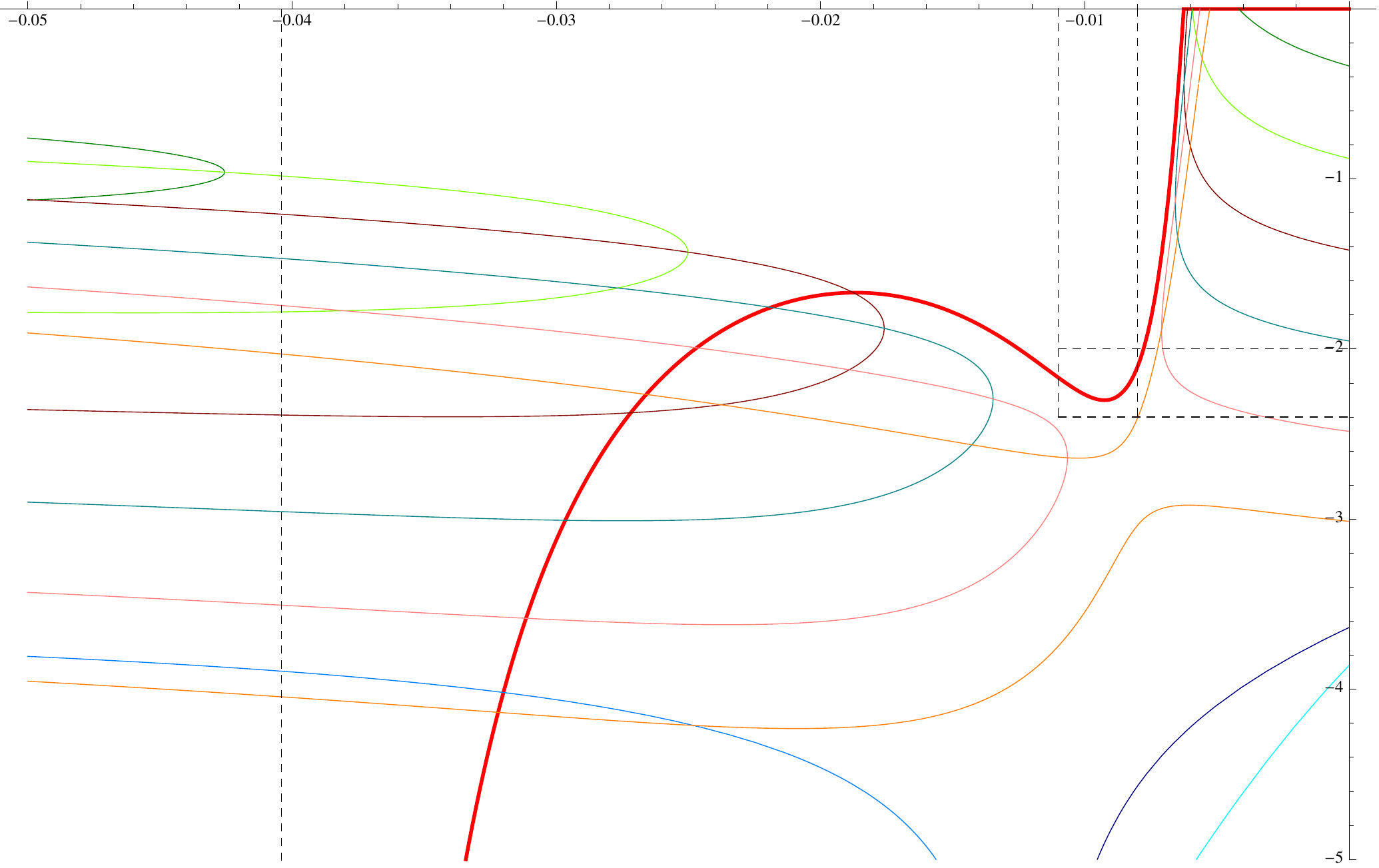}
  \caption{In the $\Pi_{3/3,5}$ class the function $a\mapsto -B(a)$ is shown as the thick red curve (with  $-B\equiv -\infty$ to the left of the long, vertical dashed line). The small dashed rectangle will contain the optimal value of $-B$. The other curves are the roots of the derivatives $\psi^{(k)}_a(\cdot)$ for $k=0, 1, 2, 8, 9, 10, 11, 12, 13$. In this plot window, these are the only roots of $\psi^{(k)}_a(\cdot)$ for $k\le 13$. We remark that altogether 17947 digits are needed just to write down the integer polynomials appearing in the numerators of $\psi^{(k)}_a(\cdot)$  for these $k$ values.  The structures analogous to pitchfork bifurcations in this figure nicely illustrate Rolle's theorem, further, the fact that if a smooth function has a root of multiplicity $m$ at a point, then its derivative also has a root there (with multiplicity $m-1$).}\label{352figure}
\end{figure}

Having completely described the function $B$, in the next step we will determine its maximal value  on $a\in \left(-\frac{1}{50},0\right)$ by differentiation. We check by computing the discriminant of the denominator of $\psi_a$ that all roots of the denominator are simple for all $a<0$. Moreover, this denominator is smooth in $x$ and $a$, so the implicit function theorem yields that the functions $a\mapsto \vr_{351}(a)$ and $a\mapsto \vr_{353}(a)$ are differentiable with derivatives
\[
\vr^{\prime}_{351}(a)=-\frac{20 \vr^3_{351}(a)-120 \vr^2_{351}(a)+240 \vr _{351}(a)}{60 a \vr^2_{351}(a)+2 (1-120 a) \vr _{351}(a)+240 a-8}
\]
and
\[
\vr^{\prime}_{353}(a)=-\frac{3200 a \vr^3_{353}(a)+(-19200 a+80) \vr^2_{353}(a)+(38400 a-400) \vr_{353}(a)-28800 a+400}{4800 a^2 \vr^2_{353}(a)+2 \left(-9600 a^2+80 a\right) \vr _{353}(a)+19200 a^2-400 a+1}.
\]
Using also these formulae we determine that there are exactly two $a$ values in $\left(-\frac{1}{50},0\right)$ such that $B^{\prime}(a)=0$, one is a local minimum and the other one is a local maximum of $B$. But since we are maximizing $B$, the local minimum is ignored. The local maximum of $B$ is located at 
\[
a^*_{35}:=
\]
\[
\text{root}_2(13436928000000,1492992000000,68428800000,1427328000,13867200,61200,101) \approx\]
\[
-0.009257142292762484937472363
\]
with value 
\[
B(a^*_{35})=\text{root}_2(1,12,36,-76,-360,0,900)\approx 2.301322934003485801187482.
\]
Note that Figure \ref{352figure} depicts $-B$ so that it can be compared to the derivatives of $\psi_a(\cdot)$. We remark that the local maximum of $-B\big|_{\left(-\frac{1}{50},0\right)}$ is described by the same root object as $a_{35}^*$ but with $\mathrm{root}_2$ replaced by $\mathrm{root}_1$. \Reduce\  was able to solve the equation $B^{\prime}(a)=0$ within 2 seconds. When we solved the same equation "manually" (\textit{i.e.},\ in a \Reduce-independent way) with all the details written out fully, the corresponding \Mma notebook saved as a PDF file consisted of approx.\ 280 pages, and the intermediate expressions could be described by approximately 700\,000 characters. The final integer polynomial with both root objects $\vr _{351}(a)$ and $\vr _{353}(a)$ eliminated has degree $162$ and leading coefficient approx.\ $4\cdot 10^{311}$. In view of the above, it is very surprising that the algebraic number $a_{35}^*$ is so simple: the reason is that the final polynomial of degree 162 can be written as $-64000 a^4 P_{359}^3(a) P_{3510}^3(a) P_{3511}^{12}(a) P_{3512}^2(a) P_{3513}(a)$ with the polynomials $P_{35\ell}$ ($\ell=9,10,\ldots, 13$) having degrees 10, 26, 3,  4 and 6, respectively, further, $P_{3513}(a_{35}^*)=0$.

As a final step, we show that the interval $[-B\left(a_{35}^*\right),0]$ does not contain any roots of the derivatives $\psi^{(k)}_{a_{35}^*}(\cdot)$ ($k=0,1, \ldots $). In view of  Theorem \ref{absmon_int}  with  $x=-B\left(a_{35}^*\right)$, it is enough to show that for all $0\le k\in\mathbb{N}$ we have
\[
\psi_{a_{35}^*}^{(k)}\left(-B\left(a_{35}^*\right) \right)\ge 0,
\]
which we check by partial fraction decomposition (also showing that there are no poles of $\psi_{a_{35}^*}$ in $[-B\left(a_{35}^*\right),0]$, so Theorem \ref{absmon_int} applies).
The function  $\psi_{a_{35}^*}$ evaluated at $x$ admits the following decomposition:
\[
c+\frac{c_0}{\alpha_0-x}+\frac{c_1}{\alpha_1-x}+\frac{\overline{c_1}}{\overline{\alpha_1}-x}
\]
with some $c, c_0\in\mathbb{R}$, $\alpha_0 >0$ and $c_1, \alpha_1\in\mathbb{C}\setminus\mathbb{R}$. These constants are algebraic numbers of degree not exceeding 12, and absolute value of the  coefficients of their defining integer polynomials less that $10^{26}$. Here we give only their numerical approximations as
\[
c\approx -0.800411, \quad c_0\approx 46.829419, \quad \alpha_0\approx 4.449434, 
\]
\[
c_1\approx -13.750016- i\cdot 16.640148, \quad \alpha_1\approx 3.475900+i\cdot 3.492336.
\]
With $x=x^*:=-B\left(a_{35}^*\right)<0$ we check that 
\[
0.09999997\approx \psi_{a_{35}^*}(x^*) >0.
\]
On the other hand, if $k\ge 1$, then $\psi_{a_{35}^*}^{(k)}(x^*)$ takes the form 
\[
k! (\alpha_0-x^*)^{-k-1}\left( c_0+c_1\left(\frac{\alpha_0-x^*}{\alpha_1-x^*}\right)^{k+1}+
\overline{c_1}\left(\frac{\alpha_0-x^*}{\overline{\alpha_1}-x^*}\right)^{k+1}\right).
\]
Now we make use of the facts that $\alpha_0-x^*>0$ and $\left|\frac{\alpha_0-x^*}{\alpha_1-x^*}\right|=\left|\frac{\alpha_0-x^*}{\overline{\alpha_1}-x^*}\right|=1$ by construction, further, that $c_0>0$ and $43.172096\approx |c_1|+|\overline{c_1}|<c_0$. This proves that $\psi_{a_{35}^*}^{(k)}(x^*)>0$ also for $k\ge 1$.

Summarizing the above, the optimal radius of absolute monotonicity within the $\Pi_{3/3,5}$ class is $B\left(a_{35}^*\right)\approx 2.301322$.

\subsection{A counterexample to the conjecture that $R_{s/s,2}=2s$\label{counterexamplesection}}
A counterexample to Conjecture \ref{conj:vdgk}
for the case $m=n=3$ was found via extensive numerical search using MATLAB and the Symbolic Toolbox
employing the {\texttt{fminsearch}} function (Nelder--Mead simplex method); it
originally appeared in \cite{dkthesis}. For the sake of our presentation, let
us denote this counterexample with rational coefficients by
$\widetilde{\psi}_{32}$.\footnote{We note that there is a typo in the last
digit of one of its coefficients in \cite{dkthesis}, instead of
$7969150767159903$, one should have $7969150767159904$.} Now we are going to
present a simpler counterexample---in the sense that the rational coefficients
have much smaller numerators and denominators---with slightly improved radius
of absolute monotonicity.

By suitably perturbing the rational coefficients into nearby simpler ones and embedding $\widetilde{\psi}_{32}$ 
into a one-parameter family of rational functions in a way that the order conditions are satisfied within the family, say, as 
\[\psi_c(z)=\frac{c z^3+\frac{2289}{34970}z^2+\frac{119}{269}z+1}{-\frac{4}{327}z^3+\frac{8}{65}z^2-\frac{150}{269}z+1}\] with $c\in\mathbb{R}$, we can optimize the radius of absolute monotonicity w.r.t.\ $c$ using \textit{Mathematica}.  It turns out that the optimal parameter $c^*$ (\textit{i.e.}, the one that yields the maximal $R$ within this chosen class)  is an algebraic number of degree 5. By replacing 
$c^*$ with a nearby simple rational number we get, for example, 
\[
\psi_{32}(z) = \frac{\frac{1246}{384649}z^3+\frac{2289}{34970}z^2+\frac{119}{269}z+1}{-\frac{4}{327}z^3+\frac{8}{65}z^2-\frac{150}{269}z+1}.
\]
This function has \[R(\psi_{32})=\text{root}_1(43572620,-880461561,5950520030,-13451175530)\approx\]\[
6.778307398562974637718719>6.\]
We remark that $6.77823595\approx R(\widetilde{\psi}_{32})<R(\psi_{32})<R(\psi_{c^*})\approx 6.77830907$.

Let us briefly give some details that can be used to verify the above value of $R(\psi_{32})$. The function $\psi_{32}$ at $x$ has the following partial fraction decomposition
\[
c+\frac{c_0}{\alpha_0-x}+\frac{c_1}{\alpha_1-x}+\frac{\overline{c_1}}{\overline{\alpha_1}-x}
\]
with $c=-\frac{203721}{769298}$, $c_0\approx 24.8122$, $\alpha_0\approx 3.71417$, $c_1\approx -8.39838- i\cdot 9.53528$ and $\alpha_1\approx 3.17368+i\cdot 3.45514$.
 (These are all algebraic numbers of degree 3, and with maximal absolute value of the coefficients in their defining integer polynomials approximately $1.046\cdot 10^{42}$.) We see that $\psi_{32}$ has a unique real root at $x^*:=-R(\psi_{32})\approx -6.778307$, so Theorem \ref{vdgkLemma4.5} with $x=x^*$ and  $\ell=0$ applies.  We claim that $\psi_{32}^{(k)}(x^*)\ge 0$ holds for all $k\ge 1$ as well. If $k\ge 1$, then 
\[
\psi_{32}^{(k)}(x^*)=k! (\alpha_0-x^*)^{-k-1}\left( c_0+c_1\left(\frac{\alpha_0-x^*}{\alpha_1-x^*}\right)^{k+1}+
\overline{c_1}\left(\frac{\alpha_0-x^*}{\overline{\alpha_1}-x^*}\right)^{k+1}\right).
\]
Now since $\alpha_0-x^*>0$ and $\left|\frac{\alpha_0-x^*}{\alpha_1-x^*}\right|=\left|\frac{\alpha_0-x^*}{\overline{\alpha_1}-x^*}\right|\approx 0.995991$ (an algebraic number of degree 18, with 92-digit integers as coefficients), we see that the sufficient condition 
\[
2|c_1|\left|\frac{\alpha_0-x^*}{\alpha_1-x^*}\right|^{k+1}\le c_0
\]
for the non-negativity of $\psi_{32}^{(k)}(x^*)$ holds for $k\ge 5$. Finally, we directly check that $\psi_{32}^{(k)}(x^*)>0$ is also valid for $1\le k\le 4$. (We remark that now $B(\psi_{32})\approx 7.59982$ is an algebraic number of degree 3, so for this particular function $R(\psi_{32})<B(\psi_{32})$.) Absolute monotonicity in the whole $[x^*,0]$ interval is guaranteed by Theorem \ref{absmon_int}, by taking into account that there are no poles of $\psi_{32}$ in $[x^*,0]$.

Of course, it is to be emphasized that the maximal radius of absolute monotonicity in the \textit{whole} 
$\Pi_{3/3,2}$ class is still unknown: this class of rational functions can be described by 4 parameters. 
It is a major open challenge to find the maximal $R$ within this 4-parameter family.

\section{Determination of $R_{4/4,7}$}\label{sectionIRKs4p7}

The set $\Pi_{4/4,7}$ can be described by one real parameter $a\in\mathbb{R}$ as follows
\[
\psi_a(z)=\frac{\frac{1}{840} (7 a+4) z^4+\frac{1}{210} (21 a+13) z^3+\left(-\frac{1}{14} (7 a+2)+a+\frac{1}{2}\right) z^2+(a+1) z+1}{-\frac{1}{840} (7 a+3) z^4+\frac{1}{210} (21 a+8)
   z^3-\frac{1}{14} (7 a+2) z^2+a z+1}.
\]
The numerator and  denominator do not have common factors for any $a$, since their resultant  w.r.t.\ $z$ is $\frac{1}{870912000}\ne 0$.
In order to have $R(\psi_a)>0$, $\psi_a$ needs to have at least one positive real pole (Theorem \ref{vdgkCorollary3.4}). To obtain some preliminary information on the location of the real poles of $x\mapsto\psi_a(x)$, we apply the same trick as at the beginning of Section \ref{sectionIRKs3p5}. This time the denominator of $\psi_a(x)$ (with $x\in\mathbb{R}$) vanishes if and only if 
\[
a=\frac{-3 x^4+32 x^3-120 x^2+840}{7 x \left(x^3-12 x^2+60 x-120\right)}.
\]
Then, by analyzing the range of this rational function on the right-hand side (under the restriction $x>0$), we prove that $\psi_a$ has at least one positive real pole precisely if  
\[
a\le \vr_{471}\approx -0.843194\quad\mathrm{or}\quad a\ge \vr_{472}\approx -0.471357,
\] 
so in the rest of this section we can consider only the above parameter set. We remark that the exact values of these and the following algebraic numbers are listed in the Appendix. Furthermore, with this technique we also establish that for any $a<0$ the denominator of $\psi_a(\cdot)$ does not vanish in $[-1,0]$.

Now we show that in the left unbounded component, that is for $a\in (-\infty,\vr_{471}]$, we have 
$
R(a)\equiv R(\psi_a)\le 1.
$
To this end, let us consider the 9$^\mathrm{th}$ derivative  of $x\mapsto\psi_a(x)$, which can be written as
\[
 P_{47}(x,a)\cdot
1451520\left((7 a+3) x^4+(-84 a-32) x^3+(420 a+120) x^2-840 a
   x-840\right)^{-10},
\]
where  $P_{47}(x,a)$ is a polynomial in $x$ (of degree 30) and $a$ (of degree 10), and 5917 digits are needed to write down all of its integer coefficients. Since, as we have just seen in the previous paragraph, the $(-10)^{\mathrm{th}}$ power is never singular for $a\le \vr_{471}<0$ and $x\in [-1,0]$, we can  apply the intermediate value theorem in the interval $[-1,0]$. On one hand, \[P_{47}(0,a)=227219566849228800000000-433782809439436800000000 a^2<0\] for 
$a\in (-\infty,\vr_{471}]$. On the other hand
$
P_{47}(-1,a)=P_{471}(a)$, a polynomial of degree 10 with integer coefficients. But $P_{471}^{(10)}$ is a positive constant, and for 
$k=9, 8, \ldots , 0$ one computes that $\mathrm{sgn}\left(P_{471}^{(k)}\left(-\frac{1}{2}\right)\right)=(-1)^k$, showing inductively that for any $ a\in(-\infty,-\frac{1}{2}]$ and $k=9, 8, \ldots, 0$ we have $\mathrm{sgn}\left(P_{471}^{(k)}\left(a\right)\right)=(-1)^k$ (notice that we have used only rational arithmetic). Taking into account  $ (-\infty,-\frac{1}{2}]\supset (-\infty,\vr_{471}]$, we get that $P_{471}\big|_{(-\infty,\vr_{471}]}>0$. The intermediate value theorem now guarantees that for each $a\in (-\infty,\vr_{471}]$, the function $\psi_a^{(9)}(\cdot)$ has a root in $[-1,0]$. Due to Theorem \ref{vdgkLemma4.5} with $\ell=9$, we have $R(\psi_a)\le 1$
for $a\in (-\infty,\vr_{471}]$.\\

Next we focus our attention on the right unbounded component $a\in [\vr_{472},+\infty)$. In this region we are going to explicitly compute the function $a\mapsto B(a)\equiv B(\psi_a)$ (according to Definition \ref{Bdef}).

As a first step, 
\Reduce\  was able to give a complete description of the locus of the poles in the complex plane with real and imaginary parts separated as $a$ is varied. If $a=\vr_{472}$, we have a real pole of order 2 at $\vr_{473}\approx 7.64527$.  For
$\vr_{472}< a<-\frac{3}{7}$, it bifurcates into two real poles (of order 1) located at
$
\vr_{4711}(a)\in (\vr_{474},\vr_{473}),
$
and 
$
\vr_{4712}(a)\in (\vr_{473},+\infty),
$
with $\vr_{474}\approx 5.64849$. If  $a=-\frac{3}{7}$, then there is a unique real pole of order 1 at
$\vr_{474}$. Finally, for  $a>-\frac{3}{7}$ the smaller (negative) real pole is located at
$
\vr_{4711}(a)<0,
$
while the larger (positive) pole is found at 
$
\vr_{4712}(a)
$. 

As for the complex poles, due to symmetry, it is enough to describe only the one with imaginary part $>0$. For $a\in \left[\vr_{472},-\frac{3}{7}\right]$, the complex pole in the upper half-plane has real part $\vr_{477}(a)$ and imaginary part 
\[
\sqrt{\frac{(7 a +3)\vr_{477}^3(a)-(63 a+24) \vr_{477}^2(a)+(210 a +60)\vr_{477}(a)-210 a}{(7 a+3) \vr_{477}(a)-21 a-8}}.
\]
If $a> -\frac{3}{7}$, then the expressions for the real and the imaginary parts of the upper complex pole are obtained as above, but with changing $\text{root}_1$ in the definition of $\vr_{477}(a)$ (see the Appendix) to $\text{root}_2$.

Figure \ref{471figure} depicts the locus of the poles of $\psi_a$ as $a$ is varied. For the sake of completeness, we have also included the poles corresponding to the interval $a\in (-\infty,\vr_{472})$.

\begin{figure}[h]
  \centering
  \includegraphics[width=5in]{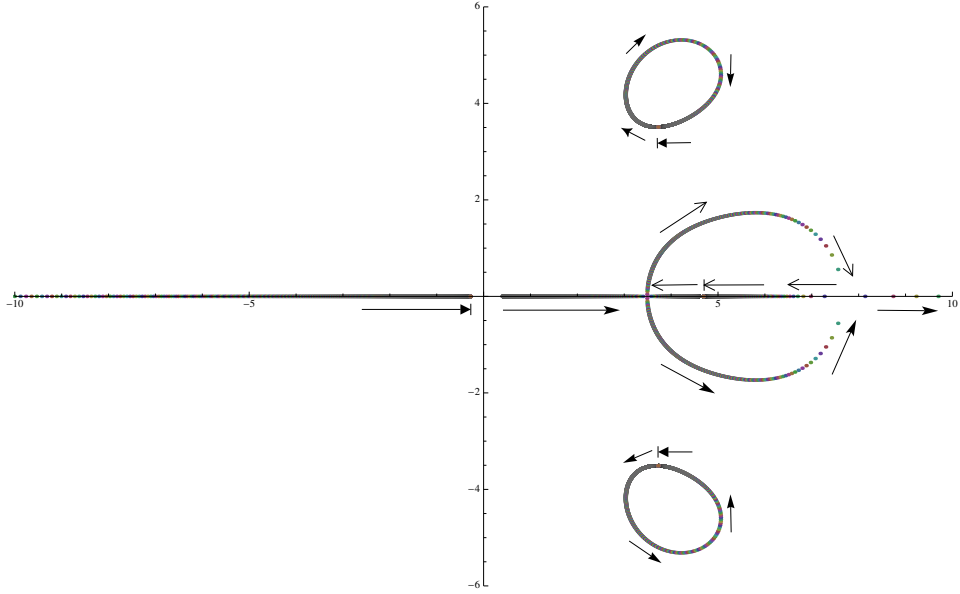}
  \caption{The trajectory of the four poles of $\psi_a$ in the complex plane in the $\Pi_{4/4,7}$ class as parameter $a$ traverses through the real line from $-\infty$ to $+\infty$. As $a\to -3/7^{-}$ and passes through this point, the larger real pole is repelled to $+\infty$, then it enters the real line again from $-\infty$ and approaches to $0^{-}$. (In order to create the figure, we have used an equidistant grid on $a\in[-3,3]$ and \textit{Mathematica}'s \NSolve\  approximated the roots of the resulting 6001 quartic equations appearing as denominators in a few seconds.) Compare these curves with the corresponding red curves in Figure \ref{351figure}.}\label{471figure}
\end{figure}

Now we combine the above formulae on the locus of the poles, real and imaginary parts separated, to determine $B(a)$ for each $a\in [\vr_{472},+\infty)$. With the help of \Reduce\  (and by applying some geometric reformulations to be able to obtain the results in a reasonable amount of computing time, but now omitting the details here) we can prove that for $-0.471357\approx \vr_{472}\le a <\vr_{478}\approx -0.469514$ the point on the real axis equidistant from the smaller (positive) real pole and the upper complex pole is positive (\textit{i.e.}, $0$ lies strictly closer to the non-real upper complex pole for $a\in [\vr_{472},\vr_{478})$), hence (in the sense of Definition \ref{Bdef}) 
$0\notin I(\alpha_0)$, so $B(a)=0$ here. 
The larger real pole---existing for $a\in \left(\vr_{472},-\frac{3}{7}\right)$--- clearly cannot influence the value of $B(a)$ now. For $a\in \left[\vr_{478},-\frac{3}{7}\right]$, the equidistant point on the real axis from the smaller positive real pole and the upper complex pole is non-positive, and is given by $\frac{f_{471}(a)}{f_{472}(a)}$ with 
\[
f_{471}(a):=2 (7 a+3) \vr _{477}^3(a)-4 (21 a+8) \vr^2 _{477}(a)+30 (7 a+2) \vr _{477}(a)+\]\[\left(21 a+8-(7 a+3) \vr _{477}(a)\right) \vr^2 _{4711}(a)-210 a
\]
and 
\[
f_{472}(a):=2 \left((7 a+3) \vr
   _{477}(a)-21 a-8\right) \left(\vr _{477}(a)-\vr _{4711}(a)\right).
\] 
(Of course, as already noted earlier, the expression "smaller positive real pole" in the previous sentence should be interpreted as the "unique real pole", if $a=-\frac{3}{7}$. Moreover, the expression "non-positive" can be replaced by "zero" if $a=\vr_{478}$, and "negative" if $a\in\left(\vr_{478},-\frac{3}{7}\right]$.) Finally, we study the interval $a>-\frac{3}{7}$. Now one should take into account the newly created negative real pole as well.  It can be proved that if $a=\vr_{479}\approx -0.358565$, then there is a unique negative real number, $\text{root}_1(1,-21,165,-520,0,3600,-6000)\approx -2.40614$ such that it is equidistant from all the four poles. If $a\in\left(-\frac{3}{7},\vr_{479}\right]$, then the negative real pole is still too far on the left to have an effect on $B(a)$, so for these $a$ values $B(a)$ is obtained as the absolute value of the
point on the real axis equidistant from the positive real pole and the upper complex pole. Here $B(a)$ will be given in terms of   
$
f_{473}(a)
$
and
$
f_{474}(a)
$. Expression $
f_{473}(a)
$ is defined just as $
f_{471}(a)
$, but with each $\text{root}_1$ occurring in the definition of $\vr_{477}(a)$ replaced by $\text{root}_2$, further, with $\vr _{4711}(a)$ replaced by $\vr _{4712}(a)$. Expression $
f_{474}(a)
$ is obtained from $
f_{472}(a)
$ via the same two replacement rules. If $a>\vr_{479}$, then the 
\[(\text{equidistant point on the real axis from the positive real pole and the upper complex pole})_1 
\]
is smaller than the
\[ (\text{equidistant point on the real axis from the positive real pole and the negative real pole})_2,\] hence here $B(a)$ is determined solely by the two real poles. Notice that quantity $(...)_{2}$ is simply the midpoint between the two real poles.  For \[-0.358565\approx \vr_{479}<a\le \vr_{4710} \approx -0.274796,\] 
the expression  $(...)_{2}$  is non-positive, so its absolute value gives $B(a)$. But if  $a>\vr_{4710}$, then $(...)_{2}$ is strictly positive, so 
$0\notin I(\alpha_0)$ again, thus $B(a)=0$ here.

We now summarize the above information in one formula for $B(a)$. (For completeness' sake we also provide some additional information on $B(a)$ in the interval $a\in(-\infty,\vr_{472}]$ already investigated earlier. We remark that  $a=\vr_{475}\approx -0.850052$ corresponds to the geometric configuration when $\psi_a$ has three poles---one positive real and two complex poles---with equal real part, and the 4$^\mathrm{th}$, real pole lies to the right.) We have

\[ 
B(a)  \begin{cases}
= +\infty & \text{if } a\in(-\infty,\vr_{475}], \\
\ge \vr_{476} & \text{if } a\in(\vr_{475},\vr_{471}],\\
= 0 & \text{if } a\in(\vr_{471},\vr_{472}),\\
= 0 & \text{if } a\in[\vr_{472},\vr_{478}],\\
=  \left|\frac{f_{471}(a)}{f_{472}(a)}\right|& \text{if } a\in\left(\vr_{478},-\frac{3}{7}\right],\\
=  \,\left|\frac{f_{473}(a)}{f_{474}(a)}\right| & \text{if } a\in\left(-\frac{3}{7},\vr_{479}\right],\\
=  \frac{1}{2}\left|\vr_{4711}(a)+\vr_{4712}(a)\right| & \text{if } a\in\left(\vr_{479},\vr_{4710}\right],\\
=  0 & \text{if } a\in\left(\vr_{4710},+\infty\right)
\end{cases}
\]
with $\vr_{476}\approx 21.5907$. 

After describing the function $B$, we find its maximal value on $a\in [\vr_{472},+\infty)$, or, equivalently, on  $[\vr_{472},\vr_{4710}]$. \textit{Mathematica}'s $\texttt{Maximize}$ was finally able to locate this unique point.  It turns out that $B\big|_{[\vr_{472},\vr_{4710}]}$ attains its maximal value at
\[
a_{47}^*\approx -0.4398493860002001824004494,
\]
where $a_{47}^*$ is the smaller real root of a polynomial with integer coefficients and of degree 30. The corresponding maximal value is 
\[B\left(a_{47}^*\right)\approx 2.743911895676330804848228,\]
expressible as the larger real root of a polynomial with integer coefficients and again of degree 30.

\begin{figure}[h]
  \centering
  \includegraphics[width=5in]{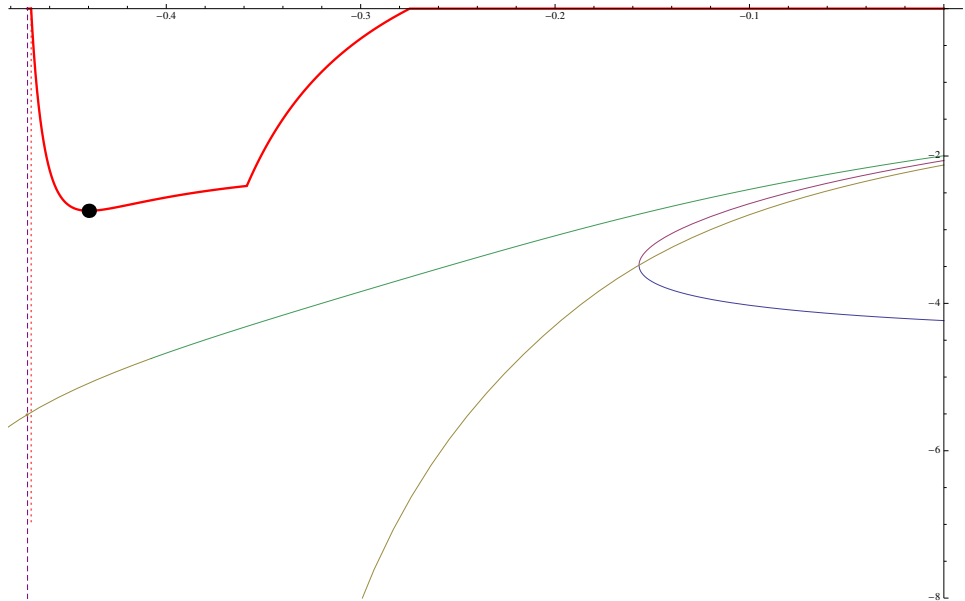}
  \caption{The graph of the function $a\mapsto -B(a)$ (red, thick curve carrying the black dot) for $\vr_{472}\le a\le 0$ is depicted together with the roots of the $0^\mathrm{th}$ and $1^\mathrm{st}$ derivatives of $\psi_a(\cdot)$ in the $\Pi_{4/4,7}$ class. $B(a)=0$ also between the two vertical dashed lines on the left. The distance between the black dot and the horizontal axis is the optimal radius of absolute monotonicity.}\label{472figure}
\end{figure}

As a last step, we check that no roots of the derivatives of the function $\psi_{a_{47}^*}(\cdot)$ can enter the interval $[-B\left(a_{47}^*\right),0]$, so its radius of absolute monotonicity is as large as it can be (Theorem \ref{vdgkTheorem3.3}), based exclusively on the location of the poles: in view of Theorem \ref{absmon_int} with $x=-B\left(a_{47}^*\right)$, it is enough to show that for all $0\le k\in\mathbb{N}$ we have
\[
\psi_{a_{47}^*}^{(k)}\left(-B\left(a_{47}^*\right) \right)\ge 0.
\]
Again, the most convenient form to use is the partial fraction decomposition of $\psi_{a_{47}^*}(x)$
(this decomposition also reveals that $\psi_{a_{47}^*}$ has no poles in the interval $[-B\left(a_{47}^*\right),0]$, hence Theorem \ref{absmon_int} is applicable). However, in this case, due to the high degree polynomial involved in the definition of $a_{47}^*$, \Mma  could not find the exact partial fraction decomposition of $\psi_{a_{47}^*}(x)$ in a reasonable amount of time (\textit{i.e.}, expressing the poles and the coefficients of the partial fractions as explicit and exact algebraic numbers). 
In order to overcome this difficulty,  we have employed complex interval arithmetic with rational endpoints. 
We started from good enough lower and upper rational bounds on $a_{47}^*$, $B\left(a_{47}^*\right)$ and the 4 poles: we applied validated numerical algorithms (such as \texttt{IsolatingInterval}, with tolerance $10^{-20}$) in the case of real algebraic numbers, and higher precision evaluation in the case of complex roots, which have been previously shown to be the unique roots within larger rational rectangles in the complex plane by \texttt{IsolatingInterval}. Then we also expressed the coefficients of the partial fraction decomposition in terms of the poles and $a_{47}^*$ symbolically in advance. At the end, we were able to give rigorous lower and upper rational bounds on the (absolute value of the) coefficients using only rational arithmetic without any difficulty.  These computations however produced quite lengthy outputs, since the numerators and denominators of some intermediate rational numbers in the bounding intervals consisted of integers with more than 220 digits. With the above simple interval technique we were able to completely reproduce the numerical partial fraction decomposition of $\psi_{a_{47}^*}(x)$ obtained directly in a much simpler manner.  

The partial fraction decomposition of $\psi_{a_{47}^*}(x)$ has the following form:
\[
c+\frac{c_0}{\alpha_0-x}+\frac{c_1}{\alpha_1-x}+\frac{c_2}{\alpha_2-x}+\frac{\overline{c_2}}{\overline{\alpha_2}-x},
\]
with some $c, c_0, c_1\in\mathbb{R}$, $\alpha_0, \alpha_1 >0$ and $c_2, \alpha_2\in\mathbb{C}\setminus\mathbb{R}$, further, the following numerical approximations are valid (to simplify our presentation, we now omit listing any exact rational bounds):
\[
c\approx 11.666934, \quad c_0\approx 202.617318, \quad \alpha_0\approx 5.779490, 
\]
\[
c_1\approx -1084.668490, \quad \alpha_1\approx 47.331517,
\]
\[
c_2\approx -19.880058- i\cdot 86.974803, \quad \alpha_2\approx 4.778363+i\cdot 4.007962.
\]
If $k\ge 1$, then $\psi_{a_{47}^*}^{(k)}(x)$ can be written as 
\[
k! (\alpha_0-x)^{-k-1}\left( c_0+c_1\left(\frac{\alpha_0-x}{\alpha_1-x}\right)^{k+1}
+c_2\left(\frac{\alpha_0-x}{\alpha_2-x}\right)^{k+1}+
\overline{c_2}\left(\frac{\alpha_0-x}{\overline{\alpha_2}-x}\right)^{k+1}\right).
\]
Now we evaluate the above expression at $x=x^*:=-B\left(a_{47}^*\right)<0$. 
We see by construction that $\left| \frac{\alpha_0-x^*}{\alpha_2-x^*}\right|=\left| \frac{\alpha_0-x^*}{\overline{\alpha_2}-x^*}\right|=1$, and 
it can also be proved that $0<\frac{\alpha_0-x^*}{\alpha_1-x^*}<\frac{1}{5}$. So by taking into account  $\alpha_0>0,-x^*>0$ and $c_0>0$, a sufficient condition for the positivity of $\psi_{a_{47}^*}^{(k)}(x^*)$ ($k\ge 1$) is that
$
|c_1|\cdot \left(\frac{1}{5}\right)^{k+1}+2|c_2|\cdot 1< c_0,
$
or, by using the rigorous bounds $|c_1|<1085$, $|c_2|<90$ and $c_0>202$, a weaker sufficient condition is given by 
$
1085\cdot \left(\frac{1}{5}\right)^{k+1}+180< 202,
$
which is seen to hold if $k\ge 2$. But for $k=0$ and $k=1$ one checks directly (with simple rational bounds on $a_{47}^*$ and 
$-B\left(a_{47}^*\right)$, and even without using partial fraction decomposition) that $\psi_{a_{47}^*}^{(k)}\left(-B\left(a_{47}^*\right) \right)> 0$ .

The above computations confirm that the optimal radius of absolute monotonicity within the $\Pi_{4/4,7}$ class is $B\left(a_{47}^*\right)\approx 2.743911$.

\section{Determination of $\widehat{R}_{2/2,2}$ and $\widehat{R}_{3/3,p}$ for $2\le p \le 4$}\label{sectionSDIRKs2s3}

 As an introduction to the $\widehat{\Pi}_{s/s,p}$ class, we first list the two optimal rational functions in the 
$\widehat{\Pi}_{2/2,p}$ family for $p=2$ and $p=3$. 

In the former case, for the sake of completeness, we also provide two direct proofs of the equality $\widehat{R}_{2/2,2}=4$. In the beginning of Section \ref{determinationRhatp2s34}, we have already indicated that this equality simply follows from the corresponding result in \cite{vdgk}. Our first proof will be similar to the earlier ones---carried out in one-parameter families of rational functions, and not knowing the optimum in advance; while the second proof is based on the same idea that we will apply in the $\widehat{\Pi}_{s/s,2}$ cases with $3\le s\le 4$---having more than one parameter, but with an already conjectured optimum.
\begin{lem}\label{lemma61.rhat2/2,2=4}
$\widehat{R}_{2/2,2}=4$.
\end{lem}
\begin{proof}[(first proof)]
Any element of $\widehat{\Pi}_{2/2,2}$ can be represented as
\[
\psi_a(z)=\frac{\frac{1}{2} \left(2 a^2-4 a+1\right) z^2+(1-2 a) z+1}{(1-a z)^2}
\]
with $a\in\mathbb{R}$. First we perform the same preliminary reduction as we did in Section \ref{determinationRhatp2s34}. If $a=0$, then the rational function reduces to the second degree Taylor polynomial of the exponential function around $0$, which has $R=1$. So we can assume $a\ne 0$. The numerator and denominator have a common root if and only if $a=\frac{1}{2}$, in which case $R=2$. If $a\ne \frac{1}{2}$, then due to Theorem \ref{vdgkCorollary3.4}, $a>0$ is necessary for $R>0$. So we can consider only the parameter values $0<a\ne\frac{1}{2}$. If $a\ge 1$, then
\[
\psi_a^{(4)}(0)=-12 a^2 (4 a-3)<0<\frac{36 a^2 \left(2 a^2-2 a+1\right)}{(a+1)^6}=\psi_a^{(4)}(-1),
\]
meaning that there is a root in $[-1,0]$ of the continuous function  $\psi_a^{(4)}(\cdot)$: for these $a$ values, $R\le 1$ by Theorem \ref{vdgkLemma4.5} with $\ell=4$. We will find the optimal $R$ in the interval $a\in(0,1)\setminus\{\frac{1}{2}\}$, determined by the intersection of certain roots of $\psi_a$ and $\psi_a^\prime$. (Let us add that a plot of the roots of the lowest order derivatives of $\psi_a(\cdot)$ in the region $0<a<1$ would be very similar to Figure \ref{332figure}.)  If $a=1-\frac{1}{\sqrt{2}}$, then the leading coefficient of the numerator of $\psi_a$ vanishes, and in this case $R\le 1+\sqrt{2}$, because $\psi_a(-1-\sqrt{2})=0$. If $a\in(0,1)\setminus\{\frac{1}{2},1-\frac{1}{\sqrt{2}}\}$ and $a>\frac{1}{4}$, then one negative root of $\psi_a(\cdot)$ is \[\frac{-1+2 a+\sqrt{4 a-1}}{2 a^2-4 a+1}>-4,\] while if $a\in(0,1)\setminus\{\frac{1}{2},1-\frac{1}{\sqrt{2}}\}$ and $a<\frac{1}{4}$,
then the only root of $\psi_a^\prime(\cdot)$ is $0>\frac{1}{3 \alpha -1}>-4,$ implying that for $a\ne \frac{1}{4}$, $R<4$. Finally, if $a=\frac{1}{4}$, then $\psi_a(z)=\frac{\left(1+\frac{z}{4}\right)^2}{\left(1-\frac{z}{4}\right)^2}$, and one directly checks that now $R=4$. 
\end{proof}

\begin{proof}[(second proof)]
We can again consider only functions of the form 
$
\psi_a(z)=\frac{\frac{1}{2} \left(2 a^2-4 a+1\right) z^2+(1-2 a) z+1}{(1-a z)^2}
$ with $a>0$, but now, as opposed to the First proof, we conjecture $\widehat{R}_{2/2,2}=4$ in advance. We show that 
\[
\psi_{a}(-4)\ge 0, \quad \psi_{a}^{\prime}(-4)\ge 0
\]
and
\[
\forall k\in\mathbb{N}, k\ge 2:\quad \psi_{a}^{(k)}(-4)\ge 0
\]
imply $a=\frac{1}{4}$. Indeed, the first two conditions with $a>0$ amount to
\[
16 a^2-24 a+5\ge 0, \quad 12 a-3\ge 0,
\]
that is, to $a\in \{1/4\}\cup [5/4,+\infty)$. Then, instead of the third condition, it is sufficient to require the weaker inequality (\ref{limitcondition})---expressing the fact that the leading coefficient of a polynomial in $k$ has to be non-negative, if the polynomial is non-negative for all $k\ge 2$---that reads as $\frac{1}{2}-a\ge 0$, implying the unique value $a=\frac{1}{4}$.
\end{proof}

As for the  $\mathbf{\widehat{\Pi}_{2/2,3}}$ class, it has only two members: 
\[
\psi_{231}(z)=\frac{-\frac{1}{6} \left(\sqrt{3}+1\right) z^2-\frac{z}{\sqrt{3}}+1}{\left(1-\frac{1}{6} \left(3+\sqrt{3}\right) z\right)^2}\]
and
\[
\psi_{232}(z)=\frac{\frac{1}{6} \left(\sqrt{3}-1\right) z^2+\frac{z}{\sqrt{3}}+1}{\left(1-\frac{1}{6} \left(3-\sqrt{3}\right) z\right)^2}.
\]
One easily checks that $R(\psi_{231})\le \frac{\sqrt{3}+\sqrt{3 \left(3+2 \sqrt{3}\right)}}{1+\sqrt{3}}<\frac{5}{2}$ (as shown by a root of $\psi_{231}$). On the other hand, $R(\psi_{232})=1+\sqrt{3}$, since $\psi_{232}=\psi_{a^*_{23}}$, \textit{i.e.,} the optimal element of the $\Pi_{2/2,3}$ class, so  
$\widehat{R}_{2/2,3}=R_{2/2,3}=1+\sqrt{3}$.

\subsection{Determination of $\widehat{R}_{3/3,4}$}\label{sectionSDIRKs3p4}

Any element of $\widehat{\Pi}_{3/3,4}$ can be represented in the form
\[
\psi(z)=\frac{a_3 z^3+a_2 z^2+a_1 z+1}{(1-a z)^3}
\]
with suitable real parameters $a, a_1, a_2$ and $a_3$. One directly computes that there are exactly three solutions $(a, a_1, a_2, a_3)\in \mathbb{R}^4$ to the system
\[
\psi^{(k)}(0)=1,\quad\quad k=0,1,2,3,4
\]
(we remark that we would get the same three \textit{real} solutions if we allowed $(a, a_1, a_2, a_3)\in \mathbb{C}^4$). For all of these three solutions, condition $a>0$ is automatically satisfied (\textit{c.f.} Theorem \ref{vdgkCorollary3.4}), and we have 
for $m=1, 2, 3$ that $a=\text{root}_m(24,-36,12,-1)$, 
$a_1=\text{root}_{4-m}(8,12,-12,1)$, $a_2=\text{root}_{\frac{1}{2} \left(3 m^2-11 m+12\right)}(64,-48,0,1)$ and $a_3=\text{root}_{\frac{1}{2} \left(3 m^2-11 m+12\right)}(1536,-1152,-24,1)$.
Let us denote the corresponding rational functions by $\psi_{34m}$ ($m=1,2,3$).  We can check that $\frac{1}{a}$ is a pole of order 3 in each case, so no cancellation between the numerator and denominator occurs. We are going to prove that 
\[
\max_{m=1, 2, 3} R(\psi_{34m})=-x^*
\]
with \[x^*:=\text{root}_1(1,12,60,120,-144,-1152,-1536,1152,2304,-1536)\approx\]
\[ 
-3.287278451851993925371346
\] and the maximum occurring at $m=1$.

Let us consider the $m=2$ case first. Then $\psi_{342}^{(6)}(0)=\text{root}_1(8,-3540,600,4625)<0$,  so $R(\psi_{342})=0$ by definition.

Next if $m=3$, then $\psi_{343}^{(5)}\left(\text{root}_5(1,66,1242,7008,1872,-648,-72)\right)=0$, so Theorem \ref{vdgkLemma4.5} with $\ell=5$ says that $R(\psi_{343})\le -\text{root}_5(1,66,1242,7008,1872,-648,-72)\approx 0.0943315<|x^*|$.

Finally, if $m=1$, then the partial fraction decomposition of $\psi_{341}$ reads as
\[
\psi_{341}(x)=c+\sum_{\ell=1}^3 \frac{c_\ell}{(\alpha_0-x)^\ell}, 
\]
where $c=\text{root}_1(1,9,-9,-9)\approx -9.82294$, $c_1=\text{root}_3(1,-288,-5184,13824)\approx 304.856$, $c_2=\text{root}_1(1,3168,-243648,-235008)\approx -3243.10$, $c_3=$ $\text{root}_3(1,-11808,-684288,110592)$ $\approx$ $11865.6$, and $\alpha_0=\text{root}_3(1,-12,36,-24)\approx 7.75877$. Then $\psi_{341}(x^*)=0$. Since now
trivially $B(\psi_{341})=+\infty$, Theorem \ref{absmon_int} with $x:=x^*$ proves absolute monotonicity of 
$\psi_{341}$ on $[x^*,0]$ provided that we show the point conditions
\[
\psi_{341}^{(k)}(x^*)\ge 0
\]
for all $k\ge 1$ (the theorem is applicable due to the lack of poles in $(-\infty,0]$). Indeed, for such $k$ values we have
\[
\psi_{341}^{(k)}(x^*)=(k+2)!(\alpha_0-x^*)^{-k-3}\left[\frac{c_3}{2}+\frac{c_2}{k+2}(\alpha_0-x^*)+\frac{c_1}{(k+1)(k+2)}(\alpha_0-x^*)^2\right].
\]
Clearly, because of $\alpha_0-x^*>0$, it is enough to verify that  the $[\ldots]$ expression is non-negative, but since $\lim_{k\to+\infty}[\ldots]$ $=\frac{c_3}{2}>0$, only finitely many $k$ values should be checked.  We easily see that $[\ldots]\ge 0$ holds if, for example,
\[
\frac{\alpha_0-x^*}{k+2}\left(|c_2|+\frac{c_1}{k+1}(\alpha_0-x^*)\right)<
\frac{23/2}{k+2}\left(3250+\frac{305}{k+1}\cdot\frac{23}{2}\right)<5500<\frac{c_3}{2}.
\]
Now the second inequality above is satisfied for $k\ge 6$, and we determine directly that $[\cdots]>0$ holds for each $k=1, 2, \ldots, 5$ as well, completing the argument that the maximal radius of absolute monotonicity within the $\widehat{\Pi}_{3/3,4}$ class is $|x^*|$.

We add that partial fraction decomposition in the $m=2$ case shows that the corresponding  "dominant coefficient" $c_3$ is negative, immediately implying $R(\psi_{342})=0$ (\textit{c.f.} Remark \ref{remark1.13aboutThm4.4}).

\subsection{Determination of $\widehat{R}_{3/3,3}$}\label{sectionSDIRKs3p3}

We will prove that the maximal radius of absolute monotonicity in this $\widehat{\Pi}_{3/3,3}$ class is $2+\sqrt{8}\approx 4.82842$. Now the one-parameter family of rational functions takes the form
\[
\psi_a(z)=\frac{\frac{1}{6} \left(-6 a^3+18 a^2-9 a+1\right) z^3+\frac{1}{2} \left(6 a^2-6 a+1\right) z^2+(1-3 a) z+1}{(1-a z)^3}
\]
with $a\in\mathbb{R}$. If $a=0$, then the rational function reduces to the third degree Taylor polynomial of the exponential function around $0$, which has $R=1$. So we can assume $a\ne 0$.  Let us exclude two more exceptional parameter values as well. It is easily seen by computing the corresponding resultant that the numerator and denominator have a common root if and only if $a=\frac{1}{6} \left(3\pm\sqrt{3}\right)$. In any of these cases,  $\psi_a$ is either $\psi_{231}$ or $\psi_{232}$ from the $\widehat{\Pi}_{2/2,3}$ class (see the paragraph just before Section \ref{sectionSDIRKs3p4}), hence $R\le 1+\sqrt{3}<2+\sqrt{8}$.

Now we can apply Theorem \ref{vdgkCorollary3.4} saying that $a>0$ is necessary to have $R>0$.  So in the following we can suppose that $0<a\ne \frac{1}{6} \left(3\pm\sqrt{3}\right)$. 

One can prove by induction that for $k\ge 1$
\[
\psi_a^{(k)}(x)=\frac{a^{k-3} k!}{2}  (1-a x)^{-k-3}\ \times \]
\[
\Bigg((2 a-1) (6 a-1) a^2 x^2+\left(8 a^3 (k-2)+a^2 (9-7 k)+a (k-1)\right)x + \]
\begin{equation}\label{SDIRKs3p3explform}
 \left(a^2 (k-3) (k-2)-a (k-3) (k-1)+\frac{1}{6} (k-2)(k-1)\right)\Bigg).
\end{equation}

\begin{figure}[h]
  \centering
  \includegraphics[width=5in]{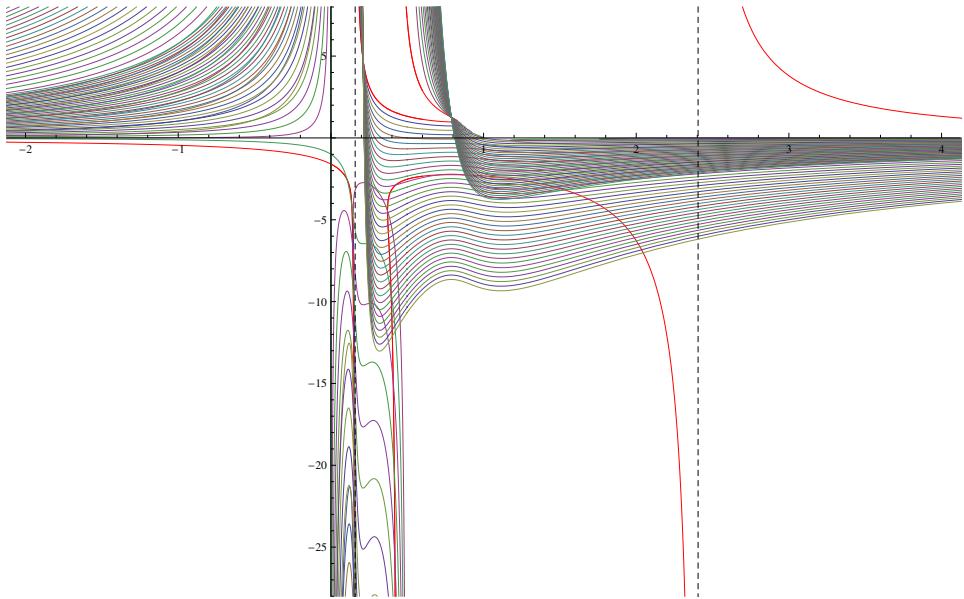}
  \caption{The $\widehat{\Pi}_{3/3,3}$ class. Each curve (corresponding to different $k$ values between $0$, $1, \ldots, 36$) shows a root of $\psi_a^{(k)}(\cdot)$ as the parameter $a$ is varied. The red curve ($k=0$) is obtained as a solution of a parametric cubic polynomial, whereas the other curves are derived from quadratic ones.}\label{331figure}
\end{figure}

\begin{figure}[h]
  \centering
  \includegraphics[width=3in]{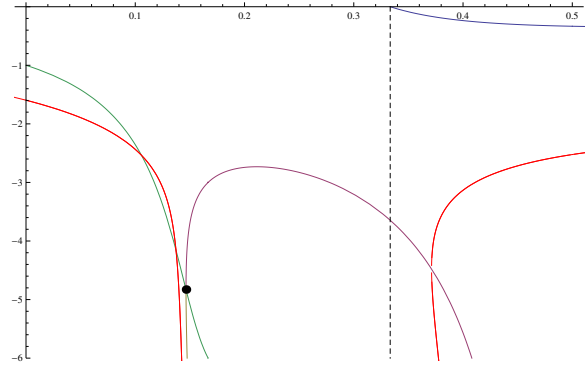}
  \caption{Zooming in on the previous figure (Figure \ref{331figure}) reveals the location of the optimum (depicted as a black dot) in the $\widehat{\Pi}_{3/3,3}$ class.}\label{332figure}
\end{figure}

 Let us consider first the $k=5$ special case of the above explicit formula with $a\ge \frac{1}{3}$. From this we see that one of the roots of $\psi_a^{(5)}(\cdot)$ is $-\frac{1}{3}$, if $a=\frac{1}{2}$, and
\[
\frac{-12 a^2+13 a-2-\sqrt{72 a^4-168 a^3+123 a^2-28 a+2}}{a (2 a-1) (6 a-1)}
\]
if $\frac{1}{3}\le a\ne \frac{1}{2}$. But this last expression has values in, say, $(-1,0]$ for  $\frac{1}{3}\le a\ne \frac{1}{2}$, implying, by Theorem \ref{vdgkLemma4.5} with $\ell=5$, that $R(\psi_a)\le 1$ for $\frac{1}{3}\le a$.

We focus on the remaining $a\in \left(0,\frac{1}{3}\right)\setminus\{\frac{3-\sqrt{3}}{6}\}$ parameter set now. If $0<a<\frac{2-\sqrt{2}}{4}$, then $\psi_a^{\prime\prime}(\cdot)$ has a root of the form
\[
\frac{5 a-1+\sqrt{-48 a^3+57 a^2-14 a+1}}{2 a (2 a-1) (6 a-1)}
\]
in the interval $(-2-\sqrt{8},-1)$, meaning that $R(\psi_a)<2+\sqrt{8}$ here. On the other hand, if 
$\frac{2-\sqrt{2}}{4}<a<\frac{1}{3}$ (the exceptional $\frac{3-\sqrt{3}}{6}$ value can safely be added again), then $\psi_a^{\prime}(\cdot)$ has a root of the form
\[
\frac{4 a-1+\sqrt{-8 a^2+8 a-1}}{(2 a-1) (6 a-1)}
\]
in the interval $(-2-\sqrt{8},-1-\sqrt{7})$, so $R(\psi_a)<2+\sqrt{8}$, too, for these $a$ values. (The $a=\frac{1}{6}$ case is a removable singularity of the above root expression with value $-3$.)

Finally, if $a=\frac{2-\sqrt{2}}{4}$, then $\psi_a\left(-2-\sqrt{8}\right)=1-\frac{2 \sqrt{2}}{3}>0$, and, by using
(\ref{SDIRKs3p3explform}), for any $k\ge 1$ we have
\[
\psi_a^{(k)}\left(-2-\sqrt{8}\right)=\frac{4}{3}  \left(\frac{3}{2}- \sqrt{2}\right)^k  k!(k-1) (k-2)\ge 0,
\]
showing that $\psi_a$ is absolutely monotonic at $x=-2-\sqrt{8}$. But then Theorem \ref{absmon_int}  (applicable because there are no poles in $(-\infty,0]$) guarantees absolute monotonicity on the whole interval $[-2-\sqrt{8},0]$, since now $B(\psi_a)=+\infty$ trivially, completing the proof of $\widehat{R}_{3/3,3}=2+\sqrt{8}$.

\subsection{Determination of $\widehat{R}_{3/3,2}$}\label{sectionSDIRKs3p2}

Formula (\ref{generalSDIRKpsi}) and Lemma \ref{a_pos} imply that now we can restrict our attention to rational functions of the form
\[
\psi_{a,c}(z)=\frac{c z^3+\left(\frac{1}{2}-3a+3a^2\right)z^2+\left(1-3a\right)z+1}{(1-a z)^{3} }
\]
with suitable parameters $a>0$ and $ c\in\mathbb{R}$. For simplicity, we have used and will use the letter $c$ instead of $a_3$. We are going to show that 
\[
\psi_{a,c}(-6)\ge 0, \quad \psi_{a,c}^{\prime}(-6)\ge 0,\quad \psi_{a,c}^{\prime\prime}(-6)\ge 0
\]
and
\[
\forall k\in\mathbb{N}, k\ge 3:\quad \psi_{a,c}^{(k)}(-6)\ge 0
\]
imply that $a=\frac{1}{6}$ and $c=\frac{1}{216}$. 

\begin{figure}[h]
 \centering
 \includegraphics[width=7in]{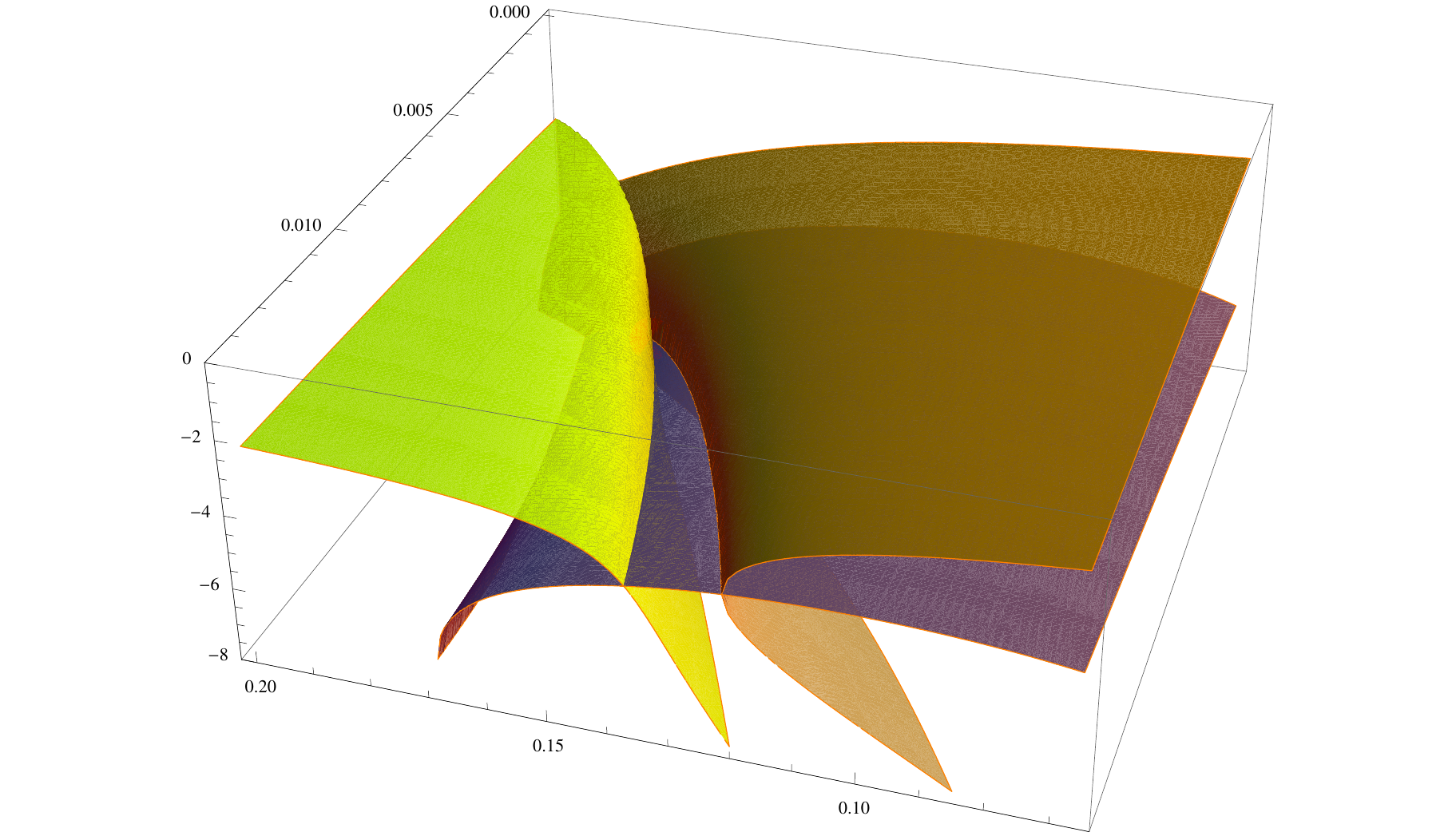}
 \caption{The $\widehat{\Pi}_{3/3,2}$ class. Depicting the roots of $\psi_{a,c}^{(k)}(\cdot)$ for $k=0, 1, 2$ as a function of $(a,c)$, the optimal point is obtained as the intersection of these three surfaces at $(a,c,x)=\left(\frac{1}{6},\frac{1}{216},-6\right)$.}
\end{figure}

First, by using (\ref{psigeneralderivative}) evaluated at $x=-6$, we see that these inequalities are equivalent to 
\begin{eqnarray}
  108 a^2-90 a-216 c+13 &\ge&  0\label{s3p21} \\ 
  108 a^3-108 a^2+42 a+108 c-5 &\ge& 0\label{s3p22} \\ 
  108 a^4-126 a^3+72 a^2-12 a+1/2 +(108 a-18) c &\ge& 0\label{s3p23} \\ 
k^2 \left(a^3-2 a^2+a/2+c\right)+k \left(-36 a^4+57 a^3-8 a^2-36 a c-a/2-3 c\right)+
 \nonumber & & \\
36 a c+2 c+216 a^5-180 a^4+26 a^3+216 a^2 c & \ge & 0, \label{s3p2k}
\end{eqnarray}
where, of course, (\ref{s3p2k}) should hold for all $k\ge 3$ integers. We will also use the necessary condition (\ref{limitcondition}), which now reads as
\begin{equation}\label{s3p2lc}
a^3-2 a^2+a/2+c \ge 0
\end{equation}
(being just the leading $k$-coefficient of (\ref{s3p2k})).

Let us express the linear parameter, $c$ from (\ref{s3p21}), (\ref{s3p22}) and (\ref{s3p2lc}), then combine the resulting inequalities to eliminate $c$ and get
\begin{equation*}\label{s3p24}
\frac{-108 a^3+108 a^2-42 a+5}{108}\le \frac{108 a^2-90 a+13}{216}\ge \frac{-2 a^3+4 a^2-a}{2}.
\end{equation*}
Rearranging both inequalities to $0$, the two cubic polynomials can be nicely factorized (containing factors $2a-1$, $6a-1$, $6a+1$ and $36 a^2-60 a+13$), so this---together with $a>0$---implies that
\begin{equation}\label{s3p25}
0<a\le \frac{1}{6}\quad\mathrm{or}\quad a\ge \frac{5+2 \sqrt{3}}{6}.
\end{equation}
Now let us consider (\ref{s3p23}) and separate two cases according to the coefficient of $c$. 

If it vanishes, that is, if $a=\frac{1}{6}$, then (\ref{s3p21}) and (\ref{s3p22}) yield $c=\frac{1}{216}$, and we are done, because we verify from the definition that $\psi_{a,c}(z)=\frac{\left(1+\frac{z}{6}\right)^3}{\left(1-\frac{z}{6}\right)^3}$ is absolutely monotonic on $[-6,0]$.

Hence the proof is finished as soon as we have shown that (\ref{s3p25}) and $a\ne\frac{1}{6}$ lead to a contradiction.

First suppose that $0<a< \frac{1}{6}$. Then an elementary calculation shows that (\ref{s3p23}) (whose 
left-hand side is just $(6 a-1) \left(36 a^3-36 a^2+18 a-1+36 c\right)/2$) and  (\ref{s3p22}) contradict to each other.

As a result, it is sufficient to exclude $a\ge \frac{5+2 \sqrt{3}}{6}$. Let us express $c$ (again) from (\ref{s3p2lc}) and  (\ref{s3p21}) to have
\begin{equation}\label{s3p26}
\frac{-2 a^3+4 a^2-a}{2}\le c\le \frac{108
   a^2-90 a+13}{216},
\end{equation}
and let us abbreviate the left-hand side of (\ref{s3p2k}) by $\lambda(k,a,c)$.

\begin{lem}\label{SDIRKs3p2Lemma6.1}
For each $a\ge \frac{5+2 \sqrt{3}}{6}$ and each $c$ satisfying (\ref{s3p26}), there is an integer $k\ge 3$  such that $\lambda(k,a,c)<0$.
\end{lem}
\begin{proof}
Let us pick and fix throughout the proof an arbitrary $a\ge \frac{5+2 \sqrt{3}}{6}$ and $c$ satisfying (\ref{s3p26}). Since
$
\partial_c \lambda(k,a,c)=k^2-(36 a+3) k+216 a^2+36 a+2,
$
we see that $\partial_c \lambda(k,a,c)<0$ holds if and only if
\begin{equation}\label{s3p2lemma1}
\varrho^{-}_{321}(a)<k<\varrho^{+}_{321}(a),
\end{equation}
where $\varrho^{\pm}_{321}(a):=\frac{3}{2} (12 a+1)\pm\frac{1}{2} \sqrt{432 a^2+72 a+1}$.
It is also easily seen that---for the allowed $a$ values---(\ref{s3p2lemma1}) automatically implies $k\ge 3$. So, if $k$ satisfies (\ref{s3p2lemma1}), then due to (\ref{s3p26}) we have that
\[
\lambda(k,a,c)\le \lambda\left(k,a,\frac{-2 a^3+4 a^2-a}{2}\right)=a (6 a+1) \left(
(1-2 a) k+36 a^2-8 a-1\right).
\]
Since $a>\frac{1}{2}$ now, this very last factor is negative if and only if 
\begin{equation}\label{s3p2lemma2}
k>\frac{36 a^2-8 a-1}{2 a-1}.
\end{equation}
These mean that (\ref{s3p2lemma1}) and (\ref{s3p2lemma2}) imply $\lambda(k,a,c)<0$. But 
$
\varrho^{-}_{321}(a)<\frac{36 a^2-8 a-1}{2 a-1}
$
holds precisely if $a\in(0,1/6)\cup (1/2,+\infty)$, which is now true by the assumption on $a$, so to finish the proof, it is enough to show that the set 
\[
\left( \frac{36 a^2-8 a-1}{2 a-1} , \varrho^{+}_{321}(a)\right) \cap \mathbb{N}
\]
is not empty: an elementary computation yields that for $a\ge \frac{5+2 \sqrt{3}}{6}$ we have 
\[
\varrho^{+}_{321}(a)-\frac{36 a^2-8 a-1}{2 a-1} >9. 
\]
\end{proof}

\begin{rem}\label{length9} 
We add that this last integer in the proof, $9$, could not be replaced by, say, $10$ on the whole interval $a\in \left( \frac{5+2 \sqrt{3}}{6} , +\infty\right)$.
\end{rem}

\begin{rem}  It is possible to finish the proof Lemma \ref{SDIRKs3p2Lemma6.1} by using the $\partial_c \lambda(k,a,c)>0$ case, but this would result in more difficult computations and would provide (asymptotically) an interval of length 6 (instead of 9) in the last step.
\end{rem}

\section{Determination of $\widehat{R}_{4/4,2}$}\label{sectionSDIRKs4p2}

We will apply the same principles as in Section \ref{sectionSDIRKs3p2} (or as in the second proof of Lemma \ref{lemma61.rhat2/2,2=4}), but, this time, we could eliminate the "linear" parameters $d$ and $c$---recursively via monotonicity---only at the cost of a little bit more computation.\\

By appealing to formula (\ref{generalSDIRKpsi}) and Lemma \ref{a_pos} again, it is sufficient to consider rational functions of the form
\[
\psi_{a,c,d }(z)=\frac{d z^4+c z^3+\left(\frac{1}{2}-4a+6a^2\right)z^2+\left(1-4a\right)z+1}{(1-a z)^{4} }
\]
with suitable parameters $a>0$ and $ c, d\in\mathbb{R}$. (Again, for easier readability we have used and will use $c$ instead of $a_3$, and $d$ instead of $a_4$.) We show that 
\[
\psi_{a,c, d}(-8)\ge 0, \quad \psi_{a,c, d}^{\prime}(-8)\ge 0,\quad \psi_{a,c, d}^{\prime\prime}(-8)\ge 0,\quad \psi_{a,c, d}^{\prime\prime\prime}(-8)\ge 0
\]
and
\[
\forall k\in\mathbb{N}, k\ge 4:\quad \psi_{a,c, d}^{(k)}(-8)\ge 0
\]
imply  $a=\frac{1}{8}$, $c=\frac{1}{128}$ and $d=\frac{1}{4096}$.

Let us use (\ref{psigeneralderivative}) at $x=-8$ (dropping the always positive factors again) and write out the above inequalities in detail to get 
\begin{eqnarray}
  1-8 (1-4 a)+64 \left(6 a^2-4 a+\frac{1}{2}\right)-512 c+4096 d &\ge&  0\label{s4p21} 
\end{eqnarray}
\begin{eqnarray}
(8 a+1) \left(1-4a-16 \left(6 a^2-4 a+\frac{1}{2}\right)+192 c-2048 d\right)+\nonumber & & \\
4 a \left(1-8 (1-4 a)+64 \left(6 a^2-4 a+\frac{1}{2}\right)-512 c+4096 d\right) &\ge& 0\label{s4p22} 
\end{eqnarray}
\begin{eqnarray}
10 a^2 \left(1-8 (1-4 a) +64 \left(6 a^2-4 a+\frac{1}{2}\right)-512 c+4096 d\right)+ \nonumber & & \\  
4 a (8 a+1) \left(1-4a-16 \left(6 a^2-4 a+\frac{1}{2}\right)+192 c-2048 d\right)+ \nonumber & & \\
\frac{1}{2} (8 a+1)^2 \left(2 \left(6 a^2-4 a+\frac{1}{2}\right)-48 c+768 d\right)  &\ge& 0\label{s4p23} 
\end{eqnarray}
\begin{eqnarray}
10 a^2 (8 a+1) \left(1-4a-16 \left(6 a^2-4 a+\frac{1}{2}\right)+192 c-2048 d\right)+ \nonumber & & \\ 
2 a (8 a+1)^2 \left(2 \left(6 a^2-4 a+\frac{1}{2}\right)-48 c+768 d\right)+\frac{1}{6} (8 a+1)^3 (6 c-192 d)\, + \nonumber & & \\
 20 a^3 \left(1-8 (1-4 a)+64 \left(6 a^2-4 a+\frac{1}{2}\right)-512 c+4096 d\right)  &\ge& 0\label{s4p24} 
\end{eqnarray}
\begin{eqnarray}
k^3 \left(6 a^4-6 a^3+a^2+2 a c+2 d\right)+\nonumber & & \\
k^2 \left( -384 a^5+324 a^4-42 a^3-144 a^2 c-6 a c-192 a d-12 d \right)+
 \nonumber & & \\
k \left( 4608 a^6-3072 a^5+618 a^4+2304 a^3 c-36 a^3+144 a^2 c+4608 a^2 d\ -\right.\nonumber & & \\
\left. a^2+4 a c+576 a d+22 d \right) +  \nonumber & & \\
4608 a^6-2688 a^5-6144 a^4 c+300 a^4-24576 a^3 d-4608 a^2 d-384 a d-12 d & \ge & 0, \label{s4p2k}
\end{eqnarray}
with  (\ref{s4p2k}) valid for any integer $k\ge 4$. The necessary condition (\ref{limitcondition}) now reads as
\begin{equation}\label{s4p2lc}
6 a^4-6 a^3+a^2+2 a c+2 d\ge 0.
\end{equation}

The domain of the "non-linear" parameter $a>0$ will be divided at the following "natural" points
(determined by the requirement that denominators in the proof below have constant sign on the corresponding intervals)
\[
0<\frac{3-\sqrt{5}}{16}<\frac{1}{12}<\frac{1}{8}<\frac{3}{16}<\frac{3+\sqrt{5}}{16}<\frac{9+2 \sqrt{6}}{24}<\ldots\]
and also at the "artificial" point $\ldots<\frac{22}{10}$ (introduced for technical reasons).  The explicit formulae for the $P_{42n}$ and $\varrho_{42n}$ expressions ($n=1, 2, \ldots$) appearing soon are listed in the Appendix (or directly in the proofs).

\textbf{The interval $a\in\left(0,\frac{3-\sqrt{5}}{16}\right)$.} From (\ref{s4p24}), $d$ is expressed as
$d\le \frac{P_{421}(a,c)}{2048 a^2-768 a+32}$, while from (\ref{s4p23}) as $d\ge \frac{P_{422}(a,c)}{ 4096 a-768}$. By joining these inequalities, $d$ is eliminated and we get
\begin{equation}\label{s4p210}
c\le \frac{-6144 a^5+5120 a^4-2144 a^3+400 a^2-32 a+1}{1536 a^2-256 a+24}.
\end{equation}
From (\ref{s4p22}) we obtain $d\le \frac{P_{423}(a,c)}{2048}$, which, together with $d\ge \frac{P_{422}(a,c)}{ 4096 a-768}$ allows us to eliminate $d$ again and get
\begin{equation}\label{s4p211}
c\ge \frac{1}{192} \left(-768 a^3+512 a^2-144 a+13\right).
\end{equation}
From (\ref{s4p210}) and (\ref{s4p211}) we get $(8 a-1) (8 a+1) (16 a-3)\le 0$, which is impossible for $a\in\left(0,\frac{3-\sqrt{5}}{16}\right)$.

\textbf{The case when $a=\frac{3-\sqrt{5}}{16}$.} From (\ref{s4p24}), $c$ can be expressed as $c\le \frac{1}{128} \left(2 \sqrt{5}-3\right)$. But  (\ref{s4p211}) now says $\frac{1}{384} \left(1+6 \sqrt{5}\right)\le c$, a contradiction.

\textbf{The interval $a\in\left(\frac{3-\sqrt{5}}{16},  \frac{3}{16}\right)$ (containing the unique solution).} We have now 3 lower estimates: $d\ge \frac{P_{422}(a,c)}{ 4096 a-768}$, $d\ge \frac{P_{421}(a,c)}{2048 a^2-768 a+32}$, and (from (\ref{s4p21})) $d\ge \frac{P_{424}(a,c)}{4096}$. Moreover, an upper estimate $d\le \frac{P_{423}(a,c)}{2048}$. By making 3 appropriate pairs from these, we eliminate $d$ in each case and get (\ref{s4p211}), 
\begin{equation}\label{s4p212}
c\le \frac{1}{128} \left(192 a^2-104 a+11\right),
\end{equation}
and
\begin{equation}\label{s4p213}
\begin{cases}
 c\le \frac{-6144 a^4+4608 a^3-1600 a^2+200 a-7}{1536 a-128}  & \mathrm{\quad if \quad} \frac{1}{16}
   \left(3-\sqrt{5}\right)<a<\frac{1}{12}, \\
 c\ge \frac{-6144 a^4+4608 a^3-1600 a^2+200 a-7}{1536 a-128} & \mathrm{\quad if \quad}  \frac{1}{12}<a<\frac{3}{16}.
\end{cases}
\end{equation}
(It is easily seen that $a=1/12$ immediately leads to a contradiction.)

On $\frac{1}{16}\left(3-\sqrt{5}\right)<a<\frac{1}{12}$, (\ref{s4p211}) and (\ref{s4p213}) imply $\frac{5}{3} (8 a-1) (8 a+1)\ge 0$, which is impossible.

If $\frac{1}{12}<a<\frac{3}{16}$, then  (\ref{s4p211}) and (\ref{s4p213}) imply
\begin{equation}\label{s4p214}
\begin{cases}
 c\ge \frac{-6144 a^4+4608 a^3-1600 a^2+200 a-7}{1536 a-128} & \mathrm{\quad if \quad} \frac{1}{12}
  <a\le\frac{1}{8}, \\
 c\ge \frac{1}{192} \left(-768 a^3+512 a^2-144 a+13\right) & \mathrm{\quad if \quad}  \frac{1}{8}<a<\frac{3}{16}.
\end{cases}
\end{equation}
On $\frac{1}{8}<a<\frac{3}{16}$, (\ref{s4p212}) and (\ref{s4p214}) cannot be true simultaneously. On the other hand, if $\frac{1}{12}<a\le\frac{1}{8}$, then (\ref{s4p212}) and (\ref{s4p214}) yield $a=1/8$ and $c=1/128$, and with $d\ge \frac{P_{424}(a,c)}{4096}$ and $d\le \frac{P_{423}(a,c)}{2048}$ we get $d=1/4096$.

\textbf{The interval $a\in\left[\frac{3}{16},  \frac{9+2 \sqrt{6}}{24}\right)$.} We will use here the  two lower estimates $d\ge \frac{P_{424}(a,c)}{4096}$, $d\ge \frac{1}{2} \left(-6 a^4+6 a^3-a^2-2 a c\right)$ (this second one derived from (\ref{s4p2lc})) and the upper estimate $d\le \frac{P_{423}(a,c)}{2048}$. Again, by forming two appropriate pairs, $d$ is eliminated and we obtain two inequalities with $a$ and $c$. After a rearrangement, we have  (\ref{s4p212})  and $c\ge \frac{1}{192}\left(-768 a^3+768 a^2-160 a+7\right)$, leading to a contradiction.\\

So far we have examined the region $0 < a < \frac{9+2 \sqrt{6}}{24}$ and proved that the system has a solution if and only if $a=1/8$. 

In the rest of this section, we will show that there are no solutions if $a\ge \frac{9+2 \sqrt{6}}{24}$. It can be shown (in a few lines, by using only  (\ref{s4p21}),  (\ref{s4p22}) and  (\ref{s4p2lc}), but skipping the details here) that if $a\ge \frac{9+2 \sqrt{6}}{24}$, then  we have
\begin{equation}\label{s4p215}
\begin{cases}
 \frac{1}{192} \left(-768 a^3+768 a^2-160 a+7\right)\le c \le
   \frac{1}{512} \left(-1536 a^3+1728 a^2-424 a+25\right),  &   \\
 \frac{1}{2} \left(-6 a^4+6 a^3-a^2-2 a c\right)\le d\le 
   \frac{1}{2048}\left(768 a^3-512 a^2-512 a c+104 a+192 c-7\right) & 
\end{cases}
\end{equation}
or
\begin{equation}\label{s4p216}
\begin{cases}
 \frac{1}{512} \left(-1536 a^3+1728 a^2-424 a+25\right)\le c\le
   \frac{1}{128} \left(192 a^2-104 a+11\right), &   \\
 \frac{1}{4096}\left(-384 a^2+224 a+512 c-25\right)\le d\le 
   \frac{1}{2048}\left(768 a^3-512 a^2-512 a c+104 a+192 c-7\right). & 
\end{cases}
\end{equation}

 Let us introduce the abbreviation $\lambda(k,a,c,d)$ to denote the left-hand side of (\ref{s4p2k}).

\begin{lem}
For each $\frac{1}{24} \left(9+2 \sqrt{6}\right)\le a\le \frac{22}{10}$, and any $c$ and $d$ satisfying (\ref{s4p215}), we have that $\lambda(77,a,c,d)<0$.
\end{lem}
\begin{proof}
Since the polynomial $\partial_d \lambda(77,a,c,d)=-48 P_{425}(a)$ has a unique root $\varrho_{421}\approx 1.17854$ in the given $a$-interval, we separate two cases.

\textbf{The interval $a\in\left[\frac{1}{24} \left(9+2 \sqrt{6}\right),  \varrho_{421}\right)$.} Here $\partial_d \lambda(77,a,c,d)>0$, so by (\ref{s4p215})
\[
\lambda(77,a,c,d)\le \lambda\left(77,a,c, \frac{1}{2048}\left(768 a^3-512 a^2-512 a c+104 a+192 c-7\right)\right)=\]\[\frac{57}{128} (8 a+1) P_{426}(a,c),
\]
from which we see that it is enough to show that $P_{426}(a,c)<0$. But for the current $a$ and $c$ values we have $\partial_c P_{426}(a,c)=192 \left(128 a^2-800 a+925\right)\ge 0$, so by (\ref{s4p215})
\[
 P_{426}(a,c)\le P_{426}\left(a, \frac{1}{512} \left(-1536 a^3+1728 a^2-424 a+25\right)\right)=
\]
\[
\frac{1}{8} (8 a+1) \left(24576 a^4-139264 a^3+212288 a^2-126000 a+17575\right)<0.
\]

\textbf{The interval $a\in\left[\varrho_{421},  \frac{22}{10}\right]$.} Here $\partial_d \lambda(77,a,c,d)\le 0$, so by (\ref{s4p215})
\[
\lambda(77,a,c,d)\le \lambda\left(77,a,c, \frac{1}{2} \left(-6 a^4+6 a^3-a^2-2 a c\right)\right)=12a (8 a+1) P_{427}(a,c),
\]
so, again, it suffices to show $P_{427}(a,c)<0$. But now we observe that $\partial_c P_{427}(a,c)=192 a^2-1824 a+2850$ is positive, if $\varrho_{421}\le a<\frac{1}{8} \left(38-\sqrt{494}\right)$, and non-positive, if $\frac{1}{8} \left(38-\sqrt{494}\right)\le a \le \frac{22}{10}$. By (\ref{s4p215}) again, in the first case, we have
\[
P_{427}(a,c)\le P_{427}\left(a,  \frac{1}{512} \left(-1536 a^3+1728 a^2-424 a+25\right)\right)=\]\[\frac{1}{256} (8 a+1) \left(6144 a^4-62976 a^3+174496 a^2-172672 a+35625\right)<0,
\]
and in the second case
\[
P_{427}(a,c)\le P_{427}\left(a,\frac{1}{192} \left(-768 a^3+768 a^2-160 a+7\right)\right)=\]\[\frac{19}{32} (8 a+1) \left(256 a^2-648 a+175\right)<0. 
\]
\end{proof}

\begin{rem} \Mma tells us that among the integers $4\le k \le 1000$, exactly members of the interval $77 \le k \le 98$ share the property that $\lambda(k,a,c,d)<0$ for any allowed $a$, $c$ and $d$ in the previous lemma.
\end{rem}

\begin{lem}
For each $\frac{1}{24} \left(9+2 \sqrt{6}\right)\le a\le \frac{22}{10}$, and any $c$ and $d$ satisfying (\ref{s4p216}), we have that $\lambda(54,a,c,d)<0$.
\end{lem}
\begin{proof}
The proof is completely analogous to the previous one. Now
$\partial_d \lambda(54,a,c,d)=-24 P_{428}(a)$ has two roots $\varrho_{422}\approx 0.808208$ and $\varrho_{423}\approx 1.97947$ in the given $a$-interval, so we again separate two cases.

\textbf{The interval $a\in\left[\frac{1}{24} \left(9+2 \sqrt{6}\right),  \varrho_{422}\right)\cup \left(\varrho_{423},\frac{22}{10}\right]$.} Now $\partial_d \lambda(54,a,c,d)>0$, so by (\ref{s4p216})
\[
\lambda(54,a,c,d)\le \lambda\left(54,a,c,  \frac{1}{2048}\left(768 a^3-512 a^2-512 a c+104 a+192 c-7\right)\right)=\]\[ \frac{159}{256} (8 a+1) P_{429}(a,c).
\]
We show that $P_{429}<0$. Indeed, $\partial_c P_{429}(a,c)=64 \left(192 a^2-832 a+663\right)$ is positive if $a\in\left[\frac{1}{24} \left(9+2 \sqrt{6}\right),  \varrho_{422}\right)$, so by (\ref{s4p216}) here
\[
P_{429}(a,c)\le P_{429}\left(a,  \frac{1}{128} \left(192 a^2-104 a+11\right)\right)=
\]
\[
\frac{1}{2} (8 a+1) \left(12288 a^4-58368 a^3+76160 a^2-35152 a+4199\right)<0.\]
On the other hand, $\partial_c P_{429}(a,c)< 0$, if $a\in \left(\varrho_{423},\frac{22}{10}\right]$, hence by (\ref{s4p216}) we have now
\[
P_{429}(a,c)\le P_{429}\left(a,  \frac{1}{512} \left(-1536 a^3+1728 a^2-424 a+25\right)\right)=
\]
\[
\frac{1}{8} (8 a+1) \left(12288 a^4-46080 a^3+53888 a^2-29328 a+4199\right)<0.
\]

\textbf{The interval $a\in\left[\varrho_{422},\varrho_{423}\right]$.} Here $\partial_d \lambda(54,a,c,d)\le 0$, so by (\ref{s4p216})
\[
\lambda(54,a,c,d)\le \lambda\left(54,a,c, \frac{1}{4096}\left(-384 a^2+224 a+512 c-25\right)\right)=\]\[ \frac{3}{512} (8 a+1) P_{4210}(a,c).
\]
By examining the sign of $\partial_c P_{4210}(a,c)=$ $-512 \left(256 a^3-5088 a^2+16536 a-11713\right)$, we prove  finally that $P_{4210}<0$ as well. We see that $\partial_c P_{4210}(a,c)>0$, if $\varrho_{422}\le a < \varrho_{424}\approx 1.00126$, but $\partial_c P_{4210}(a,c)\le 0$, if $\varrho_{424}\le a \le\varrho_{423}$. By taking into account  (\ref{s4p216}), in the first case we have
\[
P_{4210}(a,c)\le P_{4210}\left(a,  \frac{1}{128} \left(192 a^2-104 a+11\right)\right)=\] \[53 (8 a+1) \left(12288 a^4-58368 a^3+76160 a^2-35152 a+4199\right)<0,
\]
while in the second case
\[
P_{4210}(a,c)\le P_{4210}\left(a,  \frac{1}{512} \left(-1536 a^3+1728 a^2-424 a+25\right)\right)=\] \[ 8a(8 a+1) \left(6144 a^4-45312 a^3+102368 a^2-86284 a+17225\right)<0.
\]
\end{proof}

\begin{rem} According to \textit{Mathematica}, $k=54$ is the {\rm{only}} integer in the interval $[4,1000]\cap \mathbb{N}$ such that $\lambda(k,a,c,d)<0$ holds for {\rm{any}} admissible $a$, $c$ and $d$ triples in the previous lemma.
\end{rem}

\begin{lem}\label{lemmas4p27.5}
For each $\frac{22}{10}< a$, and any $c$ and $d$ satisfying (\ref{s4p215}), there exists an integer $k\ge 4$ such that $\lambda(k,a,c,d)<0$.
\end{lem}
\begin{proof}
Let us fix any admissible $a$, $c$ and $d$ throughout the proof. We first compute $\partial_d \lambda(k,a,c,d)=2 P_{4211}(k,a)$, where
\begin{equation}\label{P4211}
P_{4211}(k,a)=k^3-(96 a+6) k^2+\left(2304 a^2+288 a+11\right) k-12288 a^3-2304 a^2-192 a-6, 
\end{equation}
and show that this cubic polynomial (in $k$) has three distinct real roots and also give some bounds on the roots in terms of $a$. 

\begin{figure}[h]
 \centering
 \includegraphics[width=3in]{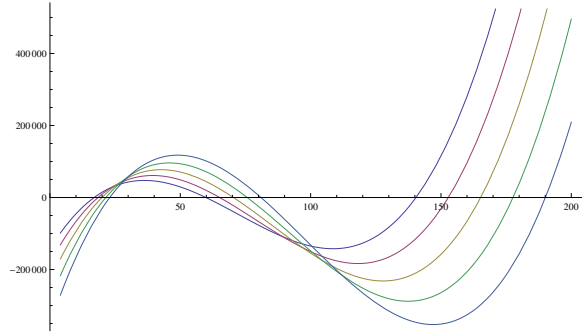}
 \caption{Graphs of the functions $k\mapsto P_{4211}(k,a)$ for some values of the parameter $a>\frac{22}{10}$ in the $\widehat{\Pi}_{4/4,2}$ class. Each curve has three real roots, denoted by $\varrho_{426}(a)_1<\varrho_{426}(a)_2<\varrho_{426}(a)_3$, and two local extrema, denoted by $\varrho^{-}_{425}(a)<\varrho^{+}_{425}(a)$.}
\end{figure}

It is seen that the function $k\mapsto \partial_k P_{4211}(k,a)$ has two distinct real roots at 
$k=\varrho^{\pm}_{425}(a):=\frac{1}{3} \left(96 a+6\pm\sqrt{3} \sqrt{768 a^2+96 a+1}\right)$, being the abscissae of the two strict local extrema of $k\mapsto P_{4211}(k,a)$. It is also easily established that
$P_{4211}(\varrho^-_{425}(a),a)>0$ and $P_{4211}(\varrho^+_{425}(a),a)<0$. By observing that $P_{4211}(4,a)<0$ and $4<\varrho^-_{425}(a)$, further, by taking into account that the leading coefficient of $k\mapsto P_{4211}(k,a)$ is positive, we get that each of the three disjoint intervals
\begin{equation}\label{preciselyone}
\Big(4,\varrho^-_{425}(a)\Big)\cup \Big(\varrho^-_{425}(a),\varrho^+_{425}(a)\Big)\cup \Big(\varrho^+_{425}(a),+\infty\Big)
\end{equation}
contains precisely one root of $k\mapsto P_{4211}(k,a)=\frac{1}{2}\partial_d \lambda(k,a,c,d)$.

By denoting these roots by $\varrho_{426}(a)_1<\varrho_{426}(a)_2<\varrho_{426}(a)_3$, we also see from the properties of the cubic polynomial that if $k\in \big(\varrho_{426}(a)_2,\varrho_{426}(a)_3\big)$, then $\partial_d \lambda(k,a,c,d)<0$. (Note that for any such $k$, $k\ge 4$ holds.) Also referring to (\ref{s4p215}), this means that if $k\in \big(\varrho_{426}(a)_2,\varrho_{426}(a)_3\big)$, then
\[
\lambda(k,a,c,d)\le \lambda\left(k,a,c,  \frac{1}{2} \left(-6 a^4+6 a^3-a^2-2 a c\right)\right)=
6 a (8 a+1)P_{4212}(k,a,c), 
\]
where
\[
P_{4212}(k,a,c)=k^2 \left(4 a^3-5 a^2+a+c\right)+k \left(-192 a^4+212 a^3-27 a^2-48 a c-2 a-3 c\right)+\]\[1536 a^5-1344 a^4+104 a^3+384 a^2 c+18 a^2+48 a c+a+2 c.
\]
Clearly, in order to finish the proof, it is sufficient to show that there is an \textit{integer} $k\in \big(\varrho_{426}(a)_2,\varrho_{426}(a)_3\big)$ such that $P_{4212}(k,a,c)<0$.

Let the two roots of  $k\mapsto \partial_c P_{4212}(k,a,c)=k^2-(48 a+3) k+384 a^2+48 a+2$ be denoted by
$\varrho^{\pm}_{427}(a):=\frac{1}{2} \left(48 a+3\pm\sqrt{768 a^2+96 a+1}\right)$, then $\partial_c P_{4212}(k,a,c)<0$ if $k\in \big(\varrho^-_{427}(a),\varrho^+_{427}(a)\big)$. We 
will show next that 
\begin{equation}\label{intersection}
\big(\varrho_{426}(a)_2,\varrho_{426}(a)_3\big)\cap\big(\varrho^-_{427}(a),\varrho^+_{427}(a)\big)=\big(\varrho_{426}(a)_2,\varrho^+_{427}(a)\big).
\end{equation}
In fact, we have seen in (\ref{preciselyone}) that $\varrho^-_{425}(a)<\varrho_{426}(a)_2$, and an elementary calculation shows that $\varrho^-_{427}(a)<\varrho^-_{425}(a)$. An analogous argument tells us that 
$\varrho^+_{427}(a)<\varrho^+_{425}(a)<\varrho_{426}(a)_3$, verifying (\ref{intersection}).

Thus, by picking any $k\in\big(\varrho_{426}(a)_2,\varrho^+_{427}(a)\big)$, we can conclude (with the help of (\ref{s4p215}) also) that
\[
P_{4212}(k,a,c)\le P_{4212}\left(k,a, \frac{1}{192} \left(-768 a^3+768 a^2-160 a+7\right)\right)=
\]
\[
\frac{1}{192} (8 a+1) (k-1) \left((7-24 a) k+768 a^2-96 a-14\right),
\]
from which we see that the proof is finished if we find any $k\in\big(\varrho_{426}(a)_2,\varrho^+_{427}(a)\big)\cap\mathbb{N}$ such that $(7-24 a) k+768 a^2-96 a-14<0$, or, in other words, $k>\frac{768 a^2-96 a-14}{24 a-7}$.

We aim to show now that $\varrho_{426}(a)_2<\frac{768 a^2-96 a-14}{24 a-7}$. (Notice that this time both $\varrho_{426}(a)_2<\varrho^+_{425}(a)$ and $\frac{768 a^2-96 a-14}{24 a-7}<\varrho^+_{425}(a)$ are true, so we have to take an extra step.) But by referring to (\ref{preciselyone}) we have
$\varrho_{426}(a)_1<\varrho^-_{425}(a)$, and  an elementary computation yields that $\varrho^-_{425}(a)<\frac{768 a^2-96 a-14}{24 a-7}$, further, it is not hard to see  that
\[
P_{4211}\left( \frac{768 a^2-96 a-14}{24 a-7} ,a\right)=-\frac{48 a (8 a+1) (16 a-3) \left(9216 a^3-1512 a+49\right)}{(24 a-7)^3}<0.
\]
However, we already know that $P_{4211}(\cdot,a)<0$ if and only if $k\in \big(-\infty,\varrho_{426}(a)_1\big)\cup \big(\varrho_{426}(a)_2,\varrho_{426}(a)_3\big)$, so these force
 $\varrho_{426}(a)_2<\frac{768 a^2-96 a-14}{24 a-7}$. 

Summarizing, it is enough to prove as a last step that $\varnothing\ne\big(\frac{768 a^2-96 a-14}{24 a-7},\varrho^+_{427}(a)\big)\cap\mathbb{N}$, being true, since  
$\varrho^+_{427}(a)-\frac{768 a^2-96 a-14}{24 a-7}>9$, if $a> \frac{22}{10}$.
\end{proof}

\begin{rem} Similarly to Remark \ref{length9}, the bound 9 could not be replaced by, say, 10 on $a\in(22/10,+\infty)$. 
\end{rem}

\begin{rem} The proof we have just presented for Lemma \ref{lemmas4p27.5} (as well as the one we will present for the next lemma) would break down on the larger interval $a\ge \frac{9+2 \sqrt{6}}{24}$. This explains why we have chosen the finer subdivision $a\in\left[\frac{1}{24} \left(9+2 \sqrt{6}\right),  \frac{22}{10}\right]\cup \left(\frac{22}{10},+\infty\right)$.
\end{rem}

\begin{lem}
For each $\frac{22}{10}< a$, and any $c$ and $d$ satisfying (\ref{s4p216}), there is an integer $k\ge 4$ such that $\lambda(k,a,c,d)<0$.
\end{lem}
\begin{proof}
We have seen in the preceding proof that $\partial_d \lambda(k,a,c,d)\ge 0$ if 
$k\in \big[\varrho_{426}(a)_1,\varrho_{426}(a)_2\big]\subset (4,+\infty)$, so---by (\ref{s4p216})---for these $k$ values we have 
$\lambda(k,a,c,d)\le$
\[
 \lambda\left(k,a,c,  \frac{1}{2048}\left(768 a^3-512 a^2-512 a c+104 a+192 c-7\right)\right)=\frac{(8 a+1) (k-1)}{1024} P_{4213}(k,a,c),
\]
where \[P_{4213}(k,a,c)=k^2 \left(768 a^3-768 a^2+160 a+192 c-7\right)+\]
\[
k \left(-49152 a^4+39168 a^3-5376 a^2-12288 a c-128 a-960c+35\right)+
\]\[
589824 a^5-294912 a^4+29184 a^3+147456 a^2 c+3840 a^2+24576 a c-384 a+1152 c-42.
\]
To show that $P_{4213}(k,a,c)<0$ for a suitable $k\in \big[\varrho_{426}(a)_1,\varrho_{426}(a)_2\big]\cap \mathbb{N}$, we first compute 
\[
\partial_c P_{4213}(k,a,c)=192 \left(k^2-(64 a+5) k+768 a^2+128 a+6\right), 
\]
and denote the roots of $k\mapsto \partial_c P_{4213}(k,a,c)$ by
$\varrho^{\pm}_{428}(a):=\frac{1}{2} \left(64 a+5\pm\sqrt{1024 a^2+128 a+1}\right)$.
Note that $\partial_c P_{4213}(k,a,c)<0$ if $\varrho_{428}^{-}(a)<k<\varrho_{428}^{+}(a)$.
Now by using formulae (\ref{preciselyone}) and (\ref{P4211}), one easily shows that $\varrho^{-}_{425}(a)<\varrho^{-}_{428}(a)<\varrho^{+}_{425}(a)$ and 
\[
P_{4211}\left( \varrho^{-}_{428}(a) ,a\right)=32 a (8 a+1) \left(\sqrt{1024 a^2+128 a+1}-16 a-1\right)>0, 
\]
 implying $\varrho_{426}(a)_1<\varrho^{-}_{428}(a)<\varrho_{426}(a)_2$. Hence if $k\in \big(\varrho_{428}^{-}(a),\varrho_{426}(a)_2\big]\cap \mathbb{N}$, then, also by (\ref{s4p216}), 
\[
P_{4213}(k,a,c)\le P_{4213}\left(k,a,\frac{1}{512} \left(-1536 a^3+1728 a^2-424 a+25\right)\right)=
\frac{1}{8} (8 a+1)P_{4214}(k,a),
\]
with
\[
P_{4214}(k,a)=k^2\left(192 a^2-144 a+19\right) +k\left(-12288 a^3+2112 a^2+1296 a-95\right) +\]\[147456 a^4+110592 a^3-27264 a^2-2016 a+114.
\]
Now the leading coefficient, $192 a^2-144 a+19$ of $k\mapsto P_{4214}(k,a)$ is positive, and its discriminant, $\Delta_{42}(a):=$
\[
37748736 a^6-51904512 a^5+46043136 a^4-14751744 a^3+2101632 a^2-27360 a+361
\]
is also positive. (This last statement can easily be shown by noticing that $\Delta^{\prime\prime\prime\prime}_{42}(a)>0$ for all $a\in \mathbb{R}$, and $\Delta^{\prime\prime\prime}_{42}(22/10)>0$, so $\Delta^{\prime\prime\prime}_{42}(a)>0$ for $a>22/10$. Repeating this recursively and similarly for the lower order derivatives we get that $\Delta_{42}(a)>0$ for all $a>22/10$.) This means that we can denote by $\varrho^{\pm}_{429}(a)$ the two real roots of the quadratic polynomial $k\mapsto P_{4214}(k,a)$. Clearly, $\varrho^{-}_{429}(a)<k<\varrho^{+}_{429}(a)$ implies $P_{4214}(k,a)<0$.

In light of the above, in order to finish the proof it is enough to show that
\[
\varnothing\ne \Big( \varrho^{-}_{429}(a),\varrho^{+}_{429}(a)\Big)\cap \Big(\varrho_{428}^{-}(a),\varrho_{426}(a)_2\Big]\cap \mathbb{N}.
\]
More specifically, we can prove that the following sufficient condition 
\begin{equation}\label{16a+24}
\varrho^{-}_{428}(a)<\varrho^{-}_{429}(a)<16a+19<16a+24<\varrho_{426}(a)_2<\varrho^{+}_{429}(a)
\end{equation}
holds: the first inequality is true because of \[P_{4214}\left(\varrho^{-}_{428}(a),a\right)=128 a \left((24 a-7) \sqrt{1024 a^2+128 a+1}+136 a+7\right)>0\] and 
\[
\partial_k P_{4214}\left(\varrho^{-}_{428}(a),a\right)=-6144 a^2+1792 a-\left(192 a^2-144 a+19\right) \sqrt{1024 a^2+128 a+1}<0;
\] 
the second inequality is a consequence of \[P_{4214}(16a+19,a)=-16 \left(576 a^3-1264 a^2+1209 a-323\right)<0;\] the fourth inequality follows from $P_{4211}(16a+24,a)=4096 a^3+768 a^2-25360 a+10626>0$ and $16a+24<\varrho^{+}_{425}(a)$; while the fifth inequality is derived from the fact that
\[\varrho^{-}_{425}(a)<\varrho^{+}_{429}(a)=\frac{12288 a^3-2112 a^2-1296 a+95+\sqrt{\Delta_{42}(a) }}{2 \left(192 a^2-144 a+19\right)}\]
and
\[
P_{4211}\left(\varrho^{+}_{429}(a),a\right)=-\frac{96 a}{{\left(192 a^2-144 a+19\right)^3}}\times \]\[
\Big(301989888 a^8-622854144 a^7+474218496 a^6-1089110016 a^5+816156672 a^4-250172928 a^3+\]\[ 31469760 a^2-960184 a+9025+\]\[\left(98304 a^5-135168 a^4-132096 a^3+106880 a^2-19464a+475\right) \sqrt{\Delta_{42}(a) }\Big)<0.
\]
\end{proof}

\begin{rem} In the proof of this lemma---in order to establish $\lambda(k,a,c,d)<0$---one would have 4 initial possibilities to choose from according to the signs of the appropriate $\partial_d$ and $\partial_c$ partial derivatives. It turns out that if one chooses certain 2 out of these 4 possibilities, then the proof could {\rm{not}} be completed. On the other hand, the remaining two feasible directions (out of which we have presented one) require approximately the same amount and type of computations.
\end{rem}

\begin{rem} The motivation for (\ref{16a+24}) came from the fact that
\[
\lim_{a\to +\infty}\frac{\varrho_{429}^{-}(a)}{a}=16<\lim_{a\to +\infty}\frac{\varrho_{426}(a)_2}{a}=\mathrm{root}_2(1,-96,2304,-12288)\approx 26.4432.
\]
In fact, for $a>22/10$, we have $\varrho_{426}(a)_2-\varrho_{429}^{-}(a)>6$, but $\ldots>7$ does not hold.
\end{rem}

\section{Determination of $\widehat{R}_{4/4,p}$ for $3\le p\le 5$\label{sectionSDIRKs4p543}}

For the sake of brevity and due to the fact that the techniques here are very similar to the corresponding earlier ones, we provide only brief (indication of the) proofs of the remaining results.

\subsection{Determination of $\widehat{R}_{4/4,5}$}\label{section8.1}

We prove here that $\widehat{R}_{4/4,5}\approx 3.7432$. The exposition is completely analogous to the one in Section \ref{sectionSDIRKs3p4}. There are exactly 4 rational functions in this $\widehat{\Pi}_{4/4,5}$ class:
\[
\psi(z)=\frac{a_4 z^4+a_3 z^3+a_2 z^2+a_1 z+1}{(1-a z)^4},
\]
where 
\[a_4=\text{root}_{m_4}(2160000,4032000,-93600,-120,1), \quad a_3=\text{root}_{m_3}(18000,24000,-1800,0,1),\]
\[a_2=\text{root}_{m_2}(1200,-7200,120,200,3), \quad a_1=\text{root}_{m_1}(15,60,-30,-20,7),\]
\[a=\text{root}_{m_0}(120,-240,120,-20,1),\]
with 
\[
(m_4,m_3,m_2,m_1,m_0)\in\{ (1,1,4,1,4), (2,2,2,3,2), (3,3,3,4,1), (4,4,1,2,3) \}.
\]
First we remark that in all 4 cases the only pole of $\psi$ is positive real.

If $m_4=1$, then the corresponding $\psi$ satisfies $\psi^{(5)}(-1/2)<0<\psi^{(5)}(0)=1$, hence $\psi^{(5)}$ has a root in $(-1/2,0)$, so by Theorem \ref{vdgkLemma4.5} with $\ell=5$ it has $R(\psi)< 1/2$. 

If $m_4=2$, then the corresponding $\psi$ satisfies $\psi^{(8)}(0)<0<\psi^{(8)}(-1)$, hence by Theorem \ref{vdgkLemma4.5} with $\ell=8$ it has $R(\psi)< 1$.

If $m_4=4$, then the corresponding $\psi$ satisfies $\psi^{(7)}(-1/2)<0<\psi^{(7)}(0)$, hence by Theorem \ref{vdgkLemma4.5} with $\ell=7$ it has $R(\psi)< 1/2$.

If $m_4=3$, then the corresponding $\psi$ satisfies $\psi^{\prime}(x^*)=0$ with
\[
x^*:=\text{root}_1(1,12,48,-32,-864,-2016,2784,13248,-9072,-35136,44928,-20736,3456)\approx
\]
\[
-3.743299247417768882803493,
\]
hence by Theorem \ref{vdgkLemma4.5} with $\ell=1$, it has $R(\psi)\le |x^*|$. By using partial fraction decomposition, we can show that $R(\psi)=|x^*|$, therefore $\widehat{R}_{4/4,5}=|x^*|$.

\subsection{Determination of $\widehat{R}_{4/4,4}$}\label{section8.2}

We prove that $\widehat{R}_{4/4,4}\approx 5.1672$. We refer to Section \ref{sectionSDIRKs3p3} for an analogous treatment and for further details on the theorems used. Elements of this class can be written as 
\[
\psi_a(z)=(1-a
   z)^{-4}\left(\left(a^4-4 a^3+3 a^2-\frac{2 a}{3}+\frac{1}{24}\right) z^4+\left(-4 a^3+6 a^2-2 a+\frac{1}{6}\right) z^3+\right.\]
\[\left.\left(6 a^2-4 a+\frac{1}{2}\right) z^2+(1-4 a) z+1\right)
\]
with a real parameter $a$. Here also we can assume that $a>0$ and $\psi_a$ does not have removable singularities (indeed, $a=0$ implies $R=1$; $a<0$ without removable singularities implies $R=0$; while if the numerator and denominator have a common root, then $\psi_a$ is one of the 3 rational functions in Section \ref{sectionSDIRKs3p4}, hence $R\le \widehat{R}_{3/3,4}<5$). By applying the same techniques as earlier in the one-parameter families, we can prove that
\[\widehat{R}_{4/4,4}=\text{root}_3(1,-24,240,-1168,1848,7008,-30528,7488,71568,36864)\approx\]
\[ 5.167265421277419938673374,\]
with the unique optimal rational function corresponding to the parameter value 
\[
a=a^*:=\text{root}_1(147456,-546624,799488,-601344,258432,-66576,10400,-960,48,-1)\approx
\]
\[
0.09713312764144710280835106.
\]

\subsection{Determination of $\widehat{R}_{4/4,3}$}\label{section8.3}

This class can be described by 2 real parameters. We can prove along the lines of the uniqueness-type proof in 
Section \ref{sectionSDIRKs3p2} that $\widehat{R}_{4/4,3}=3+\sqrt{15}\approx 6.8729$.

\section{Appendix}\label{appendixsection}

In this section we list some auxiliary algebraic numbers, polynomials or functions that have appeared in certain more involved proofs.

\subsection{Algebraic expressions in the proof of $R_{4/4,7}$}\label{appendixsection9.1}

\indent   $\vr_{471}=\text{root}_1(252105,936390,1441629,1175608,534576,128352,12704)\approx -0.843194$,

\indent $\vr_{472}=\text{root}_2(252105,936390,1441629,1175608,534576,128352,12704)\approx -0.471357$,

\indent $\vr_{473}=\text{root}_2(1,-30,390,-2760,11160,-25200,25200)\approx 7.64527$,

\indent $\vr_{474}=\text{root}_1(1,-15,90,-210)\approx 5.64849$,

\indent $P_{471}(a)=4593387934777490821322994 a^{10}-23474176816503998442760098 a^9+$\\
\indent \indent \indent \ \ \ \ $54351276637181597088697031 a^8-75708182541462463946985360 a^7+$\\
\indent \indent \indent \ \ \ \ $71791544313152743211806464
   a^6-51339236908163135191695540 a^5+$\\
\indent \indent \indent \ \ \ \ $32383499796235996766706978 a^4-22447171953065668650772896 a^3-$\\
\indent \indent \indent \ \ \ \ $31555722050640962925153690 a^2-13278623731050880370830074
   a-$\\
\indent \indent \indent \ \ \ \ $1667865061502294482628289$,

\indent $\vr_{475}=\text{root}_1(2914539265575,14876524399779,33497711997246,43683499824678,$\\\indent \indent \indent \indent\ \  \ 
$36368607954483,20052618149655,7324770907832,1709817444048,$
\\\indent \indent \indent \indent\ \  \ $231520801344,13859993824)\approx -0.850052$,

\indent $\vr_{476}=\text{root}_2(2,0,-570,-6360,-28755,-58950,-44775)\approx 21.5907$,

\indent $\vr_{477}(a)=\text{root}_1(343 a^3+441 a^2+189 a+27,-6174 a^3-7644 a^2-3150 a-432,$
\\\indent \indent \indent \indent \indent\ \   $47334 a^3+55860 a^2+21882 a+2844,-197568 a^3-219912 a^2-80976 a-$\\\indent \indent \indent \indent \indent\ \   
$9856,478485 a^3+500535 a^2+173460
   a+19950,-648270 a^3-643860 a^2-$\\\indent \indent \indent \indent \indent\ \   
$215460 a-24480,385875 a^3+368235 a^2+120960 a+13440)$,

\indent $\vr_{478}=\text{root}_1(1525366344600,8651006268660,22401405617094,35008454688795,$\\
\indent \indent \indent \indent  \ \ \  $36765395913141,27330042736584,14744440158000,5816571277965,$\\\indent \indent \indent \indent  \ \ \ 
$1665232178673,337417309264,45932541984,3771903744,14
   1310592)\approx $\\\indent \indent \indent \indent  \ \ \ $-0.469514$,

\indent $\vr_{479}=\text{root}_3(1715,2793,1428,232)\approx -0.358565$,

\indent $\vr_{4710}=\text{root}_1(3675,3507,1152,128)\approx -0.274796$,

\indent $\vr_{4711}(a)=\text{root}_1(7 a+3,-84 a-32,420 a+120,-840 a,-840)$,

\indent $\vr_{4712}(a)=\text{root}_2(7 a+3,-84 a-32,420 a+120,-840 a,-840)$,\\

\indent $a_{47}^*=\text{root}_1(3690134492416632557532940385381064525,$\\
$61029736485482612404619408048530290750,486403976093629932677089384897095046935$\\
$2487196669730164018264503869937911017740,9165142322735094610456262590295894934837$\\
$25915196147894075773838728122892827040386,58458916104726476816514027882738196455735,$\\
$107979092136285121513038059062379021459880,166321122780685405075074427832862986958975,$\\
$216449429305753337078085597912060934057890,240240731812681651181007922270907964916605,$\\
$228911785323865533798715657213902093130380,188040840503767270695413477386184055716415,$\\
$133446371333215106255009916213506161760670,81810565480540039001637010346621868861045,$\\
$43219005242542299485983226045186624870400,19562352897200734179995371112351361862080,$\\
$7506336744004524872039975974458313909440,2394351571156175328455192580869480281280,$\\
$610166214096532168321831058806216604160,112156289536490108521784144024335180800,$\\
$8974819895289125868510161871145369600,-2895655868721204798705192374322355200,$\\
$-1556648822533171464739057080454594560,-420567479322606922282373395209093120,$\\
$-79380339166868822139598939104215040,-11057440115321274264887170064056320,$\\
$-1128821021046790224548013859143680,-80578528495036917551879568752640,$\\
$-3612945212567877804734004854784,-76850622516036469593789693952)\approx -0.439849$,\\

\indent $B\left(a_{47}^*\right)=\text{root}_2(4191472,370695456,15701398968,423572490288,8166117227943,$\\
$119697694352106,1385540391042992,12982888147808790,100093600309232610,641253042735937920,$\\
$3429070908298155495,15292307228951973150,56488604555080263600,170258230386661619700,$\\
$405691319555093173950,698620766882164002000,475967180133964116000,$\\
$-2768737485607368840000,-19171364099094461844000,-83177899383693915360000,$\\
$-273285810570616506720000,-658873655023431060000000,-1076077292526138186000000,$\\
$-1039429806316804224000000,-402385174364630880000000,1632495959008528320000000,$\\
$15171036284831515200000000,63475278812225280000000000,141035850392839718400000000,$\\
$165407760061818624000000000,81912466393111872000000000)\approx 2.743911.$

\subsection{Algebraic expressions in the proof of $\widehat{R}_{4/4,2}$}\label{appendixsection9.2}

\indent  $P_{421}(a,c)=1536 a^5-1280 a^4-512 a^3 c+436 a^3+576 a^2 c-54 a^2-72 a c+2 a+c $, \\
\indent $P_{422}(a,c)=2304 a^4-1728 a^3-1024 a^2 c+480 a^2+768 a c-48 a-48 c+1 $, \\
\indent $P_{423}(a,c)=768 a^3-512 a^2-512 a c+104 a+192 c-7 $, \\
\indent $P_{424}(a,c)= -384 a^2+224 a+512 c-25$, \\
\indent $P_{425}(a)=512 a^3-7296 a^2+22800 a-17575$, \\
\indent $\varrho_{421}=\text{root}_1(512,-7296,22800,-17575)\approx 1.17854,$ \\
\indent $P_{426}(a,c)=98304 a^5-679936 a^4+1266432 a^3+24576 a^2 c-827264 a^2-153600 a c+156400 a+$\\\indent \indent \indent \indent \ \  $177600 c-6475$, \\
\indent $P_{427}(a,c)=768 a^5-8064 a^4+20072 a^3+192 a^2 c-15853 a^2-1824 a c+2888 a+2850 c$,\\
\indent $P_{428}(a)=1024 a^3-10176 a^2+22048 a-11713$,\\
\indent $\varrho_{422}=\text{root}_1(1024,-10176,22048,-11713)\approx 0.808208,$ \\
\indent $\varrho_{423}=\text{root}_2(1024,-10176,22048,-11713)\approx 1.97947,$ \\
\indent $P_{429}(a,c)=49152 a^5-245760 a^4+365312 a^3+12288 a^2 c-210496 a^2-53248 a c+38272 a+$\\\indent \indent \indent \indent \ \  \,$42432 c-1547$,\\
\indent $P_{4210}(a,c)=5406720 a^5-28110848 a^4-131072 a^3 c+44025856 a^3+2605056 a^2 c-26966400 a^2-$ \\\indent \indent \indent \indent \ \  \  \,$8466432 a c+5517512 a+5997056 c-292825$,\\
\indent $\varrho_{424}=\text{root}_1(256,-5088,16536,-11713)\approx 1.00126.$

\bibliographystyle{plain}
\bibliography{abs_mon_SDIRK}

\affiliationone{
   Lajos L\'oczi and David I. Ketcheson\\
   4700 KAUST \\
   Thuwal, 23955\\
   Saudi Arabia
   \email{lajos.loczi@kaust.edu.sa\\
   david.ketcheson@kaust.edu.sa}}

\end{document}